\documentclass[11pt, letterpaper]{amsart}

\usepackage[utf8]{inputenc}
\usepackage[english]{babel}
\usepackage{amsmath,amssymb,amsthm,mathrsfs}
\usepackage{esint}
\usepackage[colorlinks,linktocpage]{hyperref}
\usepackage[all]{xy}
\usepackage{import}
\usepackage{pdfpages}
\usepackage{transparent}
\usepackage{xcolor}
\usepackage{comment}
\usepackage{upgreek}

\newtheorem{thm}{Theorem}[section]
\newtheorem{lem}[thm]{Lemma}

\newtheorem{prop}[thm]{Proposition}
\newtheorem{conj}[thm]{Conjecture}

\theoremstyle{definition}
\newtheorem{defn}{Definition}[section]

\newtheorem{exmp}{Example}[section]

\theoremstyle{remark}
\newtheorem{rmk}{Remark}[section]
\newtheorem*{rmk*}{Remark}

\newtheorem*{fact*}{Fact}

\def\WL{\overset{\ast}{\rightharpoonup}}

\title[The $\sigma $-IMCF and the generalized Penrose conjecture]{The $\sigma $-inverse mean curvature flow and the generalized Penrose conjecture}
\author{Conghan Dong}
\address{Department of Mathematics, Duke University, 120 Science Dr, Durham, NC 27710, USA}
\email{conghan.dong@duke.edu}
\begin{document}
\date{\today}

\maketitle

\begin{abstract}
	Let $(M^3, g, \mathbf{k})$ be a complete asymptotically flat initial data set satisfying the dominant energy condition, and let $m$ denote its ADM mass. The generalized Penrose conjecture asserts that the area of an outermost generalized apparent horizon $N\subset M$ satisfies $|N| \leq 16 \pi m^2$. In this paper, we establish this inequality for each connected component of $N$ in the case where $\mathbf{k}$ is proportional to the metric $g$. Our approach is based on a new geometric evolution, which we call the $\sigma $-inverse mean curvature flow, together with a novel monotonicity formula that may be of independent interest.
\end{abstract}

\section{Introduction}

An initial data set for a spacetime consists of a triple $(M^3, g, \mathbf{k})$, where $M$ is a three-dimensional manifold, $g$ is a Riemannian metric, and $\mathbf{k}$ is a symmetric $(0,2)$-tensor. The associated energy density function $\mu $ and momentum density covector $J$ are defined by
\begin{align*}
	(8 \pi ) \mu &= \frac{1}{2}\left( R + (\mathrm{tr}_g \mathbf{k})^2 - \|\mathbf{k}\|^2_g \right) ,\\
	(8 \pi ) J &= \mathrm{div}_g \left( \mathbf{k} - (\mathrm{tr}_g \mathbf{k} ) g \right) ,
\end{align*} 
where $R$ denotes the scalar curvature of $g$. The initial data set is said to satisfy the dominant energy condition if
\begin{align*}
	\mu \geq |J|_g.
\end{align*} 

Following \cite{SchoenYau81}, an initial data set $(M, g, \mathbf{k})$ is called asymptotically flat if, outside a compact set, $M$ decomposes into finitely many ends $M_1, \ldots, M_{p}$, each diffeomorphic to the complement of a compact subset of $\mathbb{R}^{3}$, and under such diffeomorphisms, the tensors $g$ and $\mathbf{k}$ satisfy the decay conditions
\begin{align}\label{AF-condition-0}
	\begin{split}
		&|g_{ij} - \delta _{ij}| + |x| \cdot |\partial g_{ij}| + |x|^2 \cdot |\partial ^2 g_{ij}| \leq C |x| ^{-1},\\
	& |R| + |x| \cdot |\partial R| \leq C |x|^{-4},\\
	& |\mathbf{k}_{ij}| + |x|\cdot  |\partial \mathbf{k}_{ij}|  + |x|\cdot  \left| \textstyle \sum_{i} \mathbf{k}_{ii}\right|  \leq C |x|^{-2},
	\end{split}
\end{align} 
as $|x|\to \infty$, where norms and derivatives are taken with respect to the Euclidean metric. 

With each end $M_k$ we associate the ADM mass $m_k$ defined by the flux integral
\begin{align*}
	m_k = \frac{1}{16 \pi } \lim_{r\to \infty}   \int_{S_r} \sum_{i,j}^{}\left( g_{ij,j} - g_{jj,i} \right) n ^{i} dA,
\end{align*} 
where $S_r= \{x: |x|=r\} $ is the Euclidean $2$-sphere at infinity. 

For convenience, we modify the topology of $M$ by compactifying all of the ends of $M$ except for one chosen end. We denote by $m= m(g)$ the ADM mass of the chosen end.

Following \cite{BrayKhuri11}, let $\mathcal{S}$ denote the collection of surfaces which are $C^2$-smooth compact boundaries of open sets $U$ in $M$, where $U$ is bounded in the chosen end and contains all other ends. A $C^{2}$-smooth surface $\Sigma \in \mathcal{S}$ is called a generalized apparent horizon if its mean curvature satisfies
\begin{align*}
	H_{\Sigma } = |\mathrm{tr}_{\Sigma }\mathbf{k}|,
\end{align*} 
and a generalized trapped surface if
\begin{align*}
	H_{\Sigma } \leq |\mathrm{tr}_{\Sigma }\mathbf{k}|.
\end{align*} 
It was proved by Eichmair \cite{Eichmair10} that any generalized trapped surface $\Sigma \in \mathcal{S}$ is enclosed by a unique outermost generalized apparent horizon $\tilde{\Sigma }$, which is $C^{2,\alpha }$-smooth and strictly outer area minimizing.  

The generalized Penrose conjecture was formulated by Bray-Khuri as follows.
\begin{conj}[{\cite[Conjecture 4]{BrayKhuri11}}]
	Suppose that $(M^3, g, \mathbf{k})$ is a complete connected asymptotically flat initial data set satisfying the dominant energy condition. For a chosen end with mass $m$, if $\Sigma \in \mathcal{S}$ is a generalized trapped surface, then
	\begin{align}
		m \geq \sqrt{\frac{A}{16 \pi }} ,
	\end{align}
	where $A$ is the area of the outermost generalized apparent horizon $\tilde{\Sigma }$ enclosing $\Sigma $. 
\end{conj}

\begin{rmk}
	If one replaces generalized apparent horizons by apparent horizons, that is, surfaces satisfying $H_{\Sigma }= \mathrm{tr}_{\Sigma }\mathbf{k}$ or $H_{\Sigma }= - \mathrm{tr}_{\Sigma } \mathbf{k}$, one obtains the standard Penrose conjecture, which remains widely open. See \cite{Mars09} for an excellent survey of this conjecture and the references therein.  

An important consequence of the Penrose conjecture is the (generalized) positive mass theorem for initial data sets, which asserts the weaker conclusion $m \geq 0$. This result was established by Schoen-Yau \cite{SchoenYau81} and independently by Witten \cite{Witten81}.

\end{rmk}

When $\mathbf{k}=0$, the generalized Penrose conjecture reduces to the Riemannian Penrose inequality, which was proved by Huisken-Ilmanen \cite{HI01} in the case of a connected horizon and by Bray \cite{Bray01} in full generality. Beyond the time-symmetric setting, several special cases are also known. In particular, the spherically symmetric case has been established, see \cite{Hayward96, IMM96, BrayKhuri10Jang}. In contrast, in the fully general setting, a counterexample was constructed by Carrasco-Mars \cite{CarrascoMars10}.

In this paper, we investigate a natural intermediate class of initial data sets. Recall that any symmetric $2$-tensor admits the decomposition 
\begin{align*}
	\mathbf{k} = \frac{\tau }{3} g + \hat{\mathbf{k}},
\end{align*} 
where $\tau = \mathrm{tr}_{g} \mathbf{k}$ is a smooth function and $\hat{\mathbf{k}}$ is trace free. The time-symmetric case ($\mathbf{k}\equiv 0$) corresponds to the vanishing of both components, whereas the general case allows both the trace and trace-free parts to be nontrivial. Here we focus on the intermediate regime in which the trace-free part vanishes, or equivalently, 
\begin{align}\label{k-g-relation}
	\mathbf{k} = \frac{\tau }{3} g,
\end{align}
so that $\mathbf{k}$ is pointwise proportional to the metric $g$.

Although the condition (\ref{k-g-relation}) imposes a strong algebraic restriction on the initial data set, it still allows substantial geometric flexibility and includes a large class of nontrivial examples.

In this setting, the energy density and the momentum density become
\begin{align*}
	(8 \pi) \mu &= \frac{1}{2}\left( R + \tau ^2 - \frac{1}{3} \tau ^2 \right) = \frac{1}{2}R + \frac{1}{3} \tau ^2,\\
	(8 \pi) J &=  \mathrm{div}\left( \frac{\tau }{3} g - \tau g\right)  = -\frac{2}{3} d \tau,
\end{align*} 
and the dominant energy condition becomes 
\begin{align}\label{DEC-tau}
	|\nabla \tau |_g \leq \frac{3}{4} R + \frac{1}{2} \tau ^2.
\end{align} 

Define
\begin{align*}
	h= \mathrm{tr}_{\Sigma } \mathbf{k} = \frac{2}{3}\tau,
\end{align*}
which is independent of $\Sigma $.
Thus the generalized apparent horizons correspond to surfaces of prescribed mean curvature $|h|$. 
The dominant energy condition (\ref{DEC-tau}) is equivalent to
\begin{align}\label{DEC-h}
	R + \frac{3}{2} h^2 - 2 |\nabla h | \geq 0.
\end{align} 
\begin{rmk}
	This condition coincides with the positivity condition arising in the second variation of $\mu $-bubbles with weight $|h|$. As an application, we can provide an alternative proof of the generalized positive mass theorem; see Section \ref{sect-DEC} for further details. This observation also motivates the flow introduced below, which can be considered as a parabolic version of $\mu $-bubbles.
\end{rmk}

We call $(M, g, h)$ an asymptotically flat triple if (\ref{AF-condition-0}) is satisfied with $\mathbf{k} = \frac{h}{2} g$. Actually, we only need the following weaker asymptotic decay of $h$ : 
\begin{align}
	|h| + |x|\cdot |\partial  h| \leq C|x|^{-2} \text{ as }|x| \to \infty.
\end{align}

The following theorem is the main result of this paper.
\begin{thm}\label{thm-main}
	Suppose that $(M^3, g, h)$ is a complete connected asymptotically flat triple for a smooth function $h$ satisfying the dominant energy condition (\ref{DEC-h}). Assume that $M$ has one end with mass $m$, and the boundary of $M$ is compact and consists of $C^{2,\alpha }$-smooth outermost generalized apparent horizons. Then 
	\begin{align}
		m \geq \sqrt{\frac{|\Sigma |}{16 \pi }} ,
	\end{align}
	where $|\Sigma |$ is the area of any connected component of $\partial M$. Equality holds if and only if $h \equiv 0$ and $M$ is isometric to the spatial Schwarzschild manifold.
\end{thm}

In particular, this establishes the generalized Penrose conjecture for each connected component of the outermost generalized apparent horizon in the case where $\mathbf{k}$ is proportional to $g$.

When $h \equiv 0$, the theorem reduces to the time-symmetric case proved by Huisken-Ilmanen \cite{HI01}, where the authors developed a theory of weak solutions to the inverse mean curvature flow, established Geroch's monotonicity formula, and derived the Riemannian Penrose inequality as a consequence.

The non-time-symmetric setting considered here contains a large class of genuinely new cases of the Penrose conjecture beyond both the time-symmetric theory and the known spherically symmetric results. Indeed, one can construct nonspherically symmetric asymptotically flat triples satisfying the dominant energy condition for which the generalized apparent horizon strictly encloses the classical minimal horizon. In such situations, Theorem \ref{thm-main} yields a strictly stronger lower bound for the ADM mass than the classical Riemannian Penrose inequality. See Example \ref{exmp-stronger-lower} below for an explicit construction.

We now describe the strategy of the proof. The argument follows the general philosophy of Huisken-Ilmanen's proof of the Riemannian Penrose inequality, but requires two new ingredients adapted to the presence of $h$: a modified inverse mean curvature flow, which we call the $\sigma $-inverse mean curvature flow, and a corresponding generalized Hawking mass that is monotone along the flow. To the best of our knowledge, both this flow and the associated monotonicity formula are new. We now describe the construction of the flow and the main steps of the argument in more detail.

Let $\sigma $ be a Lipschitz nonnegative symmetric $2$-tensor. We introduce the $\sigma $-inverse mean curvature flow (or $\sigma $-IMCF):
\begin{align}\label{sigma-IMCF}
	\frac{\partial F}{\partial t} (x, t) = \frac{\nu }{H- |\nu |^2_{\sigma }}(x,t),\quad x \in N	,\  0 \leq t \leq T,
\end{align}
where $F: N ^{n-1} \times [0, T] \to M^{n}$ is a family of hypersurfaces $N_t := F(N, t)$, $H$ is the mean curvature of $N_t$, and $\nu $ is the outward unit normal. Here, $H- |\nu |^2_{\sigma }= H- \sigma (\nu , \nu )$ is assumed  positive for classical solutions, and $\frac{\partial F}{\partial t}$ denotes the normal velocity along the hypersurface $N_t$. 

In the case when $\sigma = |h|\cdot g$, (\ref{sigma-IMCF}) reduces to
\begin{align}\label{h-IMCF}
	\frac{\partial F}{\partial t}(x,t) = \frac{\nu }{H- |h|} (x,t), \quad x \in N,\  0 \leq t \leq T.
\end{align}
A key structural feature of this flow is the existence of natural monotonicity formulas for several associated geometric quantities. From now on we specialize to $n=3$. 

Define 
\begin{align}\label{defn-A-B}
	 A(t):= e ^{-t}|N_t|, \quad B(t):= e ^{\frac{t}{2}} \left( 1- \frac{1}{16 \pi } \int_{N_t}(H- |h|)^2 \right),
\end{align}
and introduce the generalized Hawking mass
\begin{align}\label{defn-H-mass0}
	m_h(N_t) := \sqrt{\frac{A(t)}{16 \pi }} \cdot B(t).
\end{align}

Under the dominant energy condition (\ref{DEC-h}), one shows that  $A(t)$, $B(t)$, and hence $m_h(N_t)$, are monotone nondecreasing along classical solutions of (\ref{h-IMCF}) as long as the evolving surfaces $N_t$ remain connected; see Section \ref{sect-monotone-classical}. Furthermore, it can be proved that $\lim_{t\to \infty}  m_h(N_t) $ is bounded above by the ADM mass $m$ (cf. Lemma \ref{lem-asymp-comp}). Consequently, the generalized Penrose inequality would follow from the existence of a classical solution of the $(|h|g)$-IMCF emanating from the generalized apparent horizon.

The central task is therefore to construct a weak solution to (\ref{h-IMCF}) and to establish the associated monotonicity formula in this weak setting. To develop such a weak formulation, it is necessary to work with the more general $\sigma $-IMCF (\ref{sigma-IMCF}). There have been extensive studies of weak solutions for several related formulations of generalized inverse mean curvature flows. In particular, Moore \cite{Moore12} considered the flow $\frac{\partial F}{\partial t} = \frac{\nu }{H + \mathrm{tr}_{\Sigma } \mathbf{k}}$ under the assumption $\mathrm{tr}_g \mathbf{k} \geq 0$, while Huisken–Wolff \cite{HuiskenWolff22} studied $\frac{\partial F}{\partial t} = \frac{\nu }{\sqrt{H^2 - (\mathrm{tr}_{\Sigma }\mathbf{k})^2} }$ in the maximal case $\mathrm{tr}_{g} \mathbf{k} =0$. By contrast, the $\sigma $-IMCF studied here possesses a fundamentally different structure and cannot be reduced to either of these flows, even when $\mathbf{k}$ is proportional to $g$. Nevertheless, by adapting techniques developed in \cite{HI01, Moore12, HuiskenWolff22} together with several necessary modifications, we establish a weak theory for $\sigma $-IMCF; see Theorem \ref{sgm-IMCF-existence} and Theorem \ref{h-IMCF-existence}. After constructing weak solutions, the monotonicity formula can be justified through arguments analogous to those in \cite{HI01}. One important difference, however, is that uniqueness is not known a priori in our setting, requiring a more delicate analysis. This ultimately yields a more general growth formula of $B(t)$, stated in Theorem \ref{thm-growth}, which requires only the dominant energy condition rather than nonnegative scalar curvature.

\textbf{Outline of the paper.} In Section \ref{sect-DEC}, we include some observations of the dominant energy condition, and briefly introduce an alternative proof of the positive mass theorem. In Section \ref{sect-classical}, we discuss basic evolution equations and interior estimates of the classical  $\sigma $-IMCF. In Section \ref{sect-weak}, we develop the weak theory for the $\sigma $-IMCF and prove the existence theorem. For the reader’s convenience and clarity of exposition, we only include main results in this section and leave most technical details in the Appendix \ref{appen-weak}. In Section \ref{sect-hg-IMCF}, we establish several properties of the $(|h|g)$-IMCF. We also prove a monotonicity formula in the classical sense.  In Section \ref{sect-monotonicity}, we establish the monotonicity formula for the weak $(|h|g)$-IMCF. In Section \ref{sect-proof-main}, we finish the proof of the main theorem.

\textbf{Acknowledgement.}
I would like to thank Hubert Bray for his supervision and encouragement.
I also thank Demetre Kazaras, Marcus Khuri and Kai Xu for their interest and discussions.

\section{Dominant energy condition and the generalized PMT}\label{sect-DEC}
In this section, we make some observations about the dominant energy condition, which can be viewed as a generalization of the nonnegative scalar curvature condition, and as  motivation for the $\sigma $-IMCF that we will study later. 

Given a complete connected asymptotically flat manifold $(M,g)$, define
\begin{align*}
	\mathcal{H}_0(M,g) &:= \left\{ h : R + \frac{3}{2} h ^2 - 2 |\nabla h| \geq 0,  (M,g, h) \text{ is an asymptotically flat triple} \right\}.
\end{align*} 
Notice that 
\begin{align*}
	0 \in \mathcal{H}_0(M,g) \iff R \geq 0.
\end{align*}

The following proposition shows that the condition $\mathcal{H}_{0}(M,g) \neq \emptyset$ reflects an underlying nonnegative scalar curvature phenomenon.

\begin{prop}
	If $R(x) \leq  0$ for all $x \in M$, then either $\mathcal{H}_0(M,g) = \emptyset$ or $R \equiv 0$ and $\mathcal{H}_0(M,g) = \{0\} $.
\end{prop}

\begin{proof}
	Assume that $R \leq 0$ and $h \in \mathcal{H}_0(M,g)$. Set $\tau = \frac{3}{2} h$. Choose one base point and a unit-length geodesic $\gamma (t), t \in [0, \infty),$ from the base point to one end. Consider $\tau (t) = \tau (\gamma (t) )$. Then
	\begin{align*}
		|\tau '(t)| \leq |\nabla \tau | (\gamma (t) ) \leq \frac{1}{2} \tau ^2.
	\end{align*} 
	 As $t\to \infty$, we know
	\begin{align*}
		|\tau (t)| \leq C (1+ t ^2)^{-1} \ll  t ^{-1}.
	\end{align*} 
	 Fix $\beta > \frac{1}{4}$. For large enough $t_0 \gg 1$, we have $\beta |\tau (t)| \leq (1+t) ^{-1}$. Let's assume that $t_1$ is the outermost $t$ so that $\beta |\tau (t)| \leq (1+t) ^{-1}$. Assume that $t_1 > 0$, then we have
	 \begin{align*}
	 	\beta |\tau (t)| \leq (1+t)^{-1},\ \forall t \geq t_1,\ \beta |\tau (t_1) | = (1+ t_1) ^{-1}>0.
	 \end{align*} 
	 Consider the inverval $(t_1, t_2)$ where $|\tau (t)| >0$. Such $t_2>t_1$ always exist. For any $t \in (t_1, t_2)$, 
	 \begin{align*}
		 \left|\frac{d}{d t} \log \frac{|\tau (t)|}{|\tau (t_0)|} \right| &= \frac{|\tau '(t)|}{|\tau (t)|} \leq \frac{1}{2} |\tau (t)| \leq \frac{1}{2 \beta } (1+t)^{-1}.
	 \end{align*} 
	 So
	 \begin{align*}
	 	|\tau (t)| \geq |\tau (t_1)| \cdot \left( \frac{1+t}{1+ t_1} \right) ^{- \frac{1}{2 \beta }} >0.
	 \end{align*} 
	 This also shows that $|\tau (t)| >0 $ and above estimate holds for all $t \geq t_1$. 

	 As $t\to \infty$, we have the estimate that
\begin{align*}
	 t ^{-\frac{1}{2 \beta }} \leq C\cdot t ^{-2},
\end{align*}
a contradiction with our choice of $\beta $ that $2- \frac{1}{2 \beta }>0$. Thus $t_1 = 0$, and we have 
\begin{align*}
	|\tau (t)| \leq \frac{1}{\beta (1+t)},\ \  \forall \beta > \frac{1}{4}.
\end{align*} 
Taking $\beta \to \infty$ implies that $|\tau | \equiv 0$, i.e. $ \mathcal{H}_0(M,g) = \{0\} $. Together with the assumption, we also have $R \equiv 0$.
\end{proof}

On the other hand, when $(M,g)$ possesses positive scalar curvature, one has $\mathcal{H}_{0}(M, g) \neq \emptyset$. More precisely, assume $(M,g)$ satisfies $R \geq 0$ everywhere and $R>0$ on some open subset. Then, by choosing a smooth function $h$ supported in the region where $R>0$ and with sufficiently small $C^1$-norm, one obtains $h \in \mathcal{H}_{0}(M,g)$.

The following example constructs an asymptotically flat triple satisfying the dominant energy condition and shows that Theorem \ref{thm-main} applies to new cases of the Penrose inequality beyond the classical Riemannian setting and the previously known spherically symmetric settings.

\begin{exmp}\label{exmp-stronger-lower}
Start with the doubled Schwarzschild metric
\begin{align*}
	(M, g_m) = (\mathbb{R}\times S^2,  d s ^2 + r(s) ^2 g_{S^2} ),
\end{align*} 
where
\begin{align*}
	r(s) \geq 2m,\ r(0) = 2m,\ r'(0)=0,
\end{align*} 
and
\begin{align*}
	(r'(s) )^2 = 1- \frac{2m}{r(s)}.
\end{align*} 
The scalar curvature is $R \equiv 0$ and the mean curvature of $s$-slice is $H(s) = \frac{2 r'(s)}{r(s)}$. 

Choose a smooth positive even function $h$ such that 
$$h(s) = \frac{\delta }{(1+ |s|)^3} $$
for sufficiently large $|s|$, where $\delta >0$ is a small constant. Note that $h$ need not be rotationally symmetric. 
Define 
$$F:= 2|\nabla h|_{g_m} - \frac{3}{2} h ^2.$$ 
As $|s| \to \infty$, the leading term of $F$ is $\frac{6 \delta }{(1+ |s|) ^{4}}>0$, hence $F$ is positive for all sufficiently large $|s|$. 

Let $F_{+} = \max(F, 0)$, $F_{-}= \max (-F, 0)$, and solve 
\begin{align*}
	- \Delta _{g_m} v = F_{+} ,\quad v\to 0 \text{ as }|s| \to \infty.
\end{align*} 
 By the maximum principle, $v>0$. Define 
 $$u = 1+ v, \quad g_{u} = u ^{4} g_{m}. $$
 Then $g_{u}$ is asymptotically flat, and the conformal scalar curvature formula together with $u>1$ yields
\begin{align*}
	R_{g_u} + \frac{3}{2} h^2 - 2 u ^{-2}|\nabla h|_{g_m}  \geq   8 u ^{-5} F_{+} + \frac{3}{2} h^2 -2 |\nabla h|_{g_m} = (8 u ^{-5} -1 ) F_{+}+ F_{-}.
\end{align*}
Moreover, by taking $\delta $ sufficiently small, the solution $v$ becomes uniformly small, so that $u$ is sufficiently close to $1$. In particular, $8 u ^{-5} - 1>0$, and hence
\begin{align*}
	R_{g_u} + \frac{3}{2} h^2 - 2|\nabla h|_{g_u} \geq 0.
\end{align*} 
This shows that $(M, g_{u}, \frac{h}{2} g_{u})$ defines an asymptotically flat initial data set satisfying the dominant energy condition. 

Since $F_{+}$ is even, $u$ is symmetric with respect to the middle sphere $\Sigma _0 = \{s=0\} $. Consequently, $\Sigma _0$ remains minimal with respect to $g_{u}$. Since $h>0 $ on $\Sigma _0$, the surface $\Sigma _0$ is generalized trapped. Therefore, by Eichmair's theorem, there exists a unique outermost generalized apparent horizon $\tilde{\Sigma }$, satisfying $H = h>0$. In particular, $\tilde{\Sigma }$ strictly encloses any minimal surface. 

Finally, one may further perturb $g_{m}$ so that it's no longer spherically symmetric and obtain constructions with similar properties. For such examples, Theorem \ref{thm-main} gives a strictly stronger conclusion. 

\end{exmp}

Now we scratch an alternative proof of the positive mass theorem, using $\mu $-bubbles and the dominant energy condition.
\begin{thm}
	Assume that $(M^3, g, \mathbf{k})$ is a complete asymptotically flat Cauchy data with $\mathbf{k} = \frac{\tau }{3}g$ for some smooth function $\tau $. Assume that $\mathbf{k}$ satisfies the dominant energy condition.
Let $m$ be the ADM mass.
Then
\begin{align*}
	m \geq 0.
\end{align*}  

\end{thm}

\begin{proof}
	By the dominant energy condition, for $h= \frac{2}{3} \tau $, we have
		\begin{align*}
		R + \frac{3}{2}|h|^2 - 2|\nabla |h| | &\geq 0.
	\end{align*} 
	We follow the same argument of \cite[Theorem 4.2]{Schoen87}, and only indicate the necessary modifications. The proof is done by contradiction and assuming $m<0$. 	
	Instead of minimal surface, one can consider $\mu $-bubbles, which by definition are minimizers of 
		\begin{align*}
		\mathcal{F}(\Omega ) = |\partial ^* \Omega | - \int_{\Omega } |h| ,
	\end{align*} 
	with appropriate obstacles. See also definition of the outward optimizing hull (\ref{defn-optimizing-hull}).

Let $\Sigma _{\sigma } = \partial ^* \Omega $ be the minimizer with boundary in a large cylinder $C_{\sigma }$ with radius $\sigma $. By the first variation, we have 
\begin{align*}
	H_{\Sigma _{\sigma }} = |h| ,
\end{align*} 
with respect to outer normal $\nu $. By the second variation,  for any compactly supported function $f$ on $\Sigma _{\sigma }$, we have
\begin{align*}
	\int_{\Sigma _{\sigma }} |\nabla ^{\Sigma } f|^2 - \int_{\Sigma _{\sigma }}\left( |\mathrm{II}|^2 + \mathrm{Ric}(\nu ,\nu ) ) + \langle \nabla |h| , \nu  \rangle \right) f^2  \geq 0.
\end{align*} 
Using the Gauss equation
\begin{align*}
	|\mathrm{II}|^2 + \mathrm{Ric}(\nu ,\nu ) = \frac{1}{2}\left( R +|\mathrm{II}|^2+  H^2 - 2 K \right) ,
\end{align*} 
we have
\begin{align*}
	\int_{\Sigma _{\sigma }} \left( |\nabla ^{\Sigma }f|^2 + K f^2\right) &\geq \frac{1}{2} \int_{\Sigma _{\sigma }} \left( R + \frac{3}{2} h^2 - 2|\nabla |h| | \right)  f^2 >0,
\end{align*} 
unless $R+ \frac{3}{2} h^2 - 2|\nabla |h| | \equiv 0$ and $\mathrm{II} = \mathrm{const}\cdot  \mathrm{Id}$, which is impossible on $C_{\sigma }$. Taking $\sigma \to \infty$, we obtain a contradiction with the Gauss-Bonnet theorem.

\end{proof}

\section{Classical solution of the $\sigma $-IMCF}\label{sect-classical}

Let $(M^n, g)$ be a smooth Riemannian manifold with $n \geq 3$, and $\sigma $ be a smooth nonnegative symmetric $2$-tensor.  Let $F: N^{n-1} \times [0, T] \to M^{n}$ be a family of hypersurfaces $N_t:= F(N, t)$. We say $F$ is a classical solution of the \textit{$\sigma $-inverse mean curvature flow} (or $\sigma $-IMCF) if the following parabolic evolution equation holds in the classical sense:
\begin{align}\label{*-eq}
	\frac{\partial F}{\partial t}(x,t) = \frac{\nu }{H- |\nu |^2_{\sigma }}(x,t),\quad x \in N,\ 0 \leq t \leq T, \tag*{$(*)$}
\end{align} 
where $H$ is the mean curvature of $N_t$ at the point $F(x, t)$, $\nu $ is the outward unit normal, $|\nu |^2_{\sigma }= \sigma _{ij} \nu ^{i} \nu ^{j} \geq 0$, $H-|\nu |^2_{\sigma }$ is assumed to be positive, $\frac{\partial F}{\partial t}$ denotes the normal velocity along the surface $N_t$, and all derivatives and geometric quantities are well defined.

\begin{lem}\label{evolution-eq-classical}
	Classical solutions of \ref{*-eq} satisfy the following evolution equations:
	\begin{itemize}
		\item[1)]
			$$\frac{d}{d t}|N_t| = \int_{N_t} \frac{H}{H- |\nu |^2_{\sigma }} dA;$$
		\item[2)]
			$$\frac{\partial }{\partial  t} \nu =   \frac{  \nabla ^{N}( H- |\nu |^2_{\sigma } )}{(H- |\nu |^2_{\sigma })^2}  ;$$
		\item[3)]
			$$\frac{\partial }{\partial t}H =   \frac{\Delta ^{N} (H-|\nu |^2_{\sigma }) }{(H-|\nu |^2_{\sigma })^2} - \frac{2|\nabla ^{N} (H- |\nu |^2_{\sigma })|^2}{(H-|\nu |^2_{\sigma })^3} - \frac{\mathrm{Ric}(\nu ,\nu ) + |\mathrm{II}|^2}{H-|\nu |^2_{\sigma }};$$
		\item[4)]
			$$\frac{\partial }{\partial  t} |\nu |_{\sigma }^2 = \frac{(\nabla _{\nu } \sigma ) (\nu , \nu )}{H- |\nu |^2_{\sigma }} + \frac{2 \sigma(\nu ,\nabla ^{N}(H- |\nu |^2_{\sigma }) )}{(H- |\nu |^2_{\sigma })^2};$$
		\item[5)]
			\begin{align*}
				&\frac{\partial }{\partial t}\frac{(H - |\nu |^2_{\sigma })^2}{2} \\
				&= \frac{\Delta ^{N} (H- |\nu |^2_{\sigma })}{H- |\nu |^2_{\sigma }} - \frac{2|\nabla ^{N}(H-|\nu |^2_{\sigma })|^2}{(H- |\nu |^2_{\sigma })^2} - \left( \mathrm{Ric}(\nu ,\nu ) + |\mathrm{II}|^2 \right)\\
				&\ \ - (\nabla _{\nu }\sigma )(\nu ,\nu ) - \frac{2\sigma ( \nu , \nabla ^{N}(H-|\nu |^2_{\sigma }) )}{H-|\nu |^2_{\sigma }}.
			\end{align*} 
	\end{itemize}
\end{lem}
\begin{proof}
	The proof is included in Section \ref{appen-sub-classical}. 
\end{proof}

We have the following local estimates of mean curvature for classical solutions of \ref{*-eq}.
\begin{lem}\label{local-estimate-mean-curv}
Let $(N_t)_{0 \leq t \leq T}$ be a classical solution of \ref{*-eq} on $M^{n}$, where $N_t$ may have boundary. Then there exists a positive function $\eta (x) >0$ on $M$, depending on $g$, so that for each $x \in N_t$ and each $r< \eta(x)$, we have 
	\begin{align}
		(H - |\nu |^2_{\sigma })(x, t) \leq \max \left( (H-|\nu |^2_{\sigma })_r, \frac{C(n, r \|\sigma \|_{C^0}, r ^2 \|\nabla \sigma \|_{C^0})}{r} \right) ,
	\end{align}
	where $ (H-|\nu |^2_{\sigma })_r = \sup_{(y,s) \in P_r} (H-|\nu |^2_{\sigma })(y,s)$, and $P_r = (B_r(x) \cap N_0) \times \{0\} \cup \left( \cup _{0 \leq s \leq t}(B_r(x)\cap \partial N_s) \times \{ s\}  \right) $ is the parabolic boundary. 
\end{lem}

\begin{proof}
	The proof is based on elliptic estimates of evolution equations, and is included in Section \ref{appen-sub-classical}.
\end{proof}
\begin{rmk}\label{rmk-choice-eta}
	If $(M^n, g, \sigma  )$ is asymptotically flat, we may take $\eta(x) \geq c |x| $ for some constant $c>0$, where $x$ is the asymptotically flat coordinate.
\end{rmk}

For later use, we derive the following identities. For any smooth cutoff function $\phi $ so that $\mathrm{supp} \phi \cap \partial N_{t} = \emptyset$, using the evolution equations, we have
\begin{align}\label{ddt-phi-area}
	\begin{split}
		\frac{d}{d t} \int_{N_t} \phi &= \int_{N_t} \frac{\nabla _{\nu } \phi + \phi H}{H-|\nu |^2_{\sigma }}.
	\end{split}
\end{align}

Using the evolution equations and integration by parts, we obtain
\begin{align}\label{dt-int-H-nu}
	\begin{split}
		&\frac{d}{d t} \int_{N_t} \phi (H- |\nu |^2_{\sigma })^2 \\
		&= \int_{N_t} (H- |\nu |^2_{\sigma }) \nabla _{\nu } \phi + \phi  \frac{\partial }{\partial t}(H-|\nu |^2_{\sigma })^2 + \phi H (H-|\nu |^2_{\sigma })\\
									&= \int_{N_t} (H-|\nu |^2_{\sigma }) \nabla _{\nu }\phi -2 \langle \nabla ^{N} \phi , \frac{\nabla ^{N}(H-|\nu |^2_{\sigma })}{H- |\nu |^2_{\sigma }} \rangle \\
									&\ \ - 2 \int_{N_t}\phi \left(   \frac{|\nabla ^{N} (H-|\nu |^2_{\sigma })|^2}{(H-|\nu |^2_{\sigma })^2} + \left( \mathrm{Ric}(\nu ,\nu ) +|\mathrm{II}|^2 \right) \right) \\
									&\ \ -2 \int_{N_t} \phi (\nabla _{\nu } \sigma) (\nu , \nu ) - 4 \int_{N_t} \phi \frac{\sigma ( \nu , \nabla ^{N}(H- |\nu |^2_{\sigma }) )}{H- |\nu |^2_{\sigma }} + \int_{N_t}\phi H(H-|\nu |^2_{\sigma }).
	\end{split}
\end{align}

\section{Weak theory of the $\sigma $-IMCF}\label{sect-weak}

\subsection{Level-set description and elliptic regularization}
The level-set description of the evolution by $\sigma $-IMCF can be formulated as follows. We assume that the evolving surfaces are given by the level-sets of a function $u: M \to \mathbb{R}$ via
\begin{align*}
	E_t:= \{ x: u(x) <t\} ,\quad N_t := \partial E_t. 
\end{align*}
When $u$ is smooth with $\nabla u \neq 0$, equation \ref{*-eq} is equivalent to 
\begin{align}\label{*2-eq}
	\mathrm{div}\left( \frac{\nabla u}{|\nabla u|} \right) = |\nabla u| + \frac{\sigma _{ij} \nabla ^{i} u \nabla ^{j}u}{|\nabla u|^2}, \tag*{$(* *)$}
\end{align}
where the left side is the mean curvature of $\{ u = t\} $ and $|\nabla u|$ is the inverse speed.

In the following, we assume that $\sigma $ is a Lipschitz nonnegative symmetric $2$-tensor and for any $0<\epsilon <1$, there exists smooth nonnegative symmetric $2$-tensor $\sigma ^{\epsilon }$, such that $\sigma ^{\epsilon }$ has uniform $C^1$-bound depending on $\|\sigma \|_{C^{0,1}}$ and independent of $\epsilon $, and $\|\sigma ^{\epsilon} - \sigma\|_{C^0} \leq \Psi (\epsilon ) $ for a function $\Psi $ satisfying $\lim_{\epsilon \to 0} \Psi (\epsilon ) =0 $. We call such $\sigma $ an admissible nonnegative symmetric $2$-tensor.

In order to solve the degenerate elliptic problem \ref{*2-eq} with Dirichlet boundary condition, we employ the following approximate equation, known as elliptic regularization:
\begin{align}\label{epsilon-eq}
	\begin{cases}
		\mathcal{E}^{\epsilon } (u^{\epsilon }) := \mathrm{div}\left( \frac{\nabla u ^{\epsilon }}{ \sqrt{|\nabla u^{\epsilon }|^2 + \epsilon ^2} } \right) - \sqrt{|\nabla u^{\epsilon }|^2+ \epsilon ^2} -  \frac{\sigma^{\epsilon } _{ij}\nabla ^{i} u^{\epsilon } \nabla ^{j} u ^{\epsilon }}{|\nabla u^{\epsilon }|^2+ \epsilon ^2} =0 &\text{ in } \Omega _{L},\\
		u ^{\epsilon }=0 &\text{ on } \partial E_0,\\
		u ^{\epsilon }= L-2 & \text{ on } \partial F_{L},
	\end{cases}
	\tag*{ $(*)_\epsilon $}
\end{align}
where $F_{L} := \{ v < L\} $ for an appropriate comparison function $v$, which is a well defined proper function outside of a compact subset $\Omega \subset M$, $\Omega _{L}:= F_{L} \setminus \bar{E}_0$ is precompact, and the relation between $\epsilon $ and $L$ will be revealed below. 

In the following, we will assume $(M^{n},g)$ is a complete connected Riemannian $n$-manifold without boundary, and there exists a proper smooth function $v$ so that $v$ is a smooth subsolution of the level-set equation \ref{*2-eq} on $M \setminus \Omega $ for some compact subset $\Omega $, i.e. with precompact initial condition. 
\begin{rmk}\label{rmk-subsolution}
	For an asymptotically flat metric and a Lipschitz $2$-tensor $\sigma $ with faster decay than $|x| ^{-1}$ at infinity, by a slight modification of the function $\alpha \log|x|$ in an asymptotically flat coordinate chart at infinity for a proper $\alpha >0$, one can construct a smooth function $v$ so that it is a subsolution of \ref{*2-eq} with precompact initial condition.
\end{rmk}

We have the following existence result for the regularised problem.
\begin{lem}\label{smooth-sol-*-epsilon}
	For any nonempty, precompact, smooth open set $E_0$ in $M$, there exists a smooth solution $u ^{\epsilon }$ of \ref{epsilon-eq} for all $0<\epsilon \leq \epsilon (L)$, where $ \epsilon (L) \to 0$ as $L\to \infty$. Moreover, we have the uniform estimates 
	\begin{align*}
		u ^{\epsilon } \geq - \epsilon \text{ in } \bar{\Omega }_{L},\ u ^{\epsilon } \geq v - 2 \text{ in } \bar{F}_{L} \setminus F_0,
	\end{align*} 
	\begin{align*}
		u ^{\epsilon } \leq L-2 \text{ in } \bar{\Omega }_{L},
	\end{align*} 
	\begin{align*}
		|\nabla u ^{\epsilon }| \leq \max(0, H_{\partial E_0}) + \epsilon \text{ on } \partial E_0,\ |\nabla u^{\epsilon }| \leq C(L) \text{ on } \partial F_L,
	\end{align*}
	\begin{align*}
		|\nabla u^{\epsilon }(x)| \leq \max_{\partial \Omega _{L} \cap B_r(x)} |\nabla u^{\epsilon }| + \epsilon + C(n, \|\sigma^{\epsilon } \|_{C^1})\cdot r ^{-1},\ x \in \bar{\Omega }_{L},
	\end{align*} 
	for any $0<r \leq \eta (x)$.
\end{lem}
\begin{proof}
	The proof is based on a barrier function argument by constructing appropriate subsolutions and supersolutions, and is included in Section \ref{appen-elliptic-reg}.
\end{proof}

\subsection{Variational formulations}

By freezing the $|\nabla u| + \sigma (\nu ,\nu )$ term on the right-hand side of \ref{*2-eq}, for compact subset $K \subset M$, we consider equation \ref{*2-eq} as the Euler-Lagrange equation of the functional
\begin{align*}
	J_{u, \nu }^{\sigma }(v) = J_{u, \nu }^{\sigma ,K}(v) := \int_{K}\left( |\nabla v| + v (|\nabla u| + \sigma (\nu ,\nu )) \right).
\end{align*} 
Given an open subset $\Omega \subset M$, a locally Lipschitz function $u$,  and a measurable vector field $\nu $ on $T\Omega $, we say $(u, \nu )$ is a \textit{weak $J^{\sigma }$-solution} in $\Omega $ provided 
\begin{align}\label{weak-J}
	J_{u, \nu }^{\sigma ,K}(u) \leq J_{u, \nu }^{\sigma ,K}(v)
\end{align}
for every locally Lipschitz function $v$ such that $\{ v \neq u\} \subset \subset \Omega $, where the integration is performed over any compact set $K$ containing $\{v \neq u\} $.

Alternatively, we define the functional
\begin{align*}
	J_{u, \nu }^{\sigma }(F) = J_{u, \nu }^{\sigma ,K}(F):= |\partial ^* F \cap K| - \int_{F \cap K} \left( |\nabla u| + \sigma (\nu ,\nu ) \right) ,
\end{align*} 
for sets $F$ of locally finite perimeter. We say that $E$ \textit{minimizes $J_{u, \nu }^{\sigma }$} in $\Omega $ if 
\begin{align}\label{mini-J}
	J_{u, \nu }^{\sigma ,K}(E) \leq J^{\sigma ,K}_{u, \nu }(F),
\end{align}
for each $F$ such that $F \Delta E:= (E \setminus F) \cup (F \setminus E) \subset \subset \Omega $, and any compact set $K$ containing $F \Delta E$.

As in \cite[Lemma 1.1]{HI01}, those two variational formulations are equivalent.
\begin{lem}\label{equivalent-J-E}
	Let $u$ be a locally Lipschitz function in the open set $\Omega $, $\nu $ a measurable vector field on $T \Omega $, and $\sigma $ a Lipschitz nonnegative symmetric $2$-tensor. Then $(u, \nu )$ is a weak $J ^{\sigma }$-solution in $\Omega $ if and only if for each $t$, $E_t:= \{ u< t\} $ minimizes $J_{u, \nu }^{\sigma }$ in $\Omega $.
\end{lem}

One can combine the above definition with an initial condition consisting of an open set $E_0$ with a boundary that is at least $C^1$. We say that $(u, \nu )$ is a \textit{weak $J^{\sigma }$-solution with initial condition $E_0$} if
\begin{align}\label{ddag}
	\begin{split}
		u &\in C^{0,1}_{loc}(M), \nu \text{ measurable on } T(M\setminus E_0),\\
	  &E_0= \{u < 0\}, \text{ and } (u, \nu ) \text{ satisfies } (\ref{weak-J}) \text{ in } M\setminus E_0. 
	\end{split}
\tag*{$(\ddag)$}
	\end{align}
	Let $E_t$ be a nested family of open sets in $M$, closed under ascending union. Define $u$ by the characterization $E_t = \{u<t\} $. Then \ref{ddag} is equivalent to 
\begin{align}\label{dag}
	\begin{split}
		u &\in C^{0,1}_{loc}(M), \nu \text{ measurable on } T(M\setminus E_0),\\
	  &E_t \text{ minimizes } J^{\sigma }_{u, \nu } \text{ in } M\setminus E_0 \text{ for each } t>0. 
	\end{split}
\tag*{$(\dag)$}
	\end{align}
	By approximating $s \searrow t$, we see that \ref{ddag} and \ref{dag} are equivalent to
	\begin{align}\label{leq-t-weak}
		 \begin{split}
			 u & \in C^{0,1}_{loc}(M), \nu \text{ measurable on } T(M\setminus E_0),\\
			   & E_t ^{+}:=\{u \leq t\} \text{ minimizes } J^{\sigma }_{u, \nu } \text{ in } M\setminus E_0 \text{ for each } t\geq 0 
		 \end{split}
	\end{align}

	In order to define weak solutions of \ref{*2-eq} and \ref{*-eq} using above variational formulations, one needs to define the normal vector field $\nu $ in an appropriate way. For this purpose, we will formulate the weak solution $(u, \nu )$ of \ref{*2-eq} one dimension higher, in terms of a translation invariant function $U$ which extends $u$ and a translation invariant vector field $\tilde{\nu }$ which extends $\frac{\tilde{\nabla }U}{|\tilde{\nabla }U|}$, so that $(U, \tilde{\nu })$ minimizes the variation functional $J_{U, \tilde{\nu }}^{\sigma }$, as considered in \cite[Definition 15]{Moore12} and \cite[Definition 6.6]{HuiskenWolff22}. 

	\begin{defn}\label{defn-weak}
	Let $E_0 \subset M$ be a precompact open set with $C^2$-boundary $N_0= \partial E_0$, and $\sigma $ be a Lipschitz nonnegative symmetric $2$-tensor. We call the pair $(U, \tilde{\nu })$ a \textit{weak solution} of \ref{*2-eq} with initial condition $E_0$ if $U \in C^{0,1}_{loc}(M \times \mathbb{R})$ and $\tilde{\nu } $ is a measurable unit vector field which satisfy
	\begin{itemize}
		\item [(i)] $U$ is translation invariant in the vertical direction in the sense that $U(x,z) = u(x)$ for a locally Lipschitz function $u: M \to \mathbb{R}$ satisfying $u \geq 0$ on $M \setminus E_0$, $u|_{\partial E_0} =0$, $u<0$ in $E_0$, and $u(x) \to \infty$ as $d(x, E_0) \to \infty$.
		\item [(ii)] $\tilde{E}_t := \{ U< t\} $ minimizes $J_{U, \tilde{\nu }} ^{\sigma }$ in $(M\setminus \bar{E}_0) \times \mathbb{R}$ for each $t>0$, and $\sigma $ the translation invariant extension. At each jump time $t_0$, each point $\tilde{x}_0=(x_0, z_0)$ in the interior $\tilde{\mathcal{K}}_{t_0}$ of the jump region $\{U = t_0\} $ lies in the boundary $\partial \tilde{E}_{\tilde{x}_0} \in C^{1,\alpha }_{loc}$ of a Caccioppoli set $\tilde{E}_{\tilde{x}_0}$ that minimizes $J_{U, \tilde{\nu }}$ in $\tilde{\mathcal{K}}_{t_0}$.
		\item [(iii)] $\tilde{\nu }$ is a translation invariant in the sense that $\tilde{\nu }(x, z_1) = \tilde{\nu} (x, z_2)$ for any $z_1, z_2 \in \mathbb{R}$, $\tilde{\nu }(\tilde{x}) \in C^{\alpha }_{loc}$ away from jump times and is the unit normal vector to $\partial \tilde{E}_t$ at each point $\tilde{x} \in \partial \tilde{E}_t$, and $\tilde{\nu }(\tilde{x}) \in C^{1,\alpha }_{loc}(\tilde{\mathcal{K}}_{t_0}) $ is the unit normal vector to $\partial \tilde{E}_{\tilde{x}_0}$ at each point $\tilde{x} \in \partial \tilde{E}_{\tilde{x}_0}$ and each jump time $t_0$.
	\end{itemize}
\end{defn}
\begin{rmk}\label{rmk-defn-weak}
	By this definition, as in \cite[Lemma 16]{Moore12}, one can show that if $(U, \tilde{\nu })$ is a weak solution of \ref{*2-eq} in the sense of above definition, then $(u, \nu:= \tilde{\nu }|_{TM} )$ is a weak $J^{\sigma }$-solution on $M \setminus E_0$ in the sense of (\ref{weak-J}), and $E_t= \{ u<t\} $ minimizes $J_{u, \nu } ^{\sigma }$ for each $t>0$. So $E_t^+$ minimizes $J^{\sigma }_{u, \nu }$ for each $t \geq 0$.
\end{rmk}

In the following, we will construct a weak solution $(U, \tilde{\nu })$ of the $\sigma $-IMCF as a limit of the $\epsilon $-translating graphs, which are equivalent to solutions of \ref{epsilon-eq} as in the proof of Lemma \ref{smooth-sol-*-epsilon}.

\subsection{Limit of $\epsilon $-translating graphs}\label{section-limit-translating}
 The $\epsilon $-translating graph is a geometric interpretation of the equation \ref{epsilon-eq}  which can be described as follows.

 For any $L \gg 1$, by Lemma \ref{smooth-sol-*-epsilon}, there exists a smooth solution $u ^{\epsilon }$ of \ref{epsilon-eq} on $\Omega _{L}$, where $\epsilon = \epsilon (L) \to 0$ as $L\to \infty$. Define
\begin{align}
	U ^{\epsilon }(x,z):= u ^{\epsilon }(x) - \epsilon z,\quad (x,z) \in \Omega _{L} \times \mathbb{R}.
\end{align}
Let $\tilde{g}= g + dz ^2$ be the product metric. Then \ref{epsilon-eq} is equivalent to 
\begin{align}
	\tilde{\mathcal{E}}(U^{\epsilon }):= \mathrm{div}_{\tilde{g}}\left( \frac{\tilde{\nabla } U^{\epsilon }}{|\tilde{\nabla } U^{\epsilon }|} \right) - |\tilde{\nabla }U ^{\epsilon }| - \frac{\sigma ^{\epsilon } _{ij} \tilde{\nabla }^{i} U^{\epsilon } \tilde{\nabla }^{j} U^{\epsilon }}{|\tilde{\nabla }U^{\epsilon }|^2} = 0 \text{ in } \Omega _{L} \times \mathbb{R},
\end{align}
where $\sigma ^{\epsilon } (x,z) = \sigma ^{\epsilon }(x)$ is the constant extension along the $z$-direction. So $U^{\epsilon }$ is a smooth solution of \ref{*2-eq} in $\Omega _{L} \times \mathbb{R}$ for $\sigma = \sigma ^{\epsilon }$.  For any $-\infty< t< \infty$, set
\begin{align*}
	\tilde{N}^{\epsilon }_t := \{ U^{\epsilon }= t\} = \mathrm{graph} \left( \frac{u^{\epsilon }}{\epsilon } - \frac{t}{\epsilon } \right) .
\end{align*} 
Then $\tilde{N}^{\epsilon }_{t}$ is a smooth solution of the $\sigma ^{\epsilon } $-IMCF \ref{*-eq} in $\Omega _{L} \times \mathbb{R}$:
\begin{align*}
	\frac{\partial }{\partial t} \tilde{F} = \frac{\tilde{\nu }_{\epsilon }}{\tilde{H} - |\tilde{\nu }_{\epsilon }|^2_{\sigma ^{\epsilon } }},
\end{align*}
where $\tilde{\nu }_{\epsilon } = \frac{\tilde{\nabla }U^{\epsilon }}{|\tilde{\nabla }U^{\epsilon }|}$ is the smooth downward unit normal vector. 
Notice that $\tilde{N}^{\epsilon }_{t}$ has boundary consisting of translations of $\partial \Omega _{L}$. 

Using the local uniform Lipschitz estimates of $u^{\epsilon }$ given in Lemma \ref{smooth-sol-*-epsilon}, by the Arzela-Ascoli theorem, there exists $L_i \to \infty, \epsilon _i \to 0$, a subsequence $u_i(x)$, and a locally Lipschitz function $u$ such that
\begin{align*}
	u_i \to u
\end{align*} 
locally uniformly on $M \setminus E_0$, and $u$ satisfies 
\begin{align}\label{u-Lip-bound}
	|\nabla u|(x) \leq \sup_{\partial E_0 \cap B_r(x)} \max(0, H_{\partial E_0}) + C(n, \|\sigma \|_{C^{0,1}}) r ^{-1},\quad \text{a.e. } x \in M\setminus E_0,
\end{align}
for each $0< r< \eta (x)$, and
\begin{align}\label{u-boundary-outer}
	u \geq 0 \text{ in } M \setminus E_0,\quad u(x) \to \infty \text{ as } d(x, E_0) \to \infty.
\end{align} 
We also extend $u$ negatively to $E_0$ so that
\begin{align}\label{u-boundary-inner}
	E_0 = \{ u< 0\} .
\end{align}

Let $U_i(x,z) = u_i(x) - \epsilon _i z$, $\tilde{N}^{i}_t = \tilde{N}^{\epsilon _i}_t, \tilde{\nu }_i = \frac{\tilde{\nabla }U_i}{|\tilde{\nabla }U_i|}$ be defined as above and $\sigma ^{i}= \sigma ^{\epsilon _i}$. Then $\tilde{N}^{i}_t$ are smooth solutions of the $\sigma ^{i}$-IMCF \ref{*-eq} in $\Omega _{L_i} \times \mathbb{R}$.
  Set $U(x, z) := u(x)$. Then
  \begin{align}\label{U_i-U}
	U_i \to U
\end{align} 
locally uniformly on $(M\setminus E_0) \times \mathbb{R}$ with local Lipschitz bounds. 

To get a good convergence of normal vectors, one applies the regularity theorem for minimizers of $J^{\sigma }$-functional. 
As in \cite[Lemma 2.3]{HI01}, we have the Smooth Flow Lemma which shows that smooth flows satisfy the weak variational formulation in the domain they foliate.
\begin{lem}\label{smooth-flow-lem}
	Let $(N_t)_{a \leq t < b}$ be a classical solution of the $\sigma $-IMCF \ref{*-eq}. Let $u=t$ on $N_t$, $u < a$ in the region bounded by $N_a$, and $E_t:= \{ u< t\} $. Then for $a \leq t< b$, $E_t$ minimizes $J_{u, \nu } ^{\sigma }$ in $E_{b} \setminus E_a$, where $\nu = \frac{\nabla u}{|\nabla u|}$ is a smooth unit vector field.	
\end{lem}
\begin{proof}
	Notice that $\mathrm{div}( \nu) = H_{N_t} = |\nabla u| + \sigma (\nu , \nu )$. Using $\nu $ as a calibration, by the divergence theorem
	\begin{align*}
		J_{u, \nu }(E_t) & = |\partial E_t|  - \int_{E_t} \left( |\nabla u| + \sigma (\nu ,\nu ) \right)\\
		&= \int_{\partial E_t} \nu _{\partial E_t}\cdot \nu - \int_{E_t} \left( |\nabla u| + \sigma (\nu ,\nu ) \right) \\
											   &= \int_{\partial ^* F} \nu _{\partial ^* F}\cdot \nu - \int_{ F} \left( |\nabla u| + \sigma (\nu ,\nu ) \right)\\
											   &\leq |\partial ^* F| - \int_{ F} \left( |\nabla u| + \sigma (\nu ,\nu ) \right),
	\end{align*} 
	for any finite perimeter set $F$ differing compactly from $E_t$.
\end{proof}

Applying this Lemma to $\tilde{N}^{i}_t$, we know that for each $i$, $\tilde{E}^{i}_t:=\{ U_{i} < t\} $ minimizes $J_{U_i, \tilde{\nu }_i}^{\sigma _i}$ on $\tilde{E}^{i}_{b} \setminus \tilde{E}^{i}_{a}$ for all $a \leq t< b$.  In the following, we can assume $n+1<8$ to avoid regularity issues of minimizing problems.

\begin{lem}\label{lem-local-C1alpha}
	For each $t$, the level sets $\tilde{N}^{i}_t $ are locally uniformly bounded in $C^{1, \alpha }$, independent of $i$ and $t$.
\end{lem}
\begin{proof}
	For each $t$, we have proved that $\tilde{E}^{i}_t:=\{U_i < t\} $ minimizes the functional 
	\begin{align*}
		|\partial ^* F| - \int_{F} \left( |\tilde{\nabla }U_i| + \sigma ^{i} (\tilde{\nu }_i, \tilde{\nu }_i) \right)
	\end{align*} 
	with respect to competitors $F$ such that $F \supset E_0 \times \mathbb{R}$ and $F \Delta \tilde{E}^{i}_t \subset \subset \Omega := \tilde{E}^{i}_{b} \setminus \tilde{E}^{i}_{a} $ for all $a \leq t < b$.
	Applying the regularity theorem \cite[Theorem 1.3]{HI01}, since $|\tilde{\nabla }U_i| + \sigma^{i} (\tilde{\nu }_i, \tilde{\nu }_i) $ is locally uniformly bounded, $\tilde{N}^{i}_{t} = \partial \tilde{E}^{i}_t$ is locally uniformly bounded in $C^{1,\alpha }$, with $C^{1,\alpha }$-estimates depend also on the distance to $\partial E_0$ and $C^{1,\alpha }$-bound of $\partial E_0$. 
\end{proof}

Thus, locally, one can represent $\tilde{N}^{i}_t$ as the graph of a function with a uniform $C^{1,\alpha }$-bound, whose unit normal vector then has uniform $C^{\alpha }$-control. In this way, one is able to construct $C^{1,\alpha }$-limiting hypersurfaces with well defined $C^{\alpha }$-normal vector field, which extends $\frac{\tilde{\nabla }U}{|\tilde{\nabla }U|}$ as a calibration across the gulfs. More precisely, we have the following lemmas. 

Set 
\begin{align*}
	\tilde{E}_{t}:= \{U<t\} ,\ \tilde{E}^{+}_{t}:= \mathrm{int} \{U \leq t\} ,\ \tilde{N}_t:= \partial \tilde{E}_t,\ \tilde{N}^{+}_t:= \partial \tilde{E}^{+}_{t} . 
\end{align*} 
We call $t$ a jump time if $\tilde{N}_t \neq \tilde{N}^{+}_t$. Note that $\tilde{N}_{0} = \partial E_0 \times \mathbb{R}$ is $C^2$ by assumption.

\begin{lem}\label{lem-limit-surface-C1alpha}
	Fix $t_0\geq 0$. 
	\begin{itemize}
		\item[(1)] If $t_0$ is not a jump time, then $\tilde{N}_{t_0} = \{U =t_0\} $ is a complete hypersurface that is locally uniformly bounded in $C^{1,\alpha }$. 
		\item[(2)] If $t_0$ is a jump time, then $\tilde{N}_{t_0}, \tilde{N}^{+}_{t_0}$ are complete hypersurfaces that are locally uniformly bounded in $C^{1,\alpha }$. 
		\item[(3)] If $t_0$ is a jump time, let $\tilde{\mathcal{K}}_{t_0}$ denote the interior of $\{U= t_0\} $, then each point $\tilde{x}_0 \in \tilde{\mathcal{K}}_{t_0}$ lies in a complete hypersurface $\tilde{N}_{\tilde{x}_0} \subset \{ U=t_0\} $ that is locally uniformly $C^{1,\alpha }$-bounded and is the limit of a subsequence $\tilde{N}^{i_j}_{t_{i_j}}$.
	\end{itemize}
\end{lem}
\begin{proof}
	The proof is based on the local graph representation of $\tilde{N}^{i}_t$, and is included in Section \ref{appen-limit-regularized}.
\end{proof}

\begin{lem}\label{limit-nu-construction}
	In Case (1) or (2) of Lemma \ref{lem-limit-surface-C1alpha}, let $\tilde{\nu }$ be the unit normal $C^{\alpha }$-vector field to $\tilde{N}_{t_0}$ or $\tilde{N}_{t_0}^{+}$, then $\tilde{\nu }_i \to \tilde{\nu }$ locally uniformly for the whole sequence. In Case (3), there exists a H\"older continuous unit vector field $\tilde{\nu }$ on $\tilde{\mathcal{K}}_{t_0}$ such that for a subsequence $i_j$, $\tilde{\nu }_{i_j} \to \tilde{\nu }$ locally uniformly, and $\tilde{\nu }$ is translation invariant and normal to $\tilde{N}_{\tilde{x}_0}$ for any $\tilde{x}_0 \in \tilde{\mathcal{K}}_{t_0}$.
\end{lem}
\begin{proof}
	The proof is based on a choice of subsequence, and is included in Section \ref{appen-limit-regularized}.
\end{proof}

We conclude this subsection by the following foliation result with improved regularity of the jump region.

\begin{lem}\label{lem-foliation-jump}
	The interior $\tilde{\mathcal{K}}_{t_0}$ of the jump region is foliated by $C^{2,\alpha }$-hypersurfaces, where each such hypersurface is either a vertical cylinder or a graph over an open subset of $\{u=t_0\} $. Furthermore, each hypersurface bounds a Caccioppoli set that minimizes $J_{U, \tilde{\nu }}$ in $\tilde{\mathcal{K}}_{t_0}$, where $\tilde{\nu }$ denotes the $C^{1,\alpha }$-unit normal vector field to the hypersurface foliation.
\end{lem}
\begin{proof}
	The proof is based on a boot-strapping argument of uniformly elliptic PDE and a compactness result, and is included in Section \ref{appen-limit-regularized}.
\end{proof}

\subsection{Existence of weak solutions}
We first introduce the following compactness theorem, whose proof is the same as \cite[Theorem 2.1]{HI01}, by applying the bounded convergence theorem to the additional term $\sigma ^{i}(\tilde{\nu }_i, \tilde{\nu }_i)$. 

\begin{lem}\label{lemma-compactness}
	Let $\tilde{\Omega }_i \subset M \times \mathbb{R}$ be open sets, let $U_i \in C^{0,1}_{loc}(\tilde{\Omega }_i)$ satisfy $\sup_{K}|\tilde{\nabla }U_i| \leq C(K)$ for each $K \subset \subset \tilde{\Omega }$ and all large $i$, let $\tilde{\nu }_i \in C^{\alpha }_{loc}(T \tilde{\Omega }_i)$ be a sequence of unit vector fields, and let $\sigma ^{i}$ be a sequence of Lipschitz nonnegative symmetric $2$-tensors satisfying $\|\sigma ^{i}\|_{C^1} \leq C$ uniformly. Assume that $\tilde{\Omega }_i \to \tilde{\Omega }$ locally uniformly, $U \in C^{0,1}_{loc}(\tilde{\Omega })$, $\tilde{\nu }$ is a measurable unit vector field on $\tilde{\Omega }$, and $\sigma $ is a Lipschitz nonnegative symmetric $2$-tensor, such that 
\begin{align*}
	U_i \to U, \ \sigma ^{i} \to \sigma \text{ locally uniformly, }\quad  \tilde{\nu }_i \to \tilde{\nu } \text{ a.e.}
\end{align*} 
If $(U_i, \tilde{\nu }_i)$ minimizes $J^{\sigma ^{i}}_{U_i, \tilde{\nu }_i}$ on $\tilde{\Omega }_i$, then $(U, \tilde{\nu })$ minimizes $J^{\sigma }_{U, \tilde{\nu }}$ on $\tilde{\Omega }$.
\end{lem}

As a limit of the $\epsilon $-translating graphs, we have the following existence theorem for the weak $\sigma $-IMCF. 
\begin{thm}\label{sgm-IMCF-existence}
	Let $(M^n, g)$ be a complete, connected Riemannian $n$-manifold without boundary, and $\sigma $ be an admissible Lipschitz nonnegative symmetric $2$-tensor with globally uniform $C^{0,1}$-bound. Suppose there exists a proper smooth subsolution of \ref{*2-eq} with a precompact initial condition. Then for any nonempty, precompact, open set $E_0 \subset M^n$ with $C^2$-boundary, there exists a weak solution of the $\sigma $-IMCF with initial condition $E_0$ as in Definition \ref{defn-weak}.
\end{thm}

\begin{proof}
	We adapt the same notations as in previous sections. So  $\tilde{N}^{i}_t = \{ U_i = t\} $ are smooth solutions of the $\sigma^{i} $-IMCF \ref{*-eq} in $\tilde{\Omega }_i:= \Omega _i \times \mathbb{R}$ for $\Omega _i = \Omega _{L_i}$ with $L_i \to \infty$. Let $\tilde{\nu }_i = \frac{\tilde{\nabla }U_i}{|\tilde{\nabla }U_i|}$ be the smooth unit normal vectors to $\tilde{N}^{i}_t$. Then $(U_i, \tilde{\nu }_i)$ are weak $J ^{\sigma ^{i}}$-solutions on $\tilde{\Omega }_i$, $\tilde{\Omega }_i \to (M \setminus \bar{E}_0 ) \times \mathbb{R}$ locally uniformly, $U_i$ have locally uniformly Lipschitz bounds and $U_i \to U$ locally uniformly for a locally Lipschitz function $U(x,z)= u(x)$ satisfying (\ref{u-Lip-bound}), (\ref{u-boundary-outer}), and (\ref{u-boundary-inner}).

	By Lemma \ref{limit-nu-construction}, away from jump regions, we can construct a locally H\"older continuous unit vector field $\tilde{\nu }$, locally as a uniform limit of the whole sequence $\tilde{\nu }_i$. By Lemma \ref{limit-nu-construction} and \ref{lem-foliation-jump}, in the interior of each jump region, we can construct $C^{1,\alpha }$ unit vector field $\tilde{\nu }$ so that a subsequence $\tilde{\nu }_{i_j} \to \tilde{\nu }$ locally uniformly. Since there are at most countable jump times, by taking successive subsequences once more which we still denote by $i$, we obtain a measurable translation invariant unit vector field $\tilde{\nu }$ on all of $(M\setminus E_0) \times \mathbb{R}$, such that $\tilde{\nu }_i \to \tilde{\nu }$ a.e.

	Moreover, by the construction, we know that $\tilde{\nu }(x) = \frac{\tilde{\nabla }U}{|\tilde{\nabla }U|}(x)$ if $x \in \tilde{N}_t$ at a regular time $t$; $\tilde{\nu }(x) = \lim_{j\to \infty} \frac{\tilde{\nabla }U}{|\tilde{\nabla }U|}(x_j) $ if $x \in \tilde{N}_{t_0}$,  for $x_j \to x$, $x_j \in \tilde{N}_{t_j}$ and $t_j \nearrow t_0$; $\tilde{\nu }(x) = \lim_{j\to \infty}  \frac{\tilde{\nabla }U}{|\tilde{\nabla }U|}(x_j) $ if $x \in \tilde{N}_{t_0} ^{+}$,  for $x_j \to x$, $x_j \in \tilde{N}_{t_j}$ and $t_j \searrow t_0$; and $\tilde{\nu }$ agrees with the constructed $C^{1,\alpha }$-vector in $\tilde{\mathcal{K}}_{t_0}$ at a jump time $t_0$. So $\tilde{\nu }$ is a translation invariant measurable unit vector field and $J^{\sigma }_{U, \tilde{\nu }}$ is well defined on $(M \setminus \bar{E}_0) \times \mathbb{R}$. 

	It remains to verify (ii) of Definition \ref{defn-weak}. The foliation of $\tilde{\mathcal{K}}_{t_0}$ by Caccioppoli sets is included in Lemma \ref{lem-foliation-jump}. Applying the Compactness Lemma \ref{lemma-compactness} to $(\tilde{U}_i, \tilde{\nu }_i, \sigma ^{i})$, and using Lemma \ref{equivalent-J-E}, we know $\{U< t\} $ minimizes $J^{\sigma }_{U, \tilde{\nu }}$ for each $t>0$. This completes the proof.
\end{proof}

\subsection{Weak mean curvature} We give a PDE interpretation of our weak solutions in terms of weak mean curvature. Let $\tilde{N} \subset M ^{n} \times \mathbb{R}$ be a $C^1$-hypersurface. Then a locally integrable function $H_{\tilde{N}}$
 on $\tilde{N}$ is called the weak mean curvature if 
\begin{align*}
 \int_{\tilde{N}} \mathrm{div}_{\tilde{N}} X d \mathcal{H}^{n}= \int_{\tilde{N}} H_{\tilde{N}} \tilde{\nu }\cdot X d \mathcal{H}^{n},\quad \forall X \in C^{\infty}_{c}(TM \times \mathbb{R}).	
\end{align*} 
Here  $\mathrm{div}_{\tilde{N}} X = \sum_{i}^{} \langle \tilde{\nabla }_{e_i} X, e_i  \rangle $ for any orthonormal basis of $T_{x}\tilde{N}$ and $\tilde{\nu }$ is the unit normal vector.

\begin{lem}\label{lem-weak-mean-curvature}
	Let $(U, \tilde{\nu })$ be a weak solution constructed in Theorem \ref{sgm-IMCF-existence}. Then the $C^{1,\alpha }$-hypersurfaces $\tilde{N}_t = \partial \{ U < t\} $ have weak mean curvature $H_{\tilde{N}_t} = |\tilde{\nabla }U| + |\tilde{\nu }|^2_{\sigma }$ for a.e. $t \geq 0$ and $\mathcal{H}^{n}$-a.e. $x \in \tilde{N}_{t}$. In particular, $H_{\tilde{N}_t} > |\tilde{\nu }|^2_{\sigma }$ for a.e. $t \geq 0$ and $\mathcal{H}^{n}$-a.e. $x \in \tilde{N}_{t}$.
\end{lem}
\begin{proof}
	 Let $X$ be a smooth vector field with compact support contained in an open subset $\tilde{W}$ and $\Phi _{s}$ the flow of diffeomorphisms generated by $X$ with $\Phi _{0}= \mathrm{id}$. By Lemma \ref{lem-limit-surface-C1alpha}, $\tilde{N}_{t}$ is $C^{1,\alpha }$-hypersurface. Since $\tilde{E}_t:= \{ U< t\} $ minimizes $J^{\sigma }_{U, \tilde{\nu }}$ for $t>0$, by the coarea formula and the dominated convergence theorem, together with the fact that $\tilde{\nu }= \frac{\tilde{\nabla }U}{|\tilde{\nabla }U|}$ when $|\tilde{\nabla }U| \neq 0$, we obtain 
	 \begin{align*}
	 	0 = \left.\frac{d}{d s}\right|_{s=0} J^{\sigma }_{U, \tilde{\nu }}(U\circ \Phi _s^{-1}) = \int_{-\infty}^{\infty} \int_{\tilde{N}_{t} \cap \tilde{W}} \left( \mathrm{div}_{\tilde{N}_t}X -  \langle \tilde{\nu }, X  \rangle  (|\tilde{\nabla }U| + |\tilde{\nu }|^2_{\sigma }) \right)  d \mathcal{H}^{n} dt.
	 \end{align*} 
By Riesz representation theorem and Lebesgue differentiation theorem, this proves the conclusion.
\end{proof}

\section{Further properties of the weak $(|h|g)$-IMCF}\label{sect-hg-IMCF}
Let $(M^3, g, h)$ be a complete, connected, asymptotically flat triple of dimension $3$. Consider the Lipschitz nonnegative symmetric $2$-tensor $\sigma_{|h|} := |h| g$. Then the $(|h|g)$-IMCF equation \ref{*-eq} is reduced to equation (\ref{h-IMCF}):
\begin{align}\label{*-eq-h}
	\frac{\partial F}{\partial t} = \frac{\nu }{H - |h|},
\end{align}
and the level set equation is reduced to
\begin{align}\label{*2-eq-h}
	\mathrm{div}\left( \frac{\nabla u}{|\nabla u|} \right) = |\nabla u| + |h|.
\end{align}

By Remark \ref{rmk-subsolution}, we know $(M^3, g)$ supports a smooth subsolution of (\ref{*2-eq-h}) with precompact initial condition. We show that $\sigma _{|h|}$ is also admissible. To see this, choose a fixed smooth function $\Phi (x)$ on $M$ so that $0< \Phi (x) < 1$, $\Phi (x) \sim |x| ^{-3}$ as $x \to \infty$, and for any $0< \delta < 1$, define
	\begin{align}
		h_{\delta }:= \sqrt{h ^2 + (\delta  \Phi )^2} .
	\end{align}
	Then $h_{\delta }$ is smooth, asymptotic to $0$ at the rate $|x|^{-3}$,
\begin{align*}
	|\nabla h_{\delta }| = \frac{|h \nabla h + \delta ^2 \Phi \nabla \Phi |}{\sqrt{h ^2 + (\delta \Phi )^2} } \leq |\nabla h| + \delta |\nabla \Phi | \leq |\nabla h| + C\cdot \delta ,
\end{align*} 
and
\begin{align*}
	h_{\delta }- |h|  &= \frac{(\delta \Phi )^2}{\sqrt{h^2+ (\delta \Phi )^2} + |h|} \in (0, C\cdot \delta ).
\end{align*} 
So $h_{\delta } \cdot g$ gives a desired smooth approximation of $\sigma _{|h|}$. By Theorem \ref{sgm-IMCF-existence}, we have the following existence theorem of the weak solution of (\ref{*-eq-h}).
\begin{thm}\label{h-IMCF-existence}
	Let $(M^3, g, h)$ be a complete, connected, asymptotically flat triple without boundary. Then for any nonempty, precompact, open set $E_0 \subset M$ with $C^2$-boundary, there exists a weak solution $(U, \tilde{\nu })$ of the $(|h| g) $-IMCF  with initial condition $E_0$ as in Definition \ref{defn-weak}.
\end{thm}

In the following, we denote by $(u, \nu )= (U, \tilde{\nu }|_{TM})$, and we will simply say that $u$ is a weak solution of the $(|h|g)$-IMCF. We write $E_t:= \{ u< t\} $, $E_{t}^{+}:= \{ u \leq t\} $, $N_t = \partial E_t, N_t ^{+} = \partial E_t ^{+}$, and
\begin{align*}
	J_{u}^{|h|}(v) &= \int \left( |\nabla v| + v\left( |\nabla u| + |h| \right)  \right) ,\\
	J_{u}^{|h|}(F) &= |\partial ^* F| - \int_{F}\left( |\nabla u| + |h| \right) .
\end{align*}

\subsection{Outward optimizing hulls}
In this section, we introduce the notation of outward optimizing hulls.
A similar notation has been introduced in \cite{HI01, Moore12, HuiskenWolff22} in their corresponding settings. 

Let $\Omega $ be an open set. We call $E$ an \textit{outward optimizing hull} (with weight $|h|$ in $\Omega $) if $E$ minimizes area minus bulk energy $|h|$ on the outside in $\Omega $, that is, if
\begin{align}\label{defn-optimizing-hull}
	|\partial ^* E \cap K| - \int_{E \cap K} |h| \leq |\partial ^* F \cap K| - \int_{F \cap K} |h|,
\end{align}
for any $F$ containing $E$ such that $F \setminus E \subset \subset \Omega $, and any compact set $K$ containing $F \setminus E$. We say that $E$ is \textit{strictly outward optimizing} (with weight $|h|$ in $\Omega $) if equality implies that $F \cap \Omega = E \cap \Omega $ a.e. 

Let $E$ be any measurable set. Define $E' = E'_{\Omega }$ to be the intersection of (the Lebesgue points of) all the strictly outward optimizing hulls in $\Omega $ that contain $E$. Up to sets of measure zero, this may be realized by a countable intersection, so $E'$ itself is a strictly outward optimizing hull and open. Due to the asymptotic flatness of $g$ and $h$, if $E \subset \subset M$, then $E'$, taken in $M$, exists and is precompact as well. See also \cite{Eichmair10}. In particular, if $M \setminus E_0$ is an exterior region, since $\partial E_0$ is outermost generalized apparent horizon, we know $E_0$ is strictly outward optimizing.

If $\partial E$ is $C^2$, then by the regularity result of \cite{BK74}, we know $\partial E'$ is $C^{1,1}$, depending on the $C^{0,1}$-norm of $|h|$ and $C^2$-bound of $\partial E$, and $C^{2,\alpha }$ where it does not contact the obstacle $E$. For the weak mean curvature, we have
\begin{align}\label{mean-curv-hull}
	\begin{split}
		H_{\partial E'} -|h| &= 0 \text{ on } \partial E' \setminus \partial E,\\
		H_{\partial E'} - |h| &= H_{\partial E} - |h| \geq 0 \quad \mathcal{H}^{n-1}\text{-a.e. on } \partial E' \cap \partial E.
	\end{split}
\end{align}

\begin{lem}\label{lem-optimizing-property}
	Suppose that $u$ is a weak solution of the $(|h|g)$-IMCF with initial condition $E_0$. Then
	\begin{itemize}
		\item [(i)] $E_t$ is outward optimizing in $M$ for $t>0$;
		\item [(ii)] $E_t ^{+}$ is outward optimizing in $M$ for $t \geq 0$;
		\item [(iii)] $|\partial E_t ^{+}| = |\partial E_t| + \int_{E_t ^{+}\setminus E_t}|h|$, for all $t>0$. This extends to $t=0$ precisely if $E_0$ is outward optimizing.
	\end{itemize}
\end{lem}
\begin{proof}
	The proof is similar to \cite[Theorem 1.4]{HI01}, and is included in Section \ref{appen-optimizing}.
\end{proof}

\begin{lem}\label{E_t+-vs-E_t'}
	Suppose that $u$ is a weak solution of the $(|h|g)$-IMCF with initial condition $E_0$. Let $\Omega $ be a domain and $t \geq 0$ such that $E_t^+ \subset \Omega $, and assume $E_t$ admits a precompact outward optimizing hull $(E_{t})_{\Omega }'$ in $\Omega $. Then 
 $E_t ^+ \subset (E_t)_{\Omega }'$. In particular, if $E_t'= E_t$, then $E_t ^+ = E_t$.
\end{lem}
\begin{proof}
	The proof is similar to \cite[Proposition 8.3]{HuiskenWolff22}, and is included in Section \ref{appen-optimizing}. 
\end{proof}

\subsection{Uniqueness and smoothness of weak solutions}
Regarding the weak solutions of the $(|h|g)$-IMCF, or more generally of the $\sigma $-IMCF, a natural question is the uniqueness problem.  Unfortunately, unlike the standard IMCF, \cite[Uniqueness Theorem 2.2]{HI01} doesn't extend directly. It seems that the variation minimizing property alone doesn't imply the uniqueness. See also \cite[Remark 6.7]{HuiskenWolff22}. Instead, in the following, we only consider a special case
where a classical solution exists for a short time and show that our constructed weak solution is unique for a short time in that case.

\begin{lem}\label{lem-short-time}
	Let $E_0$ be a precompact open set in $M$ such that $\partial E_0$ is $C^2$ with $H- |h|>0$. Then there exists a classical $C^{2,\alpha }$-solution $\bar{u}$ of (\ref{*-eq-h}) for a short time with initial condition $E_0$.
\end{lem}
\begin{proof}
	When $H- |h|>0$, (\ref{*-eq-h}) is strictly parabolic, so standard parabolic theory applies and we know that a classical $C^{2,\alpha }$-solution exists for a short time. See also \cite[Theorem 7.17]{HP96}.
\end{proof}

\begin{lem}\label{lem-upper-smooth}
	Let $E_0$ be a precompact open set in $M$ such that $\partial E_0$ is $C^2$ with $H- |h|>0$ and $E_0 = E_0'$. Then any weak solution $(E_t)_{0<t<\infty}$ of the $(|h|g)$-IMCF with initial condition $E_0$ is bounded from above by the $C^{2,\alpha }$-classical solution for a short time, provided that $E_t$ remains precompact for a short time. In other words, there exists $t_0>0$ so that $u \leq \bar{u}$ on $\{0 \leq \bar{u} < t_0\} $.
\end{lem}

\begin{proof}
	The proof is a modification of \cite[Theorem 2.2]{HI01}, and is included in Lemma \ref{appen-upper-smooth}.
\end{proof}

Now we prove the existence of a weak solution of $(|h|g)$-IMCF, which is unique and regular for a short time, provided the initial condition is regular enough.

\begin{thm}\label{thm-smooth-uniqueness}
	Let $E_0$ be a precompact open set in $M$ such that $\partial E_0$ is $C^2$ with $H- |h|>0$ and $E_0 = E_0'$. Then there exists a weak solution $(E_t)_{0<t<\infty}$ of the $(|h|g)$-IMCF with initial condition $E_0$, as a limiting of $\epsilon _i$-translating graphs, coincides with the $C^{2,\alpha }$-classical solution for a short time, provided that $E_t$ remains precompact for a short time. In other words, there exists $t_0>0$ so that $u = \bar{u}$ on $\{0 \leq \bar{u} < t_0\} $. In particular, $u$ is $C^{2,\alpha }$ for a short time.
\end{thm}
\begin{proof}
	It's enough to show that $u \geq \bar{u}$ on $\{0 \leq \bar{u}< t_0\} $ for some $t_0>0$. For simplicity, we only consider weak $(h_{\delta }g)$-IMCF, which is sufficient for later use, although the same argument works for weak $(|h|g)$-IMCF. 

	We aim to improve the lower bound of smooth solutions $u ^{\epsilon }$ of \ref{epsilon-eq} in Lemma \ref{smooth-sol-*-epsilon}. Recall that in this case, $\mathcal{E}^{\epsilon }(v) = \mathrm{div}\left( \frac{\nabla v}{\sqrt{|\nabla v|^2 + \epsilon ^2} } \right) - \sqrt{|\nabla v|^2+ \epsilon ^2} - \frac{h_\delta |\nabla v|^2}{|\nabla v|^2 + \epsilon ^2}$, and $u ^{\epsilon }$ is a smooth solution of $\mathcal{E}^{\epsilon }(u^{\epsilon })=0$ in $\Omega _{L}$ satisfying $u ^{\epsilon }=0$ on $\partial E_0$ and $u ^{\epsilon } = L- 2$ on $\partial F_L = \partial \Omega _{L} \setminus \partial E_0$. 

	Assume that $H - h_{\delta } \geq 2c_0>0$ on $\partial E_0$. Let $\bar{u}_{\tau }$ be the smooth solution of
\begin{align*}
	\mathrm{div}\left( \frac{\nabla \bar{u}_{\tau }}{|\nabla \bar{u}_{\tau }|} \right) = |\nabla \bar{u}_{\tau }| + h_{\delta }+ \tau ,\ \ \text{ on } W_{\tau }:= \{0 \leq \bar{u}_{\tau } < t_0\} ,
\end{align*} 
with $\partial E_0 = \{ \bar{u} = 0\} $, where $\tau \in (0, c_0)$ is a fixed small constant but $\tau \gg \epsilon $. By uniform estimates of smooth solutions, we can assume that $t_0>0$ is uniform for all $0<\tau < c_0$, $|\nabla \bar{u}_{\tau }| \geq c_1>0$ uniformly on $W_{\tau }$ and $\bar{u}_{\tau } \to \bar{u}$ smoothly on $W = \{ 0\leq \bar{u}< t_0\} $ as  $\tau \to 0$. We calculate
	\begin{align*}
		\mathcal{E}^{\epsilon }(\bar{u}_{\tau })
		&= \frac{|\nabla \bar{u}_{\tau }| (|\nabla \bar{u}_{\tau }| + h_\delta + \tau  )}{\sqrt{|\nabla \bar{u}_\tau |^2 + \epsilon ^2} } + \frac{\epsilon ^2}{\left( |\nabla \bar{u}_\tau |^2 + \epsilon ^2 \right) ^{\frac{3}{2}}} \frac{\nabla |\nabla \bar{u}_\tau | \cdot \nabla \bar{u}_\tau }{|\nabla \bar{u}_\tau |} \\
		&\ \ - \sqrt{|\nabla \bar{u}_\tau |^2+ \epsilon ^2}  - \frac{h_\delta |\nabla \bar{u}_\tau |^2}{|\nabla \bar{u}_\tau |^2 + \epsilon ^2} \\
		&\geq -C(\|\bar{u}_{\tau }\|_{C^2}) \epsilon ^2 + \frac{|\nabla \bar{u}_\tau | \epsilon ^2 h_{\delta }}{2\left( |\nabla \bar{u}_\tau |^2 + \epsilon ^2 \right) ^{\frac{3}{2}}  } + \tau \cdot \frac{|\nabla \bar{u}_{\tau }|}{\sqrt{|\nabla \bar{u}_\tau |^2 + \epsilon ^2} } \\
		&\geq - C(\|\bar{u}_{\tau }\|_{C^2}) \epsilon ^2 + \frac{1}{2} \tau \\
		&\geq \frac{1}{4} \tau ,
	\end{align*} 
	for all $0<\epsilon < \min(c_1, C(\|\bar{u}_{\tau }\|_{C^2}) \sqrt{\tau } )$. 

Recall that in the proof of Lemma \ref{appen-apriori-estimates}, we constructed a function $v_1(x)$ satisfying $\mathcal{E} ^{\epsilon }(v_1) > 0 $ in the viscosity sense for all small $\epsilon $. Taking $G_0 = W_{\tau } $ in the construction of $v_1$ and defining $v= v_1 + t_0$, we obtain $\mathcal{E}^{\epsilon }(v) >0$ in $\Omega _{L} \setminus W_\tau $ and $v= \bar{u}_{\tau }$ on $\partial W_\tau  $. 
	Now we define a barrier function $w_\tau $ by $w_\tau  = \bar{u}_{\tau }  $ in $W_{\tau }$ and $w_\tau  = v $ in $\Omega _{L} \setminus W_\tau $, which by definition is a viscosity subsolution of $\mathcal{E}^{\epsilon }$ for all small enough $\epsilon $. Since $w_{\tau } = u ^{\epsilon } =0$ on $\partial E_0$, and for all large $L$, $w_\tau  \leq t_0 < L-2 = u ^{\epsilon }$ on $\partial F_L$, it follows by the maximum principle that
	\begin{align*}
		u ^{\epsilon } \geq w_{\tau } \text{ in } \Omega _{L}.
	\end{align*} 
In particular, $u ^{\epsilon } \geq \bar{u}_{\tau }$ in $W$. By our construction of weak solution $u$, as a locally uniformly limit of $u ^{\epsilon _i}$ along $\epsilon _i\to 0, L_i \to 0$, we obtain $u \geq \bar{u}_{\tau }$ in $W$. Taking $\tau \to 0$, we have $u \geq \bar{u}$ in $W$, which completes the proof. 
\end{proof}

\subsection{Topological consequences}
Let $(M^3, g, h)$ be a complete, connected, asymptotically flat triple. We allow $M$ to have a compact $C^2$-boundary consisting of generalized trapped surfaces, that is, surfaces $\Sigma $ satisfying
\begin{align*}
	H_{\Sigma } \leq |h|.
\end{align*} 
Notice that minimal surfaces and traditional apparent horizons are always generalized trapped surfaces. 

Applying \cite[Theorem 5.1]{Eichmair10} to the initial data set $(M^3, g, \frac{1}{2} h\cdot g)$, we obtain the existence and regularity of generalized apparent horizons.
\begin{lem}\label{lem-existence-horizons}
	Let $(M^3, g, h)$ be a complete, connected, asymptotically flat triple, which contains $C^2$-generalized trapped surfaces $\Sigma $. Then there exists an open set $\Omega \subset M $ such that $\Omega $ contains one end of $M$ and $\partial \Omega $ is a closed embedded $C^{2,\alpha }$-outermost generalized apparent horizon, i.e. $H_{\partial \Omega  } = |h|$, and $\Omega $ contains no other compact generalized apparent horizons. If a connected component of $\partial \Omega $ intersects with a component of $\Sigma $, then these components coincide. Moreover, $\partial \Omega $ is strictly area outerminimizing in the sense that every other surface which encloses it in $\bar{\Omega }$ has larger area.
\end{lem}
We denote by $M'$ the metric completion of any connected component of $\Omega $ constructed in above Lemma, and call $M'$ an \textit{exterior region} of $(M^3, g, h)$. The following lemma tells us about the topology of an exterior region.

\begin{lem}\label{lem-topo-R3-balls}
	A $3$-dimensional exterior region $M'$ of $(M^3 ,g, h)$ is simply connected, contains exactly one end of $M$ and its boundary consists of $2$-spheres. That is, $M'$ is diffeomorphic to $\mathbb{R}^3$ minus a finite number of open $3$-balls.
\end{lem}
\begin{proof}
	Using above Lemma, the proof is the same as \cite[Lemma 4.1(ii)]{HI01}. 
\end{proof}

The following Connectedness Lemma implies that a connected surface evolving in an exterior region remains connected.
\begin{lem}\label{lem-connected}
\ 
	\begin{itemize}
		\item [(i)] If $(u, \nu )$ is a weak $J^{\sigma }$-solution on $\Omega $ in the sense of (\ref{weak-J}) and $\|\sigma \|_{C^0}$ is uniformly bounded, then $u$ has no strict local maxima or minima.
		\item [(ii)] Suppose that $M$ is connected and simply connected with no boundary and a single, asymptotically flat end, and $(E_t)_{t>0}$ is a weak solution of the $(\sigma )$-IMCF with initial condition $E_0$. If $\partial E_0$ is connected, then $N_t$ remains connected as long as it stays compact.
	\end{itemize}
\end{lem}
\begin{proof}
	(i) The proof is a modification of \cite[Lemma 4.2]{HI01} and \cite[Lemma 21]{Moore12}.
	First assume that $u$ possesses a strict local maximum or minimum so that there is a connected, precompact component $E \subset \subset \Omega  $ of $\{u>t\} $ or of $\{u<t\} $ for some $t$. In the case of $\{u>t\} $, define the Lipschitz function $v$ by $v=u$ on $\Omega \setminus E$, $v= t$ on $E$. Then (\ref{weak-J}) yields
	\begin{align*}
		\int_{E} \left( |\nabla u| + u\left( |\nabla u| + |\nu |^2_{\sigma } \right)  \right) \leq \int_{E} t \left( |\nabla u| +|\nu |^2_{\sigma } \right).
	\end{align*} 
	This implies that $u=t$ on $E$, a contradiction. In the case of $\{u<t\} $, for $k> \inf_{E} u $ very close to $\inf_{E} u$ and $E_k = \{u< k\} $, define the Lipschitz function $v_k$ by $v_k = u$ on $\Omega \setminus E_k$, $v_k = k$ on $E_k \cap E$. Then (\ref{weak-J}) and H\"older inequality yield
	\begin{align*}
		\int_{E_k} |\nabla u| \left( 1 + u - k \right) \leq \int_{E_k} (k-u) |\nu |^2_{\sigma } \leq C(\|\sigma \|_{C^0}) \|k-u\|_{L^{\frac{n}{n-1}}(E_k)}\cdot |E_k| ^{\frac{1}{n}}.
	\end{align*} 
	Choosing $k$ small enough so that $1+ u- k \geq \frac{1}{2}$ on $E_k$, together with the Sobolev inequality, we obtain 
	\begin{align*}
		\int_{E_k} |\nabla u| \left( 1 + u - k \right) \geq \frac{1}{2} C(n) ^{-1} \| k-u\|_{L^{\frac{n}{n-1}}(E_k)},
	\end{align*} 
	which together with above inequality implies that $1 \leq C \cdot |E_k| ^{\frac{1}{n}}$, a contradiction for small enough $k- \inf_{E} u$.

	Using (i), (ii) follows exactly as in \cite[Lemma 4.2]{HI01}.
\end{proof}

\subsection{Monotonicity formulas for classical solutions}\label{sect-monotone-classical}
In this section, we compute the monotonicity formulas for classical solutions.
Suppose that a classical solution $(N_{t})_{0 \leq t \leq T} \subset M$ of the $(|h|g)$-IMCF (\ref{*-eq-h}) exists. For 
\begin{align*}
	A(t) = e ^{-t}|N_t|,
\end{align*} 
by Lemma \ref{evolution-eq-classical}, 
\begin{align*}
	A'(t) &= e ^{-t} \left( \frac{d}{d t}|N_t| - |N_t| \right)\\
	      &= e ^{-t} \left( \int_{N_t} \frac{H}{H - |h|} - |N_t| \right) \\
	      &= e ^{-t} \int_{N_t} \frac{|h|}{H - |h|}\\
	      & \geq 0,
\end{align*} 
which implies that $A(t)$ is monotone increasing. 

For
\begin{align*}
	B(t) = e ^{\frac{t}{2}} \left( 1- \frac{1}{16 \pi } \int_{N_t} (H - |h|)^2 \right) ,
\end{align*} 
using (\ref{dt-int-H-nu}) and the fact that $\nabla _{\nu } (|h| g)  = (\nabla _{\nu } |h| ) g$ and $ g  (\nu , \nabla ^{N}(H- |h|) ) =0$ a.e., we have
	\begin{align*}
		\begin{split}
			B'(t) &= \frac{1}{2  } e ^{\frac{t}{2}} \left( 1- \frac{1}{16 \pi } \int_{N_t} (H- |h|)^2 \right)  \\
		&\ \ + \frac{e ^{\frac{t}{2}}}{16 \pi }  \int_{N_t} 2  \left(   \frac{|\nabla ^{N} (H-|h| )|^2}{(H-|h| )^2} + \left( \mathrm{Ric}(\nu ,\nu ) +|\mathrm{II}|^2 \right) \right)\\
				&\ \ + \frac{e ^{\frac{t}{2}}}{16 \pi }  \int_{N_t} \left( 2 (\nabla _{\nu } |h| )  -H(H-|h| ) \right) \\
				&=  \frac{e ^{\frac{t}{2}} }{16 \pi  }  \int_{N_t} 2  \left(   \frac{|\nabla ^{N} (H-|h| )|^2}{(H-|h| )^2} + \left( \mathrm{Ric}(\nu ,\nu ) +|\mathrm{II}|^2 \right) \right)\\
				&\ \ + \frac{ e ^{\frac{t}{2}} }{16 \pi } \left( 8 \pi+ \int_{N_t} \left(  - \frac{1}{2}(H-|h| )^2 + 2 (\nabla _{\nu } |h| )    -H(H-|h| ) \right) \right) .
		\end{split}
	\end{align*}
	By the Gauss equation
	\begin{align*}
		R + |\mathrm{II}|^2 + H^2 = R_{N_t} + 2 \left( \mathrm{Ric}(\nu , \nu ) + |\mathrm{II}|^2 \right) ,
	\end{align*} 
	we have 
	\begin{align*}
		\begin{split}
			16 \pi e ^{- \frac{t}{2}} B'(t)
			&=    \int_{N_t} \frac{2 |\nabla ^{N} (H-|h| )|^2}{(H-|h| )^2} + \int_{N_t} \left( R + |\mathrm{II}|^2 + H^2 - R_{N_t} \right) \\
		&\ \ + 8 \pi +  \int_{N_t}   \left(    2 (\nabla _{\nu } |h| )    -\frac{1}{2}(3H-|h| )(H- |h| ) \right)   .
		\end{split}
	\end{align*}

Suppose that $n=3$, $N_t$ remains connected, and $h$ satisfies the dominant energy condition (\ref{DEC-h}).
	Using the fact that $ H- |h|>0$, $|\mathrm{II}|^2 \geq \frac{1}{2} H^2$, and the Gauss-Bonnet formula $\int_{N_t} R_{N_t} = 4 \pi \chi (N_t) \leq 8 \pi $,  we obtain
	\begin{align*}
		\begin{split}
			16 \pi &e ^{- \frac{t}{2}} B'(t)\\
			       &= \int_{N_t} \frac{2 |\nabla ^{N} (H-|h| )|^2}{(H-|h| )^2} + 8 \pi - 4 \pi \chi (N_{t}) + \int_{N_t}\left( |\mathrm{II}|^2 - \frac{1}{2} H^2 \right) \\
			       &\ \ +\int_{N_t}   \left(     R + \frac{3}{2} H^2  + 2 \nabla_{\nu } |h|    -\frac{1}{2}(3H-|h| )(H-|h| ) \right) \\
			&\geq  \int_{N_t}   \left(     R + \frac{3}{2} H^2  - 2 |\nabla h |   -\frac{1}{2}(3H-|h| )(H-|h| ) \right)  \\
					      & =  \int_{N_t}   \left(   R + \frac{3}{2}  |h|^2  - 2 |\nabla h |   + 2 |h| (H-|h| ) \right)  \\
					      &\geq 0,
		\end{split}
	\end{align*}
	which shows that $B(t)$ is also monotone increasing. Thus, in dimension $3$, 
	$$m_{h}(N_t) = \sqrt{\frac{A(t)}{16 \pi }} B(t)$$ 
	is monotone increasing along the classical $(|h| g)$-IMCF, under the connectedness assumption and the dominant energy condition.

\section{Monotonicity formula for weak $(|h| g)$-IMCF}\label{sect-monotonicity}
In this section, we establish the monotonicity formula for the weak $(|h|g)$-IMCF. The strategy is to first derive an $\epsilon $-version for the approximated smooth flows, and then pass to limits. Let's assume that $(M^3 ,g, h)$ is a complete, connected, asymptotically flat triple, $E_0$ is a precompact open set with $C^2$-boundary, and $(\tilde{E}_t)_{t>0} \subset (M \setminus E_0) \times \mathbb{R}$ is a weak solution of the $(|h| g)$-IMCF with initial condition $E_0$. Notice that the following arguments to derive the monotonicity formula work for all such constructed weak solutions.

Let's recall that, as in the proof of Theorem \ref{h-IMCF-existence}, there are subsequences $\epsilon _i\to 0$, $L_i \to \infty$, $\sigma ^{i}= h_{\epsilon _i} g$ with $h_{\epsilon _i} = \sqrt{h^2+ (\epsilon _i \Phi )^2} $, and $U_i(x,z) = u ^{\epsilon _i}(x) - \epsilon _i z$ such that $\tilde{N}^{i}_t= \{ U_i = t\}, -\infty< t< \infty, $ are smooth solutions of the $\sigma ^{i}$-IMCF on $\Omega _i \times \mathbb{R} = \Omega _{L_i} \times \mathbb{R}$. Moreover, $U_i(x,z) \to U(x,z) = u(x)$ locally uniformly, $\tilde{\nu }_i = \frac{\tilde{\nabla } U_i}{|\tilde{\nabla } U_i|} \to \tilde{\nu }$ a.e., and
\begin{align}\label{N_t-C1alpha-conv}
	\tilde{N}^{i}_{t} \to \tilde{N}_{t}= N_{t} \times \mathbb{R} \text{ locally in } C^{1,\alpha },\quad \text{a.e. } t \geq 0,
\end{align}
where $N_t = \partial E_t$, $\tilde{N}_t = \partial \tilde{E}_t$, $E_t = \{u< t\} $ and $\tilde{E}_t = \{ U< t\} $ for $t \geq 0$. 

Take a cutoff function $\phi \in C^{2}_{c}(\mathbb{R})$ such that $\phi \geq 0, \mathrm{supp} \phi \subset [1,5],$ $\phi =1$ on $[2,4]$, $|\phi '| \leq C \phi $, and $\fint \phi (z) dz =1$. Fix an arbitrary $T>0$. By Lemma \ref{smooth-sol-*-epsilon}, there exists $R(T)>0$ such that for all large $i$, depending on $T$,
\begin{align*}
	\tilde{N}^{i}_t \cap (M \times \mathrm{supp} \phi ) \subset K(T):= (\bar{G}_{R(T)} \setminus E_0 ) \times [1,5],\ 0 \leq t \leq T,
\end{align*}
a fixed compact set, where $G_{r}:= \{x \in M \setminus E_0: d(E_0, x) < r \} $. Since $\partial \tilde{N}^{i}_t = \partial E_0 \times \{ - \frac{t}{\epsilon _i}\} \cup \partial \Omega _i \times \{\frac{L_i - 2 - t}{\epsilon _i}\} $, for any $0 \leq t \leq T$, $\partial \tilde{N}^{i}_t \cap (M \times \mathrm{supp} \phi ) = \emptyset$ for all large $i$.

Applying (\ref{dt-int-H-nu}) to the smooth $\sigma ^{i}$-IMCF $\tilde{N}^{i}_t$, we obtain
\begin{align}\label{ddt-identity-i}
	\begin{split}
		&\frac{d}{d t} \int_{\tilde{N}^{i}_t} \phi (\tilde{H}_{i}- |\tilde{\nu}_{i} |^2_{\sigma ^{i} })^2 \\							&= \int_{\tilde{N}^{i}_t} (\tilde{H}_{i}-|\tilde{\nu}_i |^2_{\sigma ^{i} }) \tilde{\nabla} _{\tilde{\nu}_i }\phi -2 \langle \tilde{\nabla} ^{N} \phi , \frac{\tilde{\nabla} ^{N}(\tilde{H}_i-|\tilde{\nu}_i |^2_{\sigma^{i} })}{\tilde{H}_i- |\tilde{\nu}_i |^2_{\sigma ^{i} }} \rangle \\
		&\ \ - 2 \int_{\tilde{N}^{i}_t}\phi \left(   \frac{|\tilde{\nabla} ^{N} (\tilde{H}_i-|\tilde{\nu}_i |^2_{\sigma^{i} })|^2}{(\tilde{H}_i-|\tilde{\nu}_i |^2_{\sigma^{i} })^2} + \left( \tilde{\mathrm{Ric}}(\tilde{\nu}_i ,\tilde{\nu}_i ) +|\tilde{\mathrm{II}}|^2 \right) \right) \\
		&\ \ -2 \int_{\tilde{N}^{i}_t} \phi (\tilde{\nabla} _{\tilde{\nu}_i } \sigma^{i}) (\tilde{\nu}_i , \tilde{\nu}_i ) - 4 \int_{\tilde{N}^i_t} \phi \frac{\sigma ^{i}( \tilde{\nu}_i , \tilde{\nabla} ^{N}(\tilde{H}_i- |\tilde{\nu}_i |^2_{\sigma^{i} }) )}{\tilde{H}_i- |\tilde{\nu}_i |^2_{\sigma ^{i}}} + \int_{\tilde{N}^{i}_t}\phi \tilde{H}_i(\tilde{H}_i-|\tilde{\nu}_i |^2_{\sigma^{i} }).
	\end{split}
\end{align}
For each $0\leq r<s \leq T$, the integrating form is
\begin{align}\label{ddt-integral-identity-i}
	\begin{split}
		& \int_{\tilde{N}^{i}_s} \phi (\tilde{H}_{i}- |\tilde{\nu}_{i} |^2_{\sigma ^{i} })^2 - \int_{\tilde{N}^{i}_r} \phi (\tilde{H}_{i}- |\tilde{\nu}_{i} |^2_{\sigma ^{i} })^2 \\							&= \int_{r}^{s}\int_{\tilde{N}^{i}_t} (\tilde{H}_{i}-|\tilde{\nu}_i |^2_{\sigma ^{i} }) \tilde{\nabla} _{\tilde{\nu}_i }\phi -2 \langle \tilde{\nabla} ^{N} \phi , \frac{\tilde{\nabla} ^{N}(\tilde{H}_i-|\tilde{\nu}_i |^2_{\sigma^{i} })}{\tilde{H}_i- |\tilde{\nu}_i |^2_{\sigma ^{i} }} \rangle \\
		&- 2 \int_{r}^{s}\int_{\tilde{N}^{i}_t}\phi \left(   \frac{|\tilde{\nabla} ^{N} (\tilde{H}_i-|\tilde{\nu}_i |^2_{\sigma^{i} })|^2}{(\tilde{H}_i-|\tilde{\nu}_i |^2_{\sigma^{i} })^2} + \left( \tilde{\mathrm{Ric}}(\tilde{\nu}_i ,\tilde{\nu}_i ) +|\tilde{\mathrm{II}}|^2 \right) \right) \\
		&-2 \int_{r}^{s}\int_{\tilde{N}^{i}_t} \phi (\tilde{\nabla} _{\tilde{\nu}_i } \sigma^{i}) (\tilde{\nu}_i , \tilde{\nu}_i ) - 4 \int_{\tilde{N}^i_t} \phi \frac{\sigma ^{i}( \tilde{\nu}_i , \tilde{\nabla} ^{N}(\tilde{H}_i- |\tilde{\nu}_i |^2_{\sigma^{i} }) )}{\tilde{H}_i- |\tilde{\nu}_i |^2_{\sigma ^{i}}} + \int_{\tilde{N}^{i}_t}\phi \tilde{H}_i(\tilde{H}_i-|\tilde{\nu}_i |^2_{\sigma^{i} })\\
		&=: \mathcal{I}^{i}_{r,s}.
	\end{split}
\end{align}
We wish to pass this identity to limits.

Since the sublevelsets $\tilde{E}^{i}_{t} := \{ U_i < t\}  $ minimizes $J^{\sigma ^{i}}_{U_i, \tilde{\nu }_i}$ from the outside, we can conclude for all $t \in [0, T]$, for all large $i$, that
\begin{align}\label{N_t-area-bound-C(T)}
	| \tilde{N}^{i}_{t} \cap (M \times \mathrm{supp} \phi )| \leq |\partial ^* K(T)| - \int_{K(T) \setminus \tilde{E}^{i}_t} \left( |\tilde{\nabla }U_i|+ \sigma ^{i}(\tilde{\nu }_i, \tilde{\nu }_i) \right) \leq |\partial ^* K(T)| \leq C(T),
\end{align} 
where we have used $K(T) \cup \tilde{E}^{i}_{t}$ as a competitor.

By the gradient estimates in Lemma \ref{smooth-sol-*-epsilon}, since $K(T)$ is compact and $\tilde{H}_i - |\tilde{\nu }_i|^2_{\sigma ^{i}}=|\tilde{\nabla }U_i|$, we have
\begin{align*}
	|\tilde{\mathrm{Ric}}(\tilde{\nu }_i, \tilde{\nu }_i)|+ | \tilde{H}_i -|\tilde{\nu }_i|^2_{\sigma ^{i}} | \leq C(T) \text{ on } \tilde{N}^{i}_t \cap (M \times \mathrm{supp}\phi ),\ 0 \leq t \leq T.
\end{align*} 
It follows that
\begin{align*}
	\int_{\tilde{N}^{i}_t}  &|(\tilde{H}_{i}-|\tilde{\nu}_i |^2_{\sigma ^{i} }) \tilde{\nabla} _{\tilde{\nu}_i }\phi |  + 2\phi |\tilde{\mathrm{Ric}}(\tilde{\nu }_i, \tilde{\nu }_i)|+ 2\phi |(\tilde{\nabla }_{\tilde{\nu }_i}\sigma ^{i})(\tilde{\nu }_i, \tilde{\nu }_i)| \\
 &+ \int_{\tilde{N}^{i}_{t}}  \phi \tilde{H}_i(\tilde{H}_i - |\tilde{\nu }_i|^2_{\sigma ^{i}}) + \phi (\tilde{H}_i - |\tilde{\nu }_i|^2_{\sigma ^{i}}) ^2 \leq C(T, \|\sigma ^{i}\|_{C^1}) = C(T).
\end{align*} 
We estimate
\begin{align*}
	\left| 2\langle \tilde{\nabla} ^{N} \phi , \frac{\tilde{\nabla} ^{N}(\tilde{H}_i-|\tilde{\nu}_i |^2_{\sigma^{i} })}{\tilde{H}_i- |\tilde{\nu}_i |^2_{\sigma ^{i} }} \rangle \right| & \leq 2 \phi ^{-1}|\phi '|^2+ \frac{1}{2} \phi \frac{|\tilde{\nabla} ^{N}(\tilde{H}_i-|\tilde{\nu}_i |^2_{\sigma^{i} })|^2}{(\tilde{H}_i- |\tilde{\nu}_i |^2_{\sigma ^{i} })^2}\\
																							    &\leq C \phi + \frac{1}{2} \phi \frac{|\tilde{\nabla} ^{N}(\tilde{H}_i-|\tilde{\nu}_i |^2_{\sigma^{i} })|^2}{(\tilde{H}_i- |\tilde{\nu}_i |^2_{\sigma ^{i} })^2},\\
	\left|4 \phi \frac{\sigma ^{i}( \tilde{\nu}_i , \tilde{\nabla} ^{N}(\tilde{H}_i- |\tilde{\nu}_i |^2_{\sigma^{i} }) )}{\tilde{H}_i- |\tilde{\nu}_i |^2_{\sigma ^{i}}} \right| &\leq 8 \phi \|\sigma ^{i}\|_{C^0}+ \frac{1}{2} \phi \frac{|\tilde{\nabla} ^{N}(\tilde{H}_i-|\tilde{\nu}_i |^2_{\sigma^{i} })|^2}{(\tilde{H}_i- |\tilde{\nu}_i |^2_{\sigma ^{i} })^2}.
\end{align*} 
Thus for all $t \in [0,T]$,
\begin{align}\label{ddt-C(T)}
	\frac{d}{d t} \int_{\tilde{N}^{i}_t} \phi (\tilde{H}_{i}- |\tilde{\nu}_{i} |^2_{\sigma ^{i} })^2 &\leq -\int_{\tilde{N}^{i}_t}\phi \left( \frac{|\tilde{\nabla} ^{N} (\tilde{H}_i-|\tilde{\nu}_i |^2_{\sigma^{i} })|^2}{(\tilde{H}_i-|\tilde{\nu}_i |^2_{\sigma^{i} })^2} + 2|\tilde{\mathrm{II}}|^2 \right)  + C(T),
\end{align} 
which implies $\int_{\tilde{N}^{i}_{t}} \phi (\tilde{H}_{i}- |\tilde{\nu}_{i} |^2_{\sigma ^{i} })^2 - C(T) t $ is monotone decreasing. Therefore by choosing a subsequence, which is still denoted by $i$, we can assume
\begin{align}\label{N_t-limit-exists}
	\lim_{i\to \infty} \int_{\tilde{N}^{i}_{t}} \phi (\tilde{H}_{i}- |\tilde{\nu}_{i} |^2_{\sigma ^{i} })^2 \text{ exists, } \quad \text{a.e. } t \geq 0.
\end{align}

Integrating (\ref{ddt-C(T)}) gives
\begin{align}\label{integral-C(T)}
	\int_{0}^{T} \int_{\tilde{N}^{i}_t \cap (M \times [2,4])} \frac{|\tilde{\nabla} ^{N} (\tilde{H}_i-|\tilde{\nu}_i |^2_{\sigma^{i} })|^2}{(\tilde{H}_i-|\tilde{\nu}_i |^2_{\sigma^{i} })^2} + |\tilde{\nabla} ^{N} (\tilde{H}_i-|\tilde{\nu}_i |^2_{\sigma^{i} })|^2 + |\tilde{\mathrm{II}}|^2 \leq C(T).
\end{align}  
Applying Fatou's Lemma, we obtain
\begin{align}\label{liminf-estimates}
	\liminf_{i\to \infty} \int_{\tilde{N}^{i}_t \cap (M \times [2,4])} \frac{|\tilde{\nabla} ^{N} (\tilde{H}_i-|\tilde{\nu}_i |^2_{\sigma^{i} })|^2}{(\tilde{H}_i-|\tilde{\nu}_i |^2_{\sigma^{i} })^2} + |\tilde{\nabla} ^{N} (\tilde{H}_i-|\tilde{\nu}_i |^2_{\sigma^{i} })|^2 + |\tilde{\mathrm{II}}|^2 < \infty, \text{ a.e. } t \geq 0.
\end{align} 
Henceforth we shrink $\phi $ so that $\mathrm{supp} \phi \subset [2, 4]$.

As in the proof of Lemma \ref{lem-limit-surface-C1alpha}, if $t$ is not a jump time so that $\tilde{N}_t = \{U=t\} $,  we can write the converging hypersurfaces $\tilde{N}^{i}_t$ simultaneously as graphs of $C^{1,\alpha }_{loc}$-functions $w_i$ over some smooth surface $W $ in the sense that $\tilde{N}^{i}_{t} \cap (M \times [2,4]) = \mathrm{graph}(w_i) \cap (M \times [2,4])$, and $w_i \to w$ locally in $C^{1,\alpha }$, where $\tilde{N}_{t}$ is the graph of $w$. 
By (\ref{liminf-estimates}), along a subsequence $i_j$, $|\tilde{\nabla }U_{i_j}|\circ w_{i_j} = \tilde{H}_{i_j} - |\tilde{\nu }_{i_j}|^2_{\sigma ^{i_j}}$ is uniformly bounded in $W^{1,2}_{loc}(W)$. By Rellich's theorem, there exists a subsequence, still denoted by $i_j$, and a function $f_t$ such that $|\tilde{\nabla }U_{i_j}|\circ w_{i_j} \to f_t$ in $L^2_{loc}(W)$. By (\ref{N_t-limit-exists}), we know this convergence holds for whole sequence $i$. The local $C^{1,\alpha }$-convergence of the hypersurfaces together with the first variation of area formula implies that $H_{\tilde{N}_t}$ exists weakly as a locally $L^{1}$-function, with the weak convergence $\tilde{H}_{i} \tilde{\nu }_{i} \WL H_{\tilde{N}_t} \tilde{\nu }$. Since $|\tilde{\nu }_{i}|^2_{\sigma ^{i}} \to |h|$ uniformly, we have $f_t = H_{\tilde{N}_t} - |h| = |\tilde{\nabla }U|$ a.e., and
\begin{align*}
	\int_{\tilde{N}^{i}_t} \phi (\tilde{H}_{i} - |\tilde{\nu }_{i}|^2_{\sigma ^{i}})^2 &\to \int_{\tilde{N}_t} \phi (H_{\tilde{N}_t} - |h|)^2,\quad \text{a.e. } t \geq 0,\\
	\int_{\tilde{N}^{i}_t} \phi \tilde{H}_i (\tilde{H}_i - |\tilde{\nu }_i|^2_{\sigma ^{i}}) &\to \int_{\tilde{N}_t} \phi H_{\tilde{N}_t}(H_{\tilde{N}_t} - |h|),\quad \text{a.e. } t \geq 0.
\end{align*}

By the bounded convergence theorem, for any $0\leq r< s$,
\begin{align}
	\begin{split}
		\int_{r}^{s}  \int_{\tilde{N}^{i}_t} \phi \tilde{H}_{i}(\tilde{H}_{i} - |\tilde{\nu }_{i}|^2_{\sigma ^{i}}) &\to \int_{r}^{s} \int_{\tilde{N}_t} \phi H_{\tilde{N}_t}(H_{\tilde{N}_t} - |h|).
	\end{split}
\end{align}

Next we estimate
\begin{align*}
	 \int_{\tilde{N}^{i}_{t}} (\tilde{H}_i - |\tilde{\nu }_{i}|^2_{\sigma ^{i}})|\tilde{\nabla }_{\tilde{\nu }_i} \phi | &\leq  C(T) \sup_{\tilde{N}^{i}_t \cap (M \times \mathrm{supp} \phi )}|\langle \tilde{\nu }_i, e_{z} \rangle |  \to 0, \text{ a.e. } 0\leq t \leq T,
\end{align*} 
where we used the fact that for non-jump times $t$, $ \tilde{\nu }_i \to \tilde{\nu } $ uniformly and $\tilde{\nu } $ is perpendicular to $e_z$. By the bounded convergence theorem, we obtain
\begin{align}
	\int_{r}^{s}   \int_{\tilde{N}^{i}_t} (\tilde{H}_i - |\tilde{\nu }_{i}|^2_{\sigma ^{i}})|\tilde{\nabla }_{\tilde{\nu }_i} \phi |  \to 0.
\end{align}
Similarly, for later use, using (\ref{ddt-phi-area}), we compute
\begin{align*}
	\frac{d}{d t} \int_{\tilde{N}^{i}_t} \phi &= \int_{\tilde{N}^{i}_t} \frac{\tilde{\nabla} _{\tilde{\nu}_i } \phi + \phi \tilde{H}_i}{\tilde{H}_i - |\tilde{\nu }_i|^2_{\sigma ^{i}}} \geq \int_{\tilde{N}^{i}_t} \frac{\tilde{\nabla }_{\tilde{\nu }_i} \phi }{\tilde{H}_i - |\tilde{\nu }_i|^2_{\sigma ^i}} + \int_{\tilde{N}^{i}_t} \phi ,
\end{align*}
i.e.
\begin{align*}
	e ^{-s} \int_{\tilde{N}^{i}_s} \phi - e ^{-r} \int_{\tilde{N}^{i}_r} \phi &\geq - \int_{r}^{s}e ^{-t}\int_{\tilde{N}^{i}_t} \frac{|\tilde{\nabla }_{\tilde{\nu }_i} \phi| }{\tilde{H}_i - |\tilde{\nu }_i|^2_{\sigma ^i}} d t \\
										  &= - \int_{\{r \leq U_i \leq s\} } e ^{- U_i} |\tilde{\nabla }_{\tilde{\nu }_i} \phi|,
\end{align*} 
where in the last equality we used the coarea formula. Thus, up to a subsequence, we can assume
\begin{align}\label{N_t-limit-exists}
	\lim_{i\to \infty} e ^{-t}\int_{\tilde{N}^{i}_{t}} \phi = e ^{-t} \int_{\tilde{N}_t} \phi   \quad \text{a.e. } t \geq 0,
\end{align}
and consequently, for a.e. $0\leq r<s$, we have
\begin{align}\label{limit-area-increasing}
	e ^{-r} \int_{\tilde{N}_{r}} \phi \leq e ^{-s}\int_{\tilde{N}_{s}} \phi .
\end{align} 

Next we claim that for each $0 \leq r<s$,
\begin{align*}
	\int_{r}^{s} \int_{\tilde{N}_{t}} \phi \frac{|\tilde{\nabla }^{N} (H_{\tilde{N}_t}- |h|) |^2}{(H_{\tilde{N}_t} - |h| )^2} \leq \liminf_{i\to \infty}  \int_{r}^{s} \int_{\tilde{N}_{t}^{i}} \phi \frac{|\tilde{\nabla }^{N}(\tilde{H}_i - |\tilde{\nu }_i|^2_{\sigma ^{i}})|^2}{ (\tilde{H}_i - |\tilde{\nu }_i|^2_{\sigma ^{i}})^2}.
\end{align*} 
To see this, we first notice that by Lemma \ref{lem-weak-mean-curvature}, for a.e. $t \geq 0$, $\int_{\tilde{N}_{t}} \phi \frac{|\tilde{\nabla }^{N} (H_{\tilde{N}_t}- |h|) |^2}{(H_{\tilde{N}_t} - |h| )^2}$ is well defined. For a.e. $t \geq 0$, by (\ref{liminf-estimates}) and Poincar\'e inequality, and a similar graph argument, there is a subsequence $i_j$ and constants $a_{i_j}$ so that $\log (\tilde{H}_{i_j} - |\tilde{\nu }_{i_j}|^2_{\sigma ^{i_j}}) - a_{i_j}$ is uniformly bounded in $W^{1,2}_{loc}(W)$. By Rellich's theorem, up to a further subsequence, we can assume $a_{i_j} \to a \in [-\infty, \infty)$, $\tilde{H}_{i_j} - |\tilde{\nu }_{i_j}|^2_{\sigma ^{i_j}} \to |\tilde{\nabla }U|$ and $\log (\tilde{H}_{i_j} - |\tilde{\nu }_{i_j}|^2_{\sigma ^{i_j}}) - a_{i_j} \to F_t$ in $L^2_{loc}(W)$ and a.e. on $W$. So $|\tilde{\nabla }U| = \exp(F_t+ a)$ a.e. on $W$. Since $|\tilde{\nabla }U|>0$ a.e. on $W$, we know $a> -\infty$, and
\begin{align*}
	\log (\tilde{H}_{i_j} - |\tilde{\nu }_{i_j}|^2_{\sigma ^{i_j}}) \to \log (H_{\tilde{N}_t} - |h| ) \text{ in } L^2_{loc}(W).
\end{align*} 
It follows that $\tilde{\nabla }^{N}\log (\tilde{H}_{i_j} - |\tilde{\nu }_{i_j}|^2_{\sigma ^{i_j}}) \to \tilde{\nabla }^{N}\log (H_{\tilde{N}_t} - |h| )$ weakly in $L^2_{loc}(W)$. 
By the usual lower semicontinuity and Fatou's Lemma, we obtain the claim.

Next we apply the trick of \cite[Lemma 5.3]{HI01} to estimate the term 
$$\langle \tilde{\nabla} ^{N} \phi , \frac{\tilde{\nabla} ^{N}(\tilde{H}_i-|\tilde{\nu}_i |^2_{\sigma^{i} })}{\tilde{H}_i- |\tilde{\nu}_i |^2_{\sigma ^{i} }} \rangle = \langle \tilde{\nabla} \phi , \frac{\tilde{\nabla} ^{N}(\tilde{H}_i-|\tilde{\nu}_i |^2_{\sigma^{i} })}{\tilde{H}_i- |\tilde{\nu}_i |^2_{\sigma ^{i} }} \rangle .$$
We define
 \begin{align*}
	g_i(t):= - \epsilon _i \int_{\tilde{N}^{i}_t} \phi (z) \langle \tilde{\nabla }^{N}\log (\tilde{H}_{i} - |\tilde{\nu }_{i}|^2_{\sigma ^{i}}), e_z \rangle d A_{\tilde{N}^{i}_t}(x,z).
\end{align*} 
Since $\tilde{N}^{i}_{t} = \mathrm{graph}\left( \epsilon _i ^{-1}(u ^{\epsilon _i} - t) \right) \subset M \times \mathbb{R}$ is a vertical translation of $\tilde{N}^{i}_{0}$ and $\sigma ^{i}$ is translation invariant, we have
\begin{align*}
	g_i(t) &= - \epsilon _i \int_{\tilde{N}^{i}_{0}} \phi (z - \epsilon _i ^{-1} t)  \langle \tilde{\nabla }^{N}\log (\tilde{H}_{i} - |\tilde{\nu }_{i}|^2_{\sigma ^{i}}), e_z \rangle d A_{\tilde{N}^{i}_{0}}(x,z).
\end{align*} 
Then
\begin{align*}
	\frac{d}{d t} g_i(t) &= \int_{\tilde{N}^{i}_{0}} \phi' (z - \epsilon _i ^{-1} t)  \langle \tilde{\nabla }^{N}\log (\tilde{H}_{i} - |\tilde{\nu }_{i}|^2_{\sigma ^{i}}), e_z \rangle d A_{\tilde{N}^{i}_{0}}(x,z) \\
			     &= \int_{\tilde{N}^{i}_t} \langle \tilde{\nabla }^{N}\log (\tilde{H}_{i} - |\tilde{\nu }_{i}|^2_{\sigma ^{i}}), \tilde{\nabla } \phi  \rangle d A_{\tilde{N}^{i}_t}(x,z) =: f_i(t).
\end{align*} 
From our previous uniform estimates (\ref{integral-C(T)}), for any $T>0$, for all large $i$,
\begin{align*}
	\int_{0}^{T} |g_i(t)| dt \leq C(T) \epsilon _i \to 0,\quad \int_{0}^{T}|f_i(t)| dt \leq C(T).
\end{align*} 
Thus $f_i \WL 0$ on $[0, \infty)$ in the sense of measures, which implies that for each $0 \leq r<s$,
\begin{align*}
	\int_{r}^{s} \int_{\tilde{N}^{i}_t} \langle \tilde{\nabla} ^{N} \phi , \frac{\tilde{\nabla} ^{N}(\tilde{H}_i-|\tilde{\nu}_i |^2_{\sigma^{i} })}{\tilde{H}_i- |\tilde{\nu}_i |^2_{\sigma ^{i} }} \rangle \to 0.
\end{align*} 

Next, we consider the term
\begin{align*}
	 \phi \frac{\sigma ^{i}( \tilde{\nu}_{i} , \tilde{\nabla} ^{N}(\tilde{H}_{i}- |\tilde{\nu}_{i} |^2_{\sigma^{i} }) )}{\tilde{H}_{i}- |\tilde{\nu}_{i} |^2_{\sigma ^{i}}} & = \phi  h_{\epsilon _{i}} \langle \tilde{\nu }_{i}, \tilde{\nabla} ^{N}\log (\tilde{H}_{i}-|\tilde{\nu}_{i} |^2_{\sigma^{i} })\rangle_{g}\\
																										       &= -\phi h_{\epsilon _i} \langle \tilde{\nu }_i, e_z \rangle \langle \tilde{\nabla} ^{N}\log (\tilde{H}_{i}-|\tilde{\nu}_{i} |^2_{\sigma^{i} }), e_z \rangle.
\end{align*} 
By (\ref{N_t-area-bound-C(T)}) and Cauchy inequality, 
\begin{align*}
	\int_{\tilde{N}^{i}_{t}} &\left| \phi \frac{\sigma ^{i}( \tilde{\nu}_{i} , \tilde{\nabla} ^{N}(\tilde{H}_{i}- |\tilde{\nu}_{i} |^2_{\sigma^{i} }) )}{\tilde{H}_{i}- |\tilde{\nu}_{i} |^2_{\sigma ^{i}}}\right| \\
	&\leq C(T) \sup_{\tilde{N}^{i}_{t} \cap (M \times \mathrm{supp} \phi )} |\langle \tilde{\nu }_{i}, e_{z} \rangle| \int_{\tilde{N}^{i}_t} \phi  | \tilde{\nabla} ^{N}\log (\tilde{H}_{i}-|\tilde{\nu}_{i} |^2_{\sigma^{i} })|\\
																										      &\leq C(T) \sup_{\tilde{N}^{i}_{t} \cap (M \times \mathrm{supp} \phi )} |\langle \tilde{\nu }_{i}, e_{z} \rangle| \cdot \left(  \int_{\tilde{N}^{i}_t} \phi  | \tilde{\nabla} ^{N}\log (\tilde{H}_{i}-|\tilde{\nu}_{i} |^2_{\sigma^{i} })|^2 \right) ^{\frac{1}{2}}  .
\end{align*}
For each $0\leq r<s$, by Cauchy inequality and (\ref{integral-C(T)}), we obtain
\begin{align*}
	\int_{r}^{s} \int_{\tilde{N}^{i}_{t}} &\left| \phi \frac{\sigma ^{i}( \tilde{\nu}_{i} , \tilde{\nabla} ^{N}(\tilde{H}_{i}- |\tilde{\nu}_{i} |^2_{\sigma^{i} }) )}{\tilde{H}_{i}- |\tilde{\nu}_{i} |^2_{\sigma ^{i}}}\right|\\
						&\leq C(T) \left( \int_{r}^{s}  \sup_{\tilde{N}^{i}_{t} \cap (M \times \mathrm{supp} \phi )} |\langle \tilde{\nu }_{i}, e_{z} \rangle| ^2 \right) ^{\frac{1}{2}}.
\end{align*} 
By the bounded convergence theorem, we have
\begin{align*}
	\lim_{i\to \infty} \int_{r}^{s} \int_{\tilde{N}^{i}_{t}}\phi \frac{\sigma ^{i}( \tilde{\nu}_{i} , \tilde{\nabla} ^{N}(\tilde{H}_{i}- |\tilde{\nu}_{i} |^2_{\sigma^{i} }) )}{\tilde{H}_{i}- |\tilde{\nu}_{i} |^2_{\sigma ^{i}}} =0.
\end{align*} 

Next, we consider the term
\begin{align*}
	\phi (\tilde{\nabla }_{\tilde{\nu }_i} \sigma ^{i}) (\tilde{\nu }_i, \tilde{\nu }_i) &= \phi (\tilde{\nabla }_{\tilde{\nu }_i} h_{\epsilon _i})  g(\tilde{\nu }_i, \tilde{\nu }_i) = \phi (\tilde{\nabla }_{\tilde{\nu }_i} h_{\epsilon _i}) (1- \langle \tilde{\nu }_i, e_z \rangle^2),
\end{align*} 
where we used the fact that $\tilde{g}= g + dz^2$ is the product metric so that $\tilde{\nabla } g =0$.
Since $h_{\epsilon _i}$ is vertical invariant, we have
\begin{align*}
	|\tilde{\nabla }_{\tilde{\nu }_i} h_{\epsilon _i}| \leq  |\nabla h_{\epsilon _i}| \leq |\nabla h| + C\cdot \epsilon _i \leq C. 
\end{align*} 
By the bounded convergence theorem, for each $0\leq r <s$, we obtain
\begin{align*}
	\limsup_{i\to \infty}  2\int_{r}^{s} \int_{\tilde{N}^{i}_t} \phi |(\tilde{\nabla }_{\tilde{\nu }_i} \sigma ^{i}) (\tilde{\nu }_i, \tilde{\nu }_i)| \leq 2 \int_{r}^{s} \int_{\tilde{N}_{t}} \phi |\nabla h|.
\end{align*} 

Next, we consider the curvature term. Using the fact that $\tilde{g} $ is a product metric of $(M,g)$ and $\mathbb{R}$, we can write
\begin{align*}
	\tilde{\mathrm{Ric}}(\tilde{\nu }_i, \tilde{\nu }_i) = \mathrm{Ric}( \tilde{\nu }_i - \langle \tilde{\nu }_i, e_z \rangle e_z, \tilde{\nu }_i - \langle \tilde{\nu }_i, e_z \rangle e_z).
\end{align*}
Similarly, for a.e. $0 \leq t \leq T$, we obtain
\begin{align*}
	\int_{\tilde{N}^{i}_t} \phi \tilde{\mathrm{Ric}}(\tilde{\nu }_i, \tilde{\nu }_i) \to \int_{\tilde{N}_t} \phi \mathrm{Ric}(\tilde{\nu }, \tilde{\nu }),
\end{align*} 
which together with the bounded convergence theorem implies that for each $0\leq r<s$,
\begin{align*}
	\int_{r}^{s }\int_{\tilde{N}^{i}_t} \phi \tilde{\mathrm{Ric}}(\tilde{\nu }_i, \tilde{\nu }_i) \to \int_{r}^{s} \int_{\tilde{N}_t} \phi \mathrm{Ric}(\tilde{\nu }, \tilde{\nu }).
\end{align*}

Finally, we consider the second fundamental form term. 
We first introduce the weak second fundamental form. For a $C^1$-hypersurface $N$ in a smooth ambient Riemannian manifold $\tilde{M}$ with induced metric $g_{N}$, one can define the second fundamental form $\mathrm{II}$ by a locally integrable section $(\mathrm{II}_{ij})$ of $\mathrm{Sym}^2(T^*N)$, which satisfies 
\begin{align*}
	\int_{N} H \nu ^{j} p_{jk} \nu ^{k} = \int_{N} g_{N}^{ij} \nabla _{i} p_{jk} \nu ^{k} + g_{N}^{ij} p_{jk} g_{N}^{kl} \mathrm{II}_{li}, \ \forall (p_{ij}) \in C^{1}_{c}(\mathrm{Sym}^2(T^* \tilde{M}),
\end{align*} 
where $H$ is the weak mean curvature. It can be verified that $|\mathrm{II}| \in L^2_{loc}(N)$ if and only if $N \in W^{2,2}\cap C^1$. Moreover, along $C^{1}_{loc}$-convergence of hypersurfaces, lower semicontinuity of $\|\mathrm{II}\|_{L^2}$ holds. For more details, see \cite[Section 5]{HI01}. In particular, for a.e. $0\leq t \leq T$, $ \mathrm{II}_{\tilde{N}_{t}} $ exists in $L^2(\tilde{N}_{t})$, and by (\ref{liminf-estimates}), there is a subsequence $i_j$ such that
\begin{align*}
	\int_{\tilde{N}_t} \phi |\mathrm{II}_{\tilde{N}_t}|^2 \leq \lim_{i_j \to \infty} \int_{\tilde{N}^{i_j}_{t}} \phi |\tilde{\mathrm{II}}|^2 < \infty.
\end{align*} 
By Fatou's Lemma, for each $0 \leq r<s$,
\begin{align*}
	\int_{r}^{s} \int_{\tilde{N}_t} \phi |\mathrm{II}_{\tilde{N}_t}|^2 \leq \liminf_{i\to \infty} \int_{r}^{s} \int_{\tilde{N}^{i}_t} \phi |\tilde{\mathrm{II}}|^2.
\end{align*} 

Now we can take the limit of the integral (\ref{ddt-integral-identity-i}) and obtain that for a.e. $0\leq r<s$,
\begin{align}
	\begin{split}
		\int_{\tilde{N}_s} \phi (H_{\tilde{N}_s} - |h| )^2 & - \int_{\tilde{N}_r} \phi (H_{\tilde{N}_r} - |h| )^2 \\
								   &\leq - 2 \int_{r}^{s} \int_{\tilde{N}_t} \phi \left( \frac{|\tilde{\nabla }^{N} (H_{\tilde{N}_t} -|h|)|^2}{(H_{\tilde{N}_t} - |h|)^2} + \mathrm{Ric}(\tilde{\nu }, \tilde{\nu }) + |\mathrm{II}_{\tilde{N}_t}|^2 \right) \\
								   &\ \ + 2 \int_{r}^{s} \int_{\tilde{N}_{t}} \phi |\nabla h| + \int_{r}^{s} \int_{\tilde{N}_t} \phi H_{\tilde{N}_t}(H_{\tilde{N}_t} - |h|).
	\end{split}
\end{align}

For a.e. $t \geq 0$, we know that $\tilde{N}_{t} = N_{t} \times \mathbb{R} \subset M \times \mathbb{R}$ is a cylinder, where the geometric quantities are vertical invariant. By the normalization of $\phi $, we obtain
\begin{align}\label{monotone-H-h-1}
	\begin{split}
		\int_{N_s}  (H_{N_s} - |h| )^2 & - \int_{{N}_r}  (H_{{N}_r} - |h| )^2 \\
								   &\leq - 2 \int_{r}^{s} \int_{{N}_t}  \left( \frac{|{\nabla }^{N} (H_{{N}_t} -|h|)|^2}{(H_{{N}_t} - |h|)^2} + \mathrm{Ric}({\nu }, {\nu }) + |\mathrm{II}_{{N}_t}|^2 \right) \\
								   &\ \ + 2 \int_{r}^{s} \int_{{N}_{t}}  |\nabla h| + \int_{r}^{s} \int_{{N}_t}  H_{{N}_t}(H_{{N}_t} - |h|).
	\end{split}
\end{align}

In $3$-dimensional case, we recall the following weak Gauss-Bonnet formula.
\begin{lem}[{\cite[Lemma 5.4]{HI01}}]\label{lem-weak-Gauss-Bonnet}
	Suppose $N$ is a compact $C^1$-surface in a $3$-manifold, satisfying $\int_{N}|\mathrm{II}|^2< \infty$. Then
	\begin{align*}
		\int_{N} K_{12}+ \lambda _1 \lambda _2 = 2 \pi \chi (N),
	\end{align*}
	where $\chi $ is the Euler characteristic, $\lambda _1, \lambda _2$ are eigenvalues of $\mathrm{II}$ w.r.t. the induced metric, and $K_{12}$ is the sectional curvature of the ambient manifold evaluated on $T_xN$.
\end{lem}
Define the weak sectional curvature of surface $N_t$ by $K:= K_{12}+ \lambda _1 \lambda _2$. Applying this lemma to our surfaces $N_t \subset M$, the same proof of {\cite[Lemma 5.5]{HI01}} implies
\begin{lem}\label{lem-W22-Nt}
For each $0\leq t \leq T$, 
\begin{align*}
	\int_{N_t} |\mathrm{II}|^2 \leq C(T).
\end{align*} 
\end{lem}

As a result, we have the following growth formula.

\begin{prop}\label{prop-growth-a.e.}
	Let $(M^3, g, h)$ be a complete, connected asymptotically flat triple without boundary, and $E_0$ a precompact open set with $C^2$-boundary. Let $(E_t)_{t>0} \subset (M \setminus E_0)$ a weak solution of the $(|h|g)$-IMCF with initial condition $E_0$. Then for a.e. $r \geq 0$ and all $s> r$,
	\begin{align}\label{monotone-H-h-2}
	\begin{split}
		\int_{N_s}  &(H_{N_s} - |h| )^2  - \int_{{N}_r}  (H_{{N}_r} - |h| )^2 \\
					       &\leq  \int_{r}^{s} \int_{N_t} \left( 2|\nabla h| -R -\frac{1}{2} H^2_{N_t}- \frac{1}{2}(\lambda _1 - \lambda _2)^2 -H_{N_t} |h| \right) + 4 \pi  \int_{r}^{s} \chi (N_t) \\
					       &\ \ - 2 \int_{r}^{s} \int_{{N}_t}   \frac{|{\nabla }^{N} (H_{{N}_t} -|h|)|^2}{(H_{{N}_t} - |h|)^2} \\
					       &= -\frac{1}{2} \int_{r}^{s} \int_{N_t} \left(  (H_{N_t}-|h|)^2 + (\lambda _1-\lambda _2)^2 \right) + 4 \pi \int_{r}^{s} \chi (N_t)\\
					       &\ \ - \int_{r}^{s} \int_{N_t} \left( R + \frac{3}{2} h ^2 - 2|\nabla h| \right) - 2 \int_{r}^{s} \int_{N_t} |h| (H_{N_t}- |h|)  \\
					       &\ \ - 2 \int_{r}^{s} \int_{{N}_t}   \frac{|{\nabla }^{N} (H_{{N}_t} -|h|)|^2}{(H_{{N}_t} - |h|)^2} .
	\end{split}
\end{align}
\end{prop} 
\begin{proof}
By the Gauss equation, for a.e. $t \geq 0$, we have
\begin{align*}
	2(\mathrm{Ric}(\nu ,\nu ) + |\mathrm{II}|^2) = R + |\mathrm{II}|^2+ H^2 - 2K .
\end{align*} 
Using Lemma \ref{lem-weak-Gauss-Bonnet}, Lemma \ref{lem-W22-Nt} and $|\mathrm{II}_{N_t}|^2 = \frac{1}{2}H_{N_t}^2+ \frac{1}{2}(\lambda _1 - \lambda _2)^2$, for a.e. $0\leq r<s$, (\ref{monotone-H-h-1}) implies the growth formula. As in the proof of Lemma \ref{lem-limit-surface-C1alpha}, for each $s>0$, $N_{t} \to N_{s}$ in $C^{1, \alpha }$ as $t \nearrow s$. Then the lower semicontinuity of $\int_{N_t}(H_{N_t}-|h|)^2$ implies (\ref{monotone-H-h-2}) holds for a.e. $r \geq 0$ and all $s>r$.  
\end{proof}

We wish to extend this proposition to $r=0$. In our setting, to avoid some technical smoothing issues, by revoking the construction of the weak solutions in Theorem \ref{h-IMCF-existence}, we can alternatively construct weak $(h_{\delta }g)$-IMCF with initial condition $E_0$ firstly, and apply the compactness to obtain one weak $(|h|g)$-IMCF. The same proposition implies that (\ref{monotone-H-h-2}) holds for weak $(h_{\delta }g)$-IMCF for a.e. $r \geq 0$ and all $s>r$. Once (\ref{monotone-H-h-2}) is established for weak $(h_{\delta }g)$-IMCF at $r=0$, by the lower semicontinuity of all the quantities involved, together with the fact that
\begin{align*}
	\int_{\partial E_0} (H- h_{\delta })^2 \to \int_{\partial E_0} (H- |h|)^2 \text{ as } \delta \to 0,
\end{align*} 
we obtain (\ref{monotone-H-h-2}) for such constructed weak solution at $r=0$.

For each $\delta >0$, we will modify the proof of \cite[Proposition 5.7]{HI01} to extend (\ref{monotone-H-h-2}) to $r=0$ for one constructed weak $(h_{\delta }g)$-IMCF, by employing more technical smoothing arguments. In the following, the outward optimizing hull will be taken w.r.t. weight $h_{\delta }$ without explicitly indicated. First observe that by (\ref{mean-curv-hull}) and the regularity result, $\partial E_0'$ is $C^{1,1}$ with $H- h_{\delta } \geq 0$ in the weak sense and 
\begin{align*}
	\int_{\partial E_0'}(H- h_{\delta })^2 \leq \int_{\partial E_0} (H- h_{\delta })^2.
\end{align*} 
So it suffices to prove (\ref{monotone-H-h-2}) for $\partial E_0'$.

We need the following technical lemma.
\begin{lem}\label{modified-mcf-smoothing}
Suppose $E$ is precompact, $E' = E$ and $\partial E$ is $C^{1,1}$. Then either $\partial E$ is a smooth generalized apparent horizon in the sense that $H= h_{\delta }$, or $\partial E$ can be approximated in $C^1$ from inside by smooth sets of the form $\partial E_{\tau }$ with $H- h_{\delta }>0$, $E_{\tau }' = E_{\tau }$, and
\begin{align*}
	\sup_{\tau } \sup_{\partial E_{\tau }} |\mathrm{II}| < \infty,\quad \int_{\partial E_{\tau }} (H- h_{\delta })^2 \to \int_{\partial E} (H- h_{\delta })^2 \text{ as } \tau \to 0.
\end{align*} 
\end{lem}
\begin{proof}
	The proof is the same as in \cite[Lemma 5.6]{HI01}, except that we employ the modified mean curvature flow $\frac{\partial x}{\partial \tau } = - (H- h_{\delta }) \nu $. We first mollify $\partial E$ from inside in $C^{1}\cap W^{2,2}$ by a sequence of smooth surfaces $(Q^{i})_{i \geq 1}$ with uniformly bounded $|\mathrm{II}|$. Let $(Q ^{i}_{\tau })_{0 \leq \tau < \tau _i}$ be the modified mean curvature flow of $Q^i$.  We have the following evolution equations along classical modified mean curvature flow: (cf. Section \ref{appen-sub-classical})
	\begin{align*}
		\frac{\partial }{\partial \tau }|\mathrm{II}|^2 &\leq \Delta |\mathrm{II}|^2 + 2 |\mathrm{II}|^4 + C |\mathrm{Rm}| |\mathrm{II}|^2 + C |\nabla \mathrm{Rm}| |\mathrm{II}| + C |h_{\delta }| |\mathrm{II}|^3 + C |\nabla ^2 h_{\delta }| |\mathrm{II}|,\\
		\frac{\partial }{\partial  \tau }(H- h_{\delta }) &=  \Delta ^{N} (H- h_{\delta }) + (H- h_{\delta }) \mathrm{Ric}(\nu ,\nu ) + (H- h_{\delta }) |\mathrm{II}|^2 - (H- h_{\delta }) \nabla _{\nu } h_{\delta }.
	\end{align*} 
	So $Q^{i}_{\tau }$ exists with uniformly bounded $|\mathrm{II}|$ for a uniform time independent of $i$ (depending on metric $g$ and smooth function $h_{\delta }$ ). Passing smoothly to limits we obtain a flow $(Q_{\tau })_{0< \tau \leq \tau _0}$ with uniformly bounded $|\mathrm{II}|$. Using the evolution of mean curvature, we obtain $\int_{Q_\tau } (H- h_{\delta })^2 \leq e ^{C \tau } \int_{\partial E} (H- h_{\delta })^2$, where $C = C(g, \|h_{\delta }\|_{C^1})$. Furthermore, $Q_{\tau } \to \partial E$ in $C^1$. By lower semicontinuity, we have the convergence of $\int_{\partial E_{\tau }}(H-h_{\delta })^2 $. By a variation of this argument and strong maximum principle, either $H- h_{\delta } \equiv 0$ or $H- h_{\delta }>0$ on $Q_{\tau }$. In the latter case, $Q_{\tau }$ smoothly foliates an interior neighborhood of $\partial E$. Using $\nu _{Q_\tau }$ as a calibration, then the level set description is given by $\mathrm{div}(\nu_{Q_\tau } ) = h_{\delta } + \frac{1}{|\nabla v|}$. Let $E_{\tau } = \{ v < - \tau \} $ be the precompact set bounded by $Q_{\tau }$. For any $F$ containing $E_{\tau }$, applying the divergence theorem to $(F \setminus E_{\tau }) \cap \{ v \leq 0\} $, we obtain
	\begin{align*}
		|\partial E_{\tau }| + \int_{F \cap \{-\tau \leq v \leq 0\} } h_{\delta } \leq |F \cap Q_0| + |\partial ^* F \cap \{-\tau < v< 0\}| .
	\end{align*} 
	Using $E' = E$ and $(F \setminus \{-\tau \leq v \leq 0\}) \cup E $ as a competitor, we see that $E_{\tau }' = E_{\tau }$ with weight $h_{\delta }$.
\end{proof}

Case 1. If $H- h_{\delta }>0$ somewhere on $\partial E_0'$, then by Lemma \ref{modified-mcf-smoothing}, there is a family of smooth surfaces of the form $\partial E_{\tau }$ approximating $\partial E_0'$ from inside in $C^1$, with $H- h_{\delta }>0$ and $E_{\tau }' = E_{\tau }$. By Theorem \ref{thm-smooth-uniqueness}, there exists a proper solution $(E_{\tau ,t})_{t>0}$ of the $(h_{\delta }g)$-IMCF, as sublevelsets of $u^{\tau }$, with initial condition $E _{\tau, 0 } = E_{\tau }$, and $u^{\tau }$ is smooth for a short time. So (\ref{monotone-H-h-2}) holds at $r=0$ for $(E_{\tau ,t})_{t \geq 0}$.
Up to a subsequence, taking $\tau_j \to 0 $,  $u ^{\tau _j} \to u$ locally uniformly on $M \setminus E_0'$ for one weak solution $u$, with $C^1$-convergence of the level sets except at countable jump times. By Lemma \ref{modified-mcf-smoothing},
\begin{align*}
	\int_{\partial E_{\tau }} (H- h_{\delta })^2 \to \int_{\partial E_0'} (H- h_\delta )^2.
\end{align*} 
Using the lower semicontinuity of all quantities involved in (\ref{monotone-H-h-2}), we obtain (\ref{monotone-H-h-2}) at $r=0$ for $E_0'$ and hence for $E_0$ in this case.

Case 2. $H- h_\delta \equiv 0$ on $\partial E_0'$. In this case, $\partial E_0'$ is a smooth surface. We can do a conformal perturbation by considering $g_i = w_i g$ for smooth functions $w_i$ satisfying $w_i \to 1$ in $C^{\infty}(M)$, $w_i =1$ and $\frac{\partial w_i}{\partial \nu } >0$ on $\partial E_0'$, and $w_i>1$ in $M \setminus E_0'$. Then $H_i - w_i ^{-\frac{3}{2}} h_\delta  >0$ on $\partial E_0'$ and $E_0'$ is still a strictly outward optimizing hull w.r.t. $g_i$ and weight $w_i ^{-\frac{3}{2}} h_\delta $. This is reduced to Case 1. Therefor (\ref{monotone-H-h-2}) holds at $r=0$ for $g_i$. Similarly, passing $i\to \infty$, using the lower semicontinuity of all quantities involved and the fact that
\begin{align*}
	\int_{\partial E_0'} (H_i - h_\delta )^2 \to \int_{\partial E_0'} (H - h_\delta )^2,
\end{align*} 
we obtain (\ref{monotone-H-h-2}) at $r=0$ for $E_0'$, and hence for $E_0$. 

Thus, we have proved
\begin{thm}\label{thm-growth}
	Let $(M^3, g, h)$ be a complete, connected asymptotically flat triple without boundary, and $E_0$ a precompact open set with $C^2$-boundary. There exists a weak solution of the $(|h|g)$-IMCF with initial condition $E_0$ such that for all $0\leq r< s$,
\begin{align}\label{monotone-H-h}
	\begin{split}
		\int_{N_s}  (H_{N_s} - |h| )^2 & - \int_{{N}_r}  (H_{{N}_r} - |h| )^2 \\ 
					       & \leq -\frac{1}{2} \int_{r}^{s} \int_{N_t} \left(  (H_{N_t}-|h|)^2 + (\lambda _1-\lambda _2)^2 \right) + 4 \pi \int_{r}^{s} \chi (N_t)\\
					       &\ \ - \int_{r}^{s} \int_{N_t} \left( R + \frac{3}{2} h ^2 - 2|\nabla h| \right) - 2 \int_{r}^{s} \int_{N_t} |h| (H_{N_t}- |h|)  \\
					       &\ \ - 2 \int_{r}^{s} \int_{{N}_t}   \frac{|{\nabla }^{N} (H_{{N}_t} -|h|)|^2}{(H_{{N}_t} - |h|)^2} .
	\end{split}
\end{align}
\end{thm}

\begin{rmk}\label{rmk-C1-smoothing}
	Assuming the existence of weak flow, one can weaken the assumption that $E_0$ has $C^2$-boundary to $C^{1}$-boundary satisfying $\int_{\partial E_0} |\mathrm{II}|^2 < \infty$, in which case one can choose smooth mollifying surfaces to approximate $\partial E_0$ in $C^1 \cap W^{2,2}$. See more details in \cite[Section 5]{HI01}.
\end{rmk}

As an application, we derive the monotone property of the generalized Hawking mass:
\begin{align}\label{H-mass}
	m_{h}(N_s) = \sqrt{\frac{|N_s|}{16 \pi }} \left( 1- \frac{1}{16 \pi } \int_{N_s}(H_{N_s}- |h|)^2 \right).
\end{align}

First, for all $0 \leq r<s$,  (\ref{monotone-H-h}) shows that $\frac{d}{d s} \int_{N_s} \left( H_{N_s} - |h| \right) ^2 $ exists a.e. and  for a.e. $s \geq 0$, we have
	\begin{align*}
		\frac{d}{d s} & \int_{N_s} \left( H_{N_s} - |h| \right) ^2 + \frac{1}{2} \int_{N_s} \left( H_{N_s} - |h| \right) ^2 - 4 \pi \chi (N_s) \\
		&\leq -\int_{N_s} \left( R + \frac{3}{2} h ^2 - 2|\nabla h| \right) -2 \int_{N_s} |h| (H_{N_s}- |h|)\\
		&\ \ - \frac{1}{2} \int_{N_s} (\lambda _1 - \lambda _2)^2 - 2  \int_{{N}_s}   \frac{|{\nabla }^{N} (H_{{N}_t} -|h|)|^2}{(H_{{N}_t} - |h|)^2} .
	\end{align*} 
Recall the definition
\begin{align}\label{defn-B(s)}
		B(s)= e ^{\frac{s}{2}}\left( 1- \frac{1}{16 \pi } \int_{N_s} (H_{N_s} - |h| )^2 \right). 
	\end{align}
Then for a.e. $s \geq 0$,
\begin{align}\label{B'(s)-ineq}
	\begin{split}
		B'(s) &\geq \frac{e ^{\frac{s}{2}}}{4} \left( 2- \chi (N_s) \right) + \frac{e ^{\frac{s}{2}}}{16 \pi } \int_{N_s} \left( R + \frac{3}{2} h ^2 - 2|\nabla h| \right) \\
	&\ \ + \frac{e ^{\frac{s}{2}}}{8 \pi } \int_{N_s} |h| (H_{N_s}- |h|)  + \frac{e^{\frac{s}{2}}}{32 \pi } \int_{N_s} (\lambda _1- \lambda _2)^2 + \frac{e ^{\frac{s}{2}}}{8\pi }  \int_{{N}_s}   \frac{|{\nabla }^{N} (H_{{N}_t} -|h|)|^2}{(H_{{N}_t} - |h|)^2} .
	\end{split}
\end{align}

Notice that by taking a limit and using the $C^1$-convergence (\ref{N_t-C1alpha-conv}) , (\ref{limit-area-increasing}) implies that for a.e. $0\leq r< s$,
\begin{align}\label{e-t-Nt-monotone}
	e ^{-r} |N_r| \leq e ^{-s} |N_s|.
\end{align} 
By a similar approximating argument, this inequality holds for a.e. $r \geq 0$ and all $s>r$, and also holds at $N_0 ^{+}$. 

Let's consider the case that $\partial E_0$ is a connected component of the boundary of an exterior region in $M^3$. Then $E_0$ is strictly outward optimizing, so by Lemma \ref{E_t+-vs-E_t'}, $E_0^+ = E_0$, and it follows that $e ^{- \frac{s}{2}} \sqrt{|N_s|} $ is increasing for all $s \geq 0$.

Let's assume further that $h$ satisfies the dominant energy condition (\ref{DEC-h}). By the Connectedness Lemma \ref{lem-connected}, $\chi (N_{s}) \leq 2$.  By Lemma \ref{lem-weak-mean-curvature}, $H_{N_s} - |h| \geq 0$  $\mathcal{H}^2$-a.e. on $N_{s}$ for a.e. $s \geq 0$. Thus the RHS of (\ref{B'(s)-ineq}) is nonnegative for a.e. $s \geq 0$, which implies that $B(s)$ is increasing for all $s \geq 0$. This shows that $m_h(N_s) = e ^{-\frac{s}{2}} \sqrt{|N_s|} B(s)$ is increasing for all $s \geq 0$, and $\lim_{s\to 0} m_{h}(N_s) = m_{h}(N_0) = \sqrt{\frac{|N_0|}{16 \pi }} $. 

As a summary, we have proved the following monotonicity formula.

\begin{thm}[Monotonicity formula]\label{thm-monotonicity-connected}
	Let $(M^3, g, h)$ be a complete, connected, asymptotically flat triple without boundary. Let $E_0$ be a precompact open set with $C^{2}$-boundary, and $\partial E_0$ is a connected component of the exterior region. Then there exists a weak solution of the $(|h| g)$-IMCF with initial condition $E_0$ so that the generalized Hawking mass (\ref{H-mass}) is monotone increasing,	whenever $h$ satisfies the dominant energy condition (\ref{DEC-h}). 
\end{thm}

\section{Proof of the main theorem}\label{sect-proof-main}

\subsection{Asymptotic behavior of weak solutions}
In this subsection, we discuss the asymptotic behavior of weak solutions $u$ of the $(|h| g) $-IMCF. By a weak blowdown argument, we show that $N_{t}$ becomes $C^{1, \alpha }$ close to a large coordinate sphere as $t\to \infty$. We notice that this argument also works for general higher dimensional $\sigma $-IMCF by assuming proper asymptotic behavior of $\sigma $. As a result, we will show that $\lim_{t\to \infty} m_h(N_t) \leq m_{\mathrm{ADM}}(M, g) $ by expanding the Hawking mass as the ADM mass plus lower order terms and using the monotonicity formula for the analytic control. We will mainly follow the framework of \cite[Section 7]{HI01}.

Let $\Omega $ be the asymptotically flat end of $M^3$, embedded in $\mathbb{R}^{3}$ as the complement of a compact set. Let $g$ be the pull back metric on $\Omega $, $\delta $ be the flat metric, $\nabla , \bar{\nabla }$ be the corresponding connections, and $B_r(x), D_r(x)$ be the corresponding geodesic balls. Let $(U, \tilde{\nu })$ or simply $u$ be a weak solution of the $(|h|g)$-IMCF with initial condition $E_0 \subset M \setminus \Omega $. We can regard $u$ as a weak solution on $\Omega $ and $E _{t}:= \{ u < t\} \subset \Omega $.

Fix $\lambda >0$ and define the blow down objects by
\begin{align*}
	\Omega ^{\lambda }:= \lambda \Omega , g ^{\lambda }(y) := \lambda ^2 g(\lambda ^{-1}y), h ^{\lambda }(y):= \lambda ^{-1} h(\lambda ^{-1} y), u ^{\lambda }(y) := u(\lambda ^{-1} y),  E ^{\lambda }_t:= \lambda \cdot E_t.
\end{align*} 
By the scaling property $x \mapsto \lambda x, t \mapsto t$, $u ^{\lambda }$ is a weak solution of the $(|h ^{\lambda }| g ^{\lambda })$-IMCF in $\Omega ^{\lambda }$.

\begin{lem}\label{lem-blowdown}
	Suppose that on $\Omega $, $(M, g, h)$ satisfies the asymptotically flat condition
	\begin{align*}
		|g - \delta | = o(1),\ |\bar{\nabla }g| = o (|x| ^{-1}), \ |h| = o(|x| ^{-1}),\ |\bar{\nabla }h| = o(|x| ^{-2}),
	\end{align*} 
	as $|x| \to \infty$. Let $N_t = \partial \{ u < t\} $ be a weak solution of the $(|h| g)$-IMCF for sufficiently large $t$ so that $\{ u =t\} $ is compact. Then for some constants $c_{\lambda } \to \infty$,
	\begin{align*}
		u ^{\lambda } - c_{\lambda } \to 2 \log |y|,
	\end{align*} 
	locally uniformly in $\mathbb{R}^{3} \setminus \{ 0\} $ as $\lambda \to 0$, the standard expanding sphere solution of the standard IMCF on $\mathbb{R}^{3} \setminus \{0\} $.
\end{lem}
\begin{proof}
	The proof is the same as in \cite[Lemma 7.1]{HI01}, and is included in Section \ref{appen-asymptotic}.  
\end{proof}

We recall the following asymptotically flat conditions: 
\begin{align}\label{AF-condition-recall}
	|g-\delta | \leq C |x| ^{-1}, \ |\bar{\nabla }g| \leq C |x| ^{-2}, \ \mathrm{Ric} \geq - C |x| ^{-2} g,\ |h| + |x|\cdot |\bar{\nabla } h| \leq C |x| ^{-2},
\end{align}
the dominant energy condition
\begin{align}\label{DEC-h-recall}
	R + \frac{3}{2} h ^2 - 2 |\nabla h| \geq 0,
\end{align}
and the definition of ADM mass
\begin{align}
	m_{\mathrm{ADM}}(M, g) = \lim_{U \to M}   \frac{1}{16 \pi } \int_{\partial U} g ^{ij} \left( \bar{\nabla } _j g_{ik} - \bar{\nabla }_{k} g_{ij} \right) \nu ^{k} d \mu ,
\end{align}
where $U$ denotes a precompact open subset with smooth boundary, $\nu $ is the outward unit normal of $\partial U$ with respect to $g$, and $d \mu $ the surface measure of $\partial U$ with respect to $g$.
It's known that this definition is well defined and is independent of the choice of $U \to M$ and $\delta $ (cf. \cite{ADM61, Bartnik86, Chrusciel86}). 

We have the following Asymptotic Comparison Lemma:
\begin{lem}\label{lem-asymp-comp}
	Assume that $(M^3, g, h)$ satisfies (\ref{AF-condition-recall}) and (\ref{DEC-h-recall}), and let $u$ be a weak solution of the $(|h|g)$-IMCF so that $e ^{- \frac{t}{2}} |N_t| ^{\frac{1}{2}}$ and $m_{h}(N_t)$ are monotone increasing for all large $t$. Then
	\begin{align*}
		\lim_{t\to \infty} m_{h}(N_t) \leq m_{\mathrm{ADM}}(M, g). 
	\end{align*} 
\end{lem}

\begin{proof}
	Roughly, by the asymptotic assumptions, $|h|$ will be negligible compared with mean curvature $H$ as $|x| \to \infty$. Together with our monotonicity formula, the same argument of \cite{HI01} works in our setting with $H$ replaced by $H - |h|$. We will follow  the proof of \cite[Lemma 7.4]{HI01}. 

	1. Define $r = r(t)$ by $|N_t| = 4 \pi r ^2$. Then $|N_{t} ^{\frac{1}{r}}|_{g ^{\frac{1}{r}}} = 4 \pi $, so Lemma \ref{lem-blowdown} implies that 
	\begin{align}\label{Nt-rescale-conv}
		N_{t} ^{\frac{1}{r(t)}} \to \partial D_{1} \text{ in } C^{1} \text{ as } t\to \infty.
	\end{align}
	Let $g ^{N}, \delta ^{N}$ be the restriction of $g, \delta $ to the moving surface respectively. Let $\nu $ be the exterior unit normal, $\omega $ the unit dual normal, $\mathrm{II}$ the second fundamental form, $H$ the mean curvature, $d \mu $ the surface measure, all with respect to $g$. Define $\bar{\nu }, \bar{\omega }, \bar{\mathrm{II}}, \bar{H}, d \bar{\mu }$ correspondingly, with respect to $\delta $. Notice that $\omega = \frac{\bar{\omega }}{|\bar{\omega }|_{g}}$ and $\nu ^{i} = g ^{ij} \omega _j$. Write $p_{ij} = g_{ij} - \delta _{ij}$. Then as in \cite{HI01}, for sufficiently large $t$, we have
	\begin{align}\label{H2-lower-infty}
		\begin{split}
			\int_{N_t} H^2 d \mu &\geq 16 \pi + \int_{N_t} \frac{1}{2} H^2 g ^{N, ij} p_{ij} - 2 H g ^{N, ik}p_{ik} g ^{N, lj} \mathrm{II}_{ij} + H^2 \nu ^{i} \nu ^{j} p_{ij}\\
			&\ \ -2 H g ^{N, ij} \nu ^{l} \nabla _{i} p_{jl} + H g ^{N, ij} \nu ^{l} \nabla _{l} p_{ij} - C |p|^2 |\mathrm{II}|^2 - C |\nabla p|^2.
		\end{split}
	\end{align}

2. By Lemma \ref{smooth-sol-*-epsilon}, Remark \ref{rmk-choice-eta} (see also (\ref{gradient-decay})) and (\ref{Nt-rescale-conv}), for a.e. sufficiently large $t$, we have
\begin{align*}
	|H - |h| | = |\nabla u| \leq C |x| ^{-1} \leq C r(t) ^{-1} \text{ on } N_t.
\end{align*} 
Notice that by the assumption (\ref{AF-condition-recall}),
\begin{align}\label{h-3-decay}
	|h| \leq C r(t) ^{-2} \text{ on } N_{t},
\end{align}
so
\begin{align}\label{H|h|-inte-upper}
	\int_{N_t} |H| |h| &\leq C r(t) ^{-1}. 
\end{align} 
Using (\ref{monotone-H-h}), (\ref{AF-condition-recall}) and (\ref{DEC-h-recall}), for a.e. sufficiently large $s$, we obtain
\begin{align*}
	\int_{s}^{s+1} \int_{N_t}|\mathrm{II}|^2 &= \int_{s}^{s+1} \int_{N_t} \frac{1}{2}(H^2+ (\lambda _1 -\lambda _2)^2 )\\
						 &\leq \int_{N_s} (H - |h|)^2 + \int_{s}^{s+1} \int_{N_t}\left( \frac{3}{2} h^2 - H |h| \right)  + 4 \pi \int_{s}^{s+1} \chi (N_t) \\
						 &\leq \int_{N_s} (H - |h|)^2 + \frac{1}{2}\int_{s}^{s+1} \int_{N_t} h^2  + 4 \pi \int_{s}^{s+1} \chi (N_t)\\
						 &\leq 4 \pi C + C r(t) ^{-1} + 8 \pi .
\end{align*} 
So we can select a subsequence $t_i \to \infty$ such that 
\begin{align*}
	\sup_{i} \int_{N_{t_i}}|\mathrm{II}|^2 \leq C.
\end{align*} 
Together with (\ref{H2-lower-infty}) and (\ref{H|h|-inte-upper}), we obtain
\begin{align*}
	\int_{N_{t_i}}(H - |h|)^2 &\geq 16 \pi - C r(t_i) ^{-1},
\end{align*} 
which together with the monotonicity of $m_{h}(N_t)$ implies that
\begin{align*}
	\sup_{t} m_{h}(N_t) \leq C.
\end{align*}
Recall the definition of $B(t)$ in (\ref{defn-B(s)}). Together with (\ref{e-t-Nt-monotone}), this implies that
\begin{align*}
	\sup_{t} B(t) \leq \sup_{t} m_{h}(N_t) \cdot \sup_{t} \sqrt{16 \pi } e ^{\frac{t}{2}} |N_{t}| ^{-\frac{1}{2}} \leq C.
\end{align*} 

3.  Thus, by (\ref{B'(s)-ineq}), we have
\begin{align*}
	\int_{t_0} ^{\infty} e ^{\frac{t}{2}} \int_{N_t}\left(  (\lambda _1 - \lambda _2)^2 + \frac{|\nabla ^{N} (H-|h|)|^2}{(H-|h|)^2} \right) \leq C.
\end{align*} 
Taking a subsequence $t_{\ell} \to \infty$ such that, for $N_{\ell} := N_{t_\ell}$, 
\begin{align}\label{int-N_ell}
	\int_{N_\ell} \left(  (\lambda _1 - \lambda _2)^2 + \frac{|\nabla ^{N} (H-|h|)|^2}{(H-|h|)^2} \right) \leq C e ^{- t_\ell / 2} \to 0.
\end{align} 
So the rescaling $N_{\ell}^{1 / r_{\ell}}$ satisfy 
\begin{align*}
	\int_{N_{\ell}^{1 / r_{\ell}}} |\nabla ^{N} (H -|h|)|^2 \to 0.
\end{align*} 
Using (\ref{Nt-rescale-conv}) and the Rellich's theorem, by writing $N_{\ell} ^{1 / r_{\ell}}$ locally as graphs over $\partial D_1$, we have
\begin{align*}
	H - |h| \to H_{\partial D_1}= 2 \text{ in } L^2(\partial D_1).
\end{align*} 
So we can write
\begin{align}\label{H-|h|-L2}
	H - |h| = \frac{2}{r_{\ell}} + f_{\ell} \text{ on } N_{\ell},
\end{align} 
with $\int_{N_{\ell}} f_{\ell}^2 \to 0$, and by (\ref{int-N_ell}),
\begin{align}\label{II-g-L2}
	\mathrm{II} = \frac{1}{r_{\ell}} g ^{N} + g_{\ell} \text{ on } N_{\ell},
\end{align} 
with $\int_{N_{\ell}} g_{\ell}^2 \to 0$.

4. The remaining computation is the same as in \cite{HI01}. Using (\ref{H2-lower-infty}), (\ref{h-3-decay}), (\ref{H-|h|-L2}) and (\ref{II-g-L2}), we obtain
\begin{align*}
	32 \pi m_{h}(N_\ell) &= r_{\ell} \left( 16 \pi - \int_{N_{\ell}} (H-|h|)^2 \right) \\
			     &\leq C r_{\ell} ^{-1} \left( 1+ \int_{N_{\ell}} (|f_{\ell}| + |g_{\ell}|) \right) + \int_{N_{\ell}} 2 g ^{N, ij} \nu ^{l} \nabla _{i} p_{jl} - 2 g ^{N, ij} \nu ^{l} \nabla _{l} p_{ij}\\
			     &\to 32 \pi m_{\mathrm{ADM}}(M,g).
\end{align*} 
By the monotone increasing property of $m_{h}(N_t)$, this yields
\begin{align*}
	\lim_{t \to \infty} m_{h}(N_{t}) \leq m_{\mathrm{ADM}}(M).
\end{align*}

\end{proof}

\subsection{Multiple horizons and the selection of jump times}\label{sect-multi-horizons}

In this subsection, we describe the selection of jump times in the presence of multiple horizons, using the outward optimizing hulls. The main result is the monotonicity formula for multiple boundary components.

\begin{thm}\label{thm-monotonicity-multiple}
	Let $(M', g, h)$ be a complete, connected asymptotically flat triple so that $M'$ is an exterior region with $C^{2,\alpha }$-outermost generalized apparent horizon. Assume that $h$ satisfies the dominant energy condition. For each connected component $N$ of $\partial M'$, there exists a weak flow of compact $C^{1,\alpha }$-surfaces $(N_{t})_{t \geq 0}$, such that $N_0 = N$, the generalized Hawking mass $m_{h}(N_t)$ is monotone increasing for all time, and for sufficiently large $t$, $N_t$ satisfies the weak $(|h|g)$-IMCF.
\end{thm}

\begin{proof}
	The idea is the same as in \cite[Section 6]{HI01}. We will flow a single boundary component while treating the others as inessential occlusions to be slid across, without incurring any loss of the generalized Hawking mass. 

	Denote by $N, N_1, \ldots, N_{k}$ the boundary components of $M'$. Fill in the boundary by $3$-balls $E_0, W_1, \ldots, W_k$ to obtain a smooth, complete, connected asymptotically flat triple $(M, g, h)$ without boundary. Write $W:= \cup _{1 \leq i \leq k} W_k$. By Lemma \ref{lem-existence-horizons} and Lemma \ref{lem-topo-R3-balls}, $M$ is diffeomorphic to $\mathbb{R}^{3}$, and $E_0 \cup W$ is a strictly outward optimizing hull in $M$. The case when $W$ is empty is covered in Theorem \ref{thm-monotonicity-connected}, so we assume it's nonempty in the following.

	By Theorem \ref{thm-monotonicity-connected}, there exists a weak $(|h|g)$-IMCF $(E_t)_{t \geq 0}$ with initial condition $E_0$ so that $\partial E_t$ remains connected, and $m_h(\partial E_t)$ is monotone increasing whenever the dominant energy condition is satisfied, i.e. $\partial E_t \subset M'$. Define $s_1>0$ to be the supremum of the times when $E_t$ is disjoint from $W$. Then $E_{s_1} \cap W = \emptyset$ and $E_{s_1} ^{+} \cap W \neq \emptyset$. Define $t_1:= s_1$ if $\bar{E}_{s_1} \cap \bar{W} = \emptyset$, otherwise let $t_1$ be slightly less than $s_1$. Notice that $s_1$ is necessarily a jump time when $t_1= s_1$. 

	Now we construct the jump. Let $F$ be the connected component of the strictly outward optimizing hull $(E_{t_1} \cup W)'$ that contains $E_{t_1}$. Notice that $F$ contains at least one component of $W$, so after relabelling we can write $F= (E_{t_1} \cup W_1 \cup \cdots \cup W_j)'$. Since $\partial W$ is $C^{2,\alpha }$-generalized apparent horizon, $\partial F$ is disjoint from $\partial W$ by the strong maximum principle. By Lemma \ref{lem-optimizing-property}, $E_{t_1}$ is outward optimizing, so 
	\begin{align*}
		|\partial E_{t_1}| \leq |\partial F| - \int_{F \setminus E_{t_1}}|h| \leq |\partial F|.
	\end{align*} 
	By a smoothing argument (cf. Remark \ref{rmk-C1-smoothing}), we can choose a sequence of smooth sets $E_i$ such that $\partial E_i \to \partial E_{t_1}$ in $C^1 \cap W ^{2,2}$ and 
\begin{align*}
	\int_{\partial E_i} (H - |h|)^2 \to \int_{\partial E_{t_1}} (H- |h|)^2.
\end{align*} 
Using (\ref{mean-curv-hull}), we have
\begin{align*}
	\int_{\partial E_i'} (H- |h|)^2 \leq \int_{\partial E_i}(H-|h|)^2.
\end{align*} 
Notice that $\partial F \in C^{1, 1}$ and $\partial E_i' \to \partial F$ in $C^1$. By the lower semicontinuity, we have
\begin{align*}
	\int_{\partial F} (H-|h|)^2 \leq \liminf_{i\to \infty} \int_{\partial E_i'} (H-|h|)^2 \leq \int_{\partial E_{t_1}}(H-|h|)^2.  
\end{align*} 
Thus
\begin{align*}
	m_{h}(\partial E_{t_1}) \leq m_{h}(\partial F).
\end{align*} 

Applying \cite[Lemma 6.2]{HI01} to $J^{|h|}$-functional, one can approximate $\partial F$ in $C^1$ by smooth surfaces of uniformly bounded mean curvature. In view of the gradient estimates in Lemma \ref{smooth-sol-*-epsilon} and the Compactness Lemma \ref{lemma-compactness}, one can construct a weak $(|h|g)$-IMCF $(F_{t})_{t \geq 0}$ with initial condition $F_0= F$ as in Theorem \ref{thm-monotonicity-connected} so that $m_{h}(F_t)$ is monotone increasing as long as $\partial F_t \subset M'$. Then replace $E_{t}$ by $F_{t}$ for $t > t_1$. We may continue this procedure inductively to obtain a nested family of sets
 $(E_t)_{t \geq 0}$ so that $E_t$ is a weak $(|h|g)$-IMCF except at a finite number of times, and $m_{h}(\partial E_t)$ is monotone increasing. This completes the proof.
\end{proof}

\subsection{Proof of main theorem}
Let $(M, g, h)$ be an exterior region satisfying the asymptotically flat condition (\ref{AF-condition-recall}) and the dominant energy condition (\ref{DEC-h-recall}), with (possibly empty) outermost generalized apparent horizon.

1. Assume $M$ has no boundary. For any $x \in M$ and each $\delta >0$, by Theorem \ref{h-IMCF-existence}, there exists a weak solution $u ^{\delta }$ of the $(|h|g)$-IMCF with initial condition $B_{\delta }(x)$. Notice that $H- |h| \to \infty$ on $\partial B_{\delta }(x)$ as $\delta \to 0$, and $|h|$ becomes negligible compared with $H$. Similarly as in \cite[Lemma 8.1]{HI01}, using the techniques of Lemma \ref{lem-blowdown}, one can show that there exists a subsequence $\delta _i \to 0$, a sequence $c_i \to \infty$ and a function $u$ defined on $M \setminus \{x\} $ such that $u ^{\delta _i} - c_i \to u$ locally uniformly, $|\nabla u|(y) \leq \frac{C}{d(y,x)}$ in $B_1(x) \setminus \{x\} $, and $N_t = \partial \{ u< t\} $ is nonempty and compact for all $t$, with $N_{t}$ nearly equal to $\partial B_{e ^{t / 2}}$ as $t \to - \infty$. Using the monotone increasing property and the upper semicontinuity of $m_{h}(N ^{i}_{t})$, together with Lemma \ref{lem-asymp-comp} and the fact that $m_{h}(\partial B_{\delta _i}(x) ) \to 0$, we obtain $0 \leq m_{h}(N_t) \leq m_{\mathrm{ADM}}(M)$ for all $t$.

2. If $M$ has a boundary, let $N$ be any connected component of $\partial M$. By Theorem \ref{thm-monotonicity-multiple}, there exists a flow $(N_{t})_{t \geq 0}$ such that $N_0=N$ and $m_{h}(N_t)$ is monotone increasing. Together with Lemma \ref{lem-asymp-comp} this proves the Penrose inequality that 
\begin{align*}
	\sqrt{\frac{|N|}{16 \pi }} = m_{h}(N_0) \leq m_{\mathrm{ADM}}(M).
\end{align*} 

3. To complete the proof of Theorem \ref{thm-main}, it remains to consider the equality case. Whether $M$ has a boundary or not, since $N_t$ solves the weak $(|h|g)$-IMCF except for a finite number of times, in view of (\ref{B'(s)-ineq}), together with the fact that $H-|h|>0$ a.e. on $N_t$ for a.e. $t$, we have $|h| = 0$ on $N_t$ and $\int_{N_t} |\nabla ^{N}(H- |h|)|^2 =0$ for a.e. $t$. By the continuity of $|h|$ and lower semicontinuity of $\int_{N_t} |\nabla ^{N}(H- |h|)|^2$, we have
\begin{align*}
	|h| = 0 \text{ on } N_t,\quad \int_{N_t}|\nabla ^{N}(H- |h|)|^2 =0 \text{ for all } t.
\end{align*}
Thus 
\begin{align*}
	H_{N_t} - |h| = H_{N_t} = H(t), \text{ a.e. } x \in N_{t}, \text{ for all } t,
\end{align*} 
i.e. $N_{t}$ has constant mean curvature. It follows from the standard elliptic theory that $N_{t}$ is smooth for each $t$, and similar considerations apply to $N_{t} ^{+}, t \geq 0$ by using an approximation $N_{t_i} \to N_{t}^{+}$ with $t_i \searrow t$.

If there is a jump at time $t$, either coming from $\mathcal{K}_{t} \subset E_{t} ^+ \setminus E_{t} $ or constructed in Theorem \ref{thm-monotonicity-multiple}, then $H-|h|= 0$ on a portion of $N_t ^+ \setminus N_t $. Since $N_t ^{+}$ has constant mean curvature and $|h|=0$ on $N_t ^{+}$, this implies that $N_t ^{+}$ is a minimal surface, which in particular is a generalized apparent horizon, disjoint from $\partial M$ by the strong maximum principle. This contradicts the assumption that $\partial M$ is the outermost generalized apparent horizon. Therefore there is no jump times and $N_t = N_t ^{+}$ for all $t \geq 0$. In particular, there is at most one boundary component.

So we know that $h \equiv 0$ on $M$, $H_{N_t} = H(t)$ is locally uniformly positive for $t >0$, and the weak $(|h|g)$-IMCF becomes a smooth solution of the standard IMCF. The same argument of \cite[Section 8]{HI01} then applies and shows that each $N_{t}$ is a round sphere and $g$ is isometric to $\mathbb{R}^3$ or the Schwarzschild $3$-manifold. 

\begin{appendix}
\section{Evolving hypersurfaces by the weak $\sigma $-IMCF}\label{appen-weak}

\subsection{Evolution equations and interior estimates of classical $\sigma $-IMCF}\label{appen-sub-classical}
Let $\sigma $ be a smooth  symmetric $2$-tensor on a smooth Riemannian manifold $(M^n, g)$ with dimension $n \geq 3$. Let $F: N ^{n-1} \times [0, T] \to M^{n}$ be a classical solution of the $\sigma $-IMCF, i.e. the following evolution equation is satisfied in the classical sense:
\begin{align}
	\frac{\partial F}{ \partial t}(x,t) = \frac{\nu }{H - |\nu |^2_{\sigma }}(x, t), \quad x \in N, \ 0 \leq t \leq T.
\end{align}
In this section, we will first compute the evolution equations, which proves Lemma \ref{evolution-eq-classical}, and then prove the local estimate of mean curvature in Lemma \ref{local-estimate-mean-curv}.

For any fixed $(x,t) \in N \times (0, T)$, let $\{x ^{\alpha }\}_{\alpha =1}^{n-1}$ be coordinate functions of $N$ in a coordinate chart of $x$ and $\{y^{i} \}_{i=1}^{n}$ be coordinate functions of $M$ in a neighborhood of $F(x, t)$. Set $f:= \frac{1}{H-|\nu |^2_{\sigma }}$, $F^{i}:= y ^{i}\circ F$, and $\partial _{\alpha } := F_{ \sharp} (\frac{\partial }{ \partial x ^{\alpha }})$. We abuse the notation and denote $\frac{\partial F}{\partial t}$ by $\partial _t$. Then the evolution equation is $\partial_ t = f \nu $, the induced metric on $N_t$ is $g_{\alpha \beta } = g(\partial _{\alpha }, \partial _{\beta })$, and the area element is $d A = \sqrt{\mathrm{det} (g_{\alpha \beta })} d x^1 \wedge \cdots \wedge d x ^{n-1 }$. We calculate
	\begin{align*}
		\frac{\partial }{\partial t} g_{\alpha  \beta }& = \langle \nabla _{\partial _t} \partial _{\alpha }, \partial _{\beta } \rangle+ \langle \partial _{\alpha }, \nabla _{\partial _t} \partial _{\beta } \rangle \\
						 &= \langle \nabla _{\partial _\alpha }\partial _t , \partial _{\beta } \rangle+ \langle \partial _{\alpha }, \nabla _{\partial _\beta } \partial _t  \rangle\\
						 &= 2f \mathrm{II}_{\alpha \beta } ,
	\end{align*} 
	where $\mathrm{II}_{\alpha  \beta }= \langle \nabla _{\partial _\alpha } \nu , \partial _{\beta } \rangle$ is the second fundamental form of $N_t$ in $M$. Then 	
	\begin{align*}
		\frac{\partial }{\partial t} d A &= \frac{1}{2} g ^{\alpha \beta } \partial _t g_{\alpha \beta } d A = f H d A.
	\end{align*} 
	This proves that 
	\begin{align}
		\frac{d}{d t}|N_t| &= \int_{N} \frac{\partial }{\partial t} d A = \int_{N} f H d A = \int_{N} \frac{H}{H- |\nu |^2_{\sigma }} dA.
	\end{align}

	We calculate 
	\begin{align*}
		\langle \frac{\partial }{\partial  t} \nu, \partial _{\alpha }  \rangle &= - \langle \nu , \nabla _{\partial _t} \partial _{\alpha } \rangle
									= - \langle \nu , \nabla _{\partial _{\alpha }} \partial _t \rangle 
									= - \langle \nu , \nabla _{\partial _{\alpha }} (f \nu ) \rangle = - \nabla _{\partial _\alpha } f.
	\end{align*} 
	Since $\langle \frac{\partial }{\partial  t} \nu , \nu  \rangle= \frac{1}{2} \frac{\partial }{\partial  t} |\nu |^2 = 0$, we have
	\begin{align}
		\frac{\partial }{\partial  t} \nu = - \nabla ^{N} f = - \nabla ^{N} \frac{1}{H- |\nu |^2_{\sigma }} = \frac{\nabla ^{N}(H-|\nu |^2_\sigma )}{(H- |\nu |^2_\sigma )^2}.
	\end{align}
	So
	\begin{align}
		\begin{split}
			\frac{\partial }{\partial  t}|\nu |^2_{\sigma } &= 2\sigma ( \nu, \frac{\partial }{\partial t} \nu ) +  (\nabla _{\partial _t}\sigma ) (\nu , \nu )\\
							&= -2 \sigma \left(  \nu,  \nabla ^{N} f \right) + f (\nabla _{\nu } \sigma )(\nu ,\nu )\\
						&= \frac{2 \sigma(\nu , \nabla ^{N}(H- |\nu |^2_{\sigma }) )}{(H- |\nu |^2_{\sigma })^2} + \frac{(\nabla _{\nu } \sigma ) (\nu , \nu )}{H- |\nu |^2_{\sigma }}.
		\end{split}
	\end{align} 

We calculate
\begin{align*}
		\frac{\partial }{\partial t} \mathrm{II}_{\alpha \beta }&= \langle \nabla _{\partial _t} \nabla _{\partial _\alpha }\nu , \partial _{\beta } \rangle + \langle \nabla _{\partial _\alpha }\nu , \nabla _{\partial _t} \partial _{\beta } \rangle \\
							 &= \langle \nabla _{\partial _\alpha } \nabla _{\partial _t} \nu , \partial _{\beta } \rangle+ \langle  R(\partial _t, \partial _{\alpha }) \nu, \partial _{\beta } \rangle + f \langle \nabla _{\partial _\alpha }\nu , \nabla _{\partial _\beta }\nu  \rangle\\
							 &= - \langle \nabla _{\partial _\alpha } \nabla ^{N} f, \partial _{\beta } \rangle - f \langle R( \partial _{\alpha }, \nu )\nu , \partial _{\beta } \rangle+ f g ^{\gamma \gamma '} \mathrm{II}_{\alpha \gamma } \mathrm{II}_{\gamma ' \beta }.
	\end{align*} 
	So
	\begin{align}
		\begin{split}
			\frac{\partial }{\partial t} H &= g^{\alpha \beta } \frac{\partial }{\partial t} \mathrm{II}_{\alpha \beta } + (\frac{\partial }{\partial t} g^{\alpha \beta } )\mathrm{II}_{\alpha \beta } \\
				&= - \Delta ^{N} f - f\mathrm{Ric}(\nu , \nu ) - f |\mathrm{II}|^2 \\
				&= \frac{\Delta ^{N} (H-|\nu |^2_\sigma )}{(H-|\nu |^2_\sigma )^2}- \frac{2|\nabla ^{N}(H-|\nu |^2_\sigma )|^2}{(H- |\nu |^2_\sigma )^3} - \frac{\mathrm{Ric}(\nu ,\nu ) + |\mathrm{II}|^2}{H-|\nu |^2_\sigma }.
		\end{split}
	\end{align} 
	Thus
\begin{align}
	\begin{split}
		\frac{\partial }{\partial t}&\frac{(H - |\nu |^2_{\sigma })^2}{2} \\
		&= (H-|\nu |^2_\sigma ) \frac{\partial }{\partial t}(H-|\nu |^2_\sigma )\\
				&= \frac{\Delta ^{N} (H- |\nu |^2_{\sigma })}{H- |\nu |^2_{\sigma }} - \frac{2|\nabla ^{N}(H-|\nu |^2_{\sigma })|^2}{(H- |\nu |^2_{\sigma })^2} - \left( \mathrm{Ric}(\nu ,\nu ) + |\mathrm{II}|^2 \right)\\
				&\ \ - (\nabla _{\nu }\sigma )(\nu ,\nu ) - \frac{2\sigma (  \nu , \nabla ^{N}(H-|\nu |^2_{\sigma }) )}{H-|\nu |^2_{\sigma }}.
	\end{split}
\end{align} 

We prove the following interior mean curvature estimate.
\begin{lem}[Lemma \ref{local-estimate-mean-curv}]\label{appen-local-estimate-mean-curv}
	Let $(N_t)_{0 \leq t \leq T}$ be a classical solution of \ref{*-eq} on $M^{n}$, where $N_t$ may have boundary. Then there exists a positive function $\eta (x) >0$ on $M$, depending on $g$, so that for each $x \in N_t$ and each $r< \eta(x)$, we have 
	\begin{align}
		(H - |\nu |^2_{\sigma })(x, t) \leq \max \left( (H-|\nu |^2_{\sigma })_r, \frac{C(n, r\|\sigma \|_{C^0}, r ^2 \|\nabla \sigma \|_{C^0})}{r} \right) ,
	\end{align}
	where $ (H-|\nu |^2_{\sigma })_r = \sup_{(y,s) \in P_r} (H-|\nu |^2_{\sigma })(y,s)$, and $P_r = (B_r(x) \cap N_0) \times \{0\} \cup \left( \cup _{0 \leq s \leq t}(B_r(x)\cap \partial N_s) \times \{ s\}  \right) $ is the parabolic boundary. 
If $(M^n, g ,\sigma  )$ is asymptotically flat, we may take $\eta(x) \geq c |x| $ for some constant $c>0$, where $x$ is the asymptotically flat coordinate.
\end{lem}
\begin{proof}
	The proof is standard and similar to \cite{HI01, Moore12}. For each $x \in M$, there exists $\eta (x) >0$, depending on $(M,g)$, so that for each $0<r< \eta (x)$, the rescaled metric $r ^{-2}g$ on $B_{10r}(x)$ is close to the Euclidean metric in the $C^2$-sense. In particular, we can ensure that $\mathrm{Ric} \geq - \frac{1}{100 n r ^2}$ on $B_{r}(x)$, and there exists a $C^2$-function $p(y)$, which is a slight modification of $ d^2( x, y)$, satisfying $p(x)=0, p(y) \geq d^2(x,y)$, $|\nabla p| \leq 3 d_x,$ and $\nabla ^2 p \leq 3 g $ on $B_{r}(x)$. 

	In the following, we will fix $x \in N_t$ and consider $0<r< \eta (x)$.
Recall that $f = \frac{1}{H- |\nu |^2_{\sigma }}$. Using the evolution equations, on $N_t \cap B_r(x)$, we calculate
	\begin{align}\label{f-evo-eq}
		\begin{split}
			\frac{\partial }{\partial t} f &= - \frac{1}{(H-|\nu |^2_\sigma )^2} \frac{\partial }{\partial t}(H-|\nu |^2_\sigma )\\
						  &= f^2 \left( \Delta ^{N} f + f \mathrm{Ric}(\nu ,\nu ) + f |\mathrm{II}|^2 + f(\nabla _{\nu }\sigma )(\nu ,\nu ) - 2 \sigma (\nu , \nabla ^{N} f) \right).
		\end{split}	
	\end{align}
	Since
	\begin{align*}
		|\mathrm{II}|^2 \geq \frac{1}{n-1} H^2 &= \frac{1}{n-1}\left( (H-|\nu |^2_\sigma )^2 + 2|\nu |^2_\sigma (H-|\nu |^2_\sigma ) + |\nu |^4_{\sigma } \right) \\
						     &\geq \frac{1}{n-1}\left( (H-|\nu |^2_\sigma )^2 + |\nu |^4_\sigma  \right) ,
	\end{align*} 
	and $|(\nabla _{\nu }\sigma )(\nu ,\nu )| \leq  \|\nabla \sigma \|_{C^0}$, we have
	\begin{align*}
		\frac{\partial }{\partial t} f &\geq f^2 \Delta ^{N} f - \frac{f^3}{100 n r ^2} + \frac{f}{n-1} -  f^3 \|\nabla \sigma \|_{C^0} - 2 f ^2 \sigma(\nu , \nabla ^{N} f). 
	\end{align*} 
	We seek a function $\phi = \phi (y)$ that vanishes on $\partial B_r(x)$ and is a subsolution of (\ref{f-evo-eq}) along $N_t \cap B_r(x)$. Fix $t$ assume $x \in N_t$, and define the parabolic boundary of the flow to be
	\begin{align*}
		P_r = P_r(x,t) := (B_r(x) \cap N_0) \times \{0\} \cup ( \cup _{0 \leq s \leq t} (B_r(x) \cap \partial N_s) \times \{s\} )
	\end{align*} 
	and
	\begin{align*}
		(H-|\nu |^2_\sigma )_r = (H- |\nu |^2_\sigma )_r(x,t):= \sup_{(y,s) \in P_r} (H- |\nu |^2_\sigma )(y,s).
	\end{align*} 

	Define
	$$\phi (y):= \frac{A}{r}(r ^2 - p(y) )_{+} ,\ y \in B_r(x),$$ 
	where $p$ is as defined above. Then $\phi =0$ on $\partial B_r(x)$ and $\phi \leq r$, $\phi \leq A r < \frac{1}{(H-|\nu |^2_\sigma )_r}\leq  f$ in $P_r$, provided $A \leq 1$ and $A < \frac{1}{r(H-|\nu |^2_\sigma )_r}$. In particular, initially we have $\phi < f$ on $N_0 \cap B_r(x)$. Let's assume for now that there exists a first point $(y, s)$ with $0<s \leq t$ and $y \in N_s \cap B_r(x)$ so that $\phi = f$ at $(y,s)$. Then $\nabla ^{N} \phi = \nabla ^{N} f $ at $(y,s)$ and 
	\begin{align*}
		\left( \frac{\partial }{\partial t} - \phi ^2 \Delta ^{N} \right) (f- \phi )  \leq 0 \text{ at } (y,s).
	\end{align*} 

	On the other hand, notice the following relations to the ambient derivatives of $\phi $,
	\begin{align*}
		\frac{\partial }{\partial t} \phi &= f \nabla _{\nu } \phi ,\quad \Delta ^{N} \phi = \mathrm{tr}_{N_t} \nabla ^2 \phi - H \nabla _{\nu } \phi .
	\end{align*} 
	For $y \in N_t \cap B_r(x)$, we have
	\begin{align}
		\begin{split}
			\left( \frac{\partial }{\partial t} - \phi  ^2\Delta ^{N} \right) \phi &= f \nabla _{\nu } \phi + \phi ^2H \nabla _{\nu } \phi - \phi  ^2 \mathrm{tr}_{N_t} \nabla ^2 \phi \\
											   &= - A r ^{-1} (\nabla _{\nu } p) \left( f + f ^{-1} \phi ^2 +\phi ^2 |\nu |^2_{\sigma } \right) + A r ^{-1}\phi ^2 \mathrm{tr}_{N_t} \nabla ^2 p.
		\end{split}
	\end{align}
	So at the point $(y,s)$, we have
	\begin{align*}
		&\left( \frac{\partial }{\partial t} - \phi ^2 \Delta ^{N} \right) (f- \phi )\\
		& \geq - \frac{\phi ^3}{100 n r ^2} + \frac{\phi }{n} -  \phi ^3 \|\nabla \sigma \|_{C^0} - 2 \phi  ^2 \sigma (\nu ,\nabla ^{N} f)\\
											    &\ + A r ^{-1} (\nabla _{\nu } p) \left( 2\phi +\phi ^2 |\nu |^2_{\sigma } \right) - A r ^{-1}\phi ^2 \mathrm{tr}_{N_t} \nabla ^2 p\\
											    &\geq \phi \left( \frac{99}{100n} - A^2r ^2 \|\nabla \sigma \|_{C^0} -2A r \|\sigma \|_{C^0} |\nabla ^{N} \phi | - A r ^{-1}|\nabla p| (2+ A r \|\sigma \|_{C^0})- 3A^2(n-1) \right) \\
											    &\geq \phi \left( \frac{99}{100n} - A^2r ^2 \|\nabla \sigma \|_{C^0} - 4A ^2 r\|\sigma \|_{C^0} - 3(n+1)A^2 \right) \\
											    &>0,
	\end{align*} 
	provided that  $A^2 < \min(1, \frac{1}{40 n r\|\sigma \|_{C^0}+ 10 n  r ^2 \|\nabla \sigma \|_{C^0}}, \frac{1}{30n(n+1)}, \frac{1}{r^2(H-|\nu |^2_{\sigma })_r^2})$, thus a contradiction. In particular, we have proved the conclusion
	\begin{align*}
		f(x,t) \geq  \phi (x, t) = A r .
	\end{align*} 

If $(M^n, g , \sigma )$ is asymptotically flat in the sense that $|x| \|\sigma \| + |x|^2 \|\nabla \sigma \| \leq C$, it's clear that we can take $\eta(x) \geq c |x| $ for some constant $c>0$, where $x$ is the asymptotically flat coordinate.
\end{proof}

\subsection{Existence of elliptic regularisation problem}\label{appen-elliptic-reg}
In this subsection, we prove the existence of a smooth solution of \ref{epsilon-eq}. The argument is similar to those in \cite{HI01, Moore12}. We note that in \cite{Moore12}, the assumption $\mathrm{tr}_g \mathbf{k} \geq 0$ is required in order to construct a subsolution on the bridge region, while constant functions do not provide supersolutions. In contrast, in our setting no such assumption is needed, and constant functions indeed yield supersolutions, as in the standard IMCF. 

For $s \in [0,1]$, consider the family of approimating equations
\begin{align}\label{epsilon-s-eq}
	\begin{cases}
		\mathcal{E}^{\epsilon , s}(u ^{\epsilon ,s} ):= \mathrm{div}\left( \frac{\nabla u ^{\epsilon , s}}{ \sqrt{|\nabla u^{\epsilon, s }|^2 + \epsilon ^2} } \right) - \sqrt{|\nabla u^{\epsilon, s }|^2+ \epsilon ^2} - s\cdot \frac{\sigma^{\epsilon } _{ij}\nabla ^{i} u^{\epsilon, s } \nabla ^{j} u ^{\epsilon ,s}}{|\nabla u^{\epsilon,s }|^2 + \epsilon ^2} =0 &\text{ in } \Omega _{L},\\
		u ^{\epsilon , s}=0 &\text{ on } \partial E_0,\\
		u ^{\epsilon , s}= s(L-2) & \text{ on } \partial F_{L}.
	\end{cases}
	\tag*{$(*)_{\epsilon ,s}$}
\end{align}

Recall that $v$ is a smooth subsolution for $\mathcal{E}^{0, 1}$ on $M \setminus \Omega $, thus a smooth subsolution for $\mathcal{E}^{0, s}$ and all $s \in [0,1]$. Recall $F_L = \{v<L\} $ and $\Omega _{L} = F_L \setminus \bar{E}_0$. We can also assume that $E_0 \subset F_0  $ and $\Omega = \bar{F}_0$.

\begin{lem}\label{appen-apriori-estimates}
	For every $L>0$, there exists $\epsilon (L)>0$ such that for $0< \epsilon < \epsilon (L)$ and $s \in [0,1]$, a smooth solution of \ref{epsilon-s-eq} on $\bar{\Omega }_{L}$ satisfies the following a priori estimates:
	\begin{align}
		u^{\epsilon ,s} \geq - \epsilon \text{ in } \bar{\Omega }_{L},\ u^{\epsilon ,s} \geq v+ s(L-2) -L \text{ in } \bar{F}_{L}\setminus F_0,
	\end{align} 
	\begin{align}
		u^{\epsilon, s } \leq s(L-2) \text{ in } \bar{\Omega }_{L},
	\end{align}
	\begin{align}
		|\nabla u^{\epsilon ,s}| \leq H_{+}  + \epsilon \text{ on } \partial E_0,\ |\nabla u^{\epsilon ,s}| \leq C(L) \text{ on } \partial F_{L},
	\end{align}
	\begin{align}
		|\nabla u^{\epsilon ,s}(x)| \leq \max_{\partial \Omega _{L} \cap B_r(x)}|\nabla u^{\epsilon ,s}| + \epsilon + C(n, \|\sigma^{\epsilon } \|_{C^1}) \cdot r ^{-1}, \ x \in \bar{\Omega }_{L},
	\end{align}
	\begin{align}
		|u^{\epsilon ,s}|_{C^{2, \alpha }(\bar{\Omega }_{L})} \leq C(\epsilon , L, n, \|\sigma ^{\epsilon }\|_{C^1}),
	\end{align}
	for any $0<r\leq \eta (x)$. Here $H_{+}= \max(0, H_{\partial E_0})$.
\end{lem}

\begin{proof}
	Write $u = u ^{\epsilon ,s}$. 

	1. First we construct a subsolution that bridges from $E_0$ to where $v$ starts. Define $G_0 := E_0$ and $G_t := \{x: d(x, E_0) < t\} $. Select $t_L$ so that $G_{t_L} \supset F_L$. Let $\mathrm{Cut}_0$ be the cut locus of $E_0$ in $M$. On $M \setminus (E_0 \cup \mathrm{Cut}_0)$, $d(\cdot , E_0)$ is smooth, and each point is connected to $E_0$ by a unique minimizing geodesic $\gamma $, and $\partial G_t$ foliates a neighborhood of $\gamma $. We have
	\begin{align*}
		\frac{\partial H}{\partial t} &= - |\mathrm{II}|^2 - \mathrm{Ric}(\nu ,\nu ) \leq C_1(L)\quad  \text{ on } \partial G_t \setminus \mathrm{Cut}_0, \ 0 \leq t \leq t_L,
	\end{align*} 
	yielding
	\begin{align*}
		H_{\partial G_t} \leq \max_{\partial E_0} H_{+} + C_1 t \leq C_2(L) \quad  \text{ on } \partial G_t \setminus \mathrm{Cut}_0, \ 0 \leq t \leq t_L,
	\end{align*} 
	where $H_{+} = \max (0, H)$.  Consider the prospective subsolution
	\begin{align*}
		v_1(x) := f(x) = f(d(x, G_0) ),\quad x \in \bar{G}_{t_L}\setminus E_0,
	\end{align*} 
	with $f'<0$. Then $\nabla v_1 = f' \nu , \nabla ^2 v_1 = f'' \nu \otimes \nu + f' \mathrm{II}$, so
	\begin{align*}
		\Delta v_1 - \nabla ^2 v_1(\nu ,\nu ) = f' H_{\partial G_t} \geq - C_2|f'| .
	\end{align*} 
	Hence
	\begin{align*}
		&\sqrt{|f'|^2+ \epsilon ^2} \mathcal{E}^{\epsilon ,s}(v_1) \\
		&= \Delta v_1 -  \left( |f'|^2 + \epsilon ^2 \right) ^{-1} |f'|^2 f'' - |f'|^2- \epsilon ^2 - \left( |f'|^2+ \epsilon ^2 \right) ^{-\frac{1}{2}} s|f'|^2 \sigma ^{\epsilon } (\nu ,\nu ) \\
		&\geq -C_2 |f'| + \frac{\epsilon ^2}{|f'|^2+ \epsilon ^2} f''- |f'|^2- \epsilon ^2 - \|\sigma ^{\epsilon } \|_{C^0}\cdot |f'|.
	\end{align*} 
	Set
	\begin{align*}
		f(t) := \frac{\epsilon }{A}(-1 + e ^{-A t}),\quad 0 \leq t \leq t_L	.
	\end{align*} 
	Then we have $\epsilon ^2 \leq |f'| = \epsilon e^{-At} \leq \epsilon $ and $f''  = A |f'|$, provided that we impose $\epsilon \leq \epsilon (A,L):= e^{-A t_L}$. Choosing $A:= 2(C_2+\|\sigma^{\epsilon } \|_{C^0}+2)$, we have
	\begin{align*}
		(|f'|^2+ \epsilon ^2) \left( C_2|f'| + |f'|^2 + \epsilon ^2 + \|\sigma^{\epsilon } \|_{C^0} |f'| \right) \leq 2 \epsilon ^2 \left( C_2 + \epsilon + 1+ \|\sigma^{\epsilon } \|_{C^0} \right)  |f'| \leq \epsilon ^2 f''. 
	\end{align*} 
	This shows that the function
	$$v_1(x) := \frac{\epsilon }{2(C_2+\|\sigma ^{\epsilon } \|_{C^0} + 2)} \left( -1 + e^{- 2(C_2+\|\sigma ^{\epsilon } \|_{C^0}+2) d(x, \partial G_0)} \right) ,$$ 
	is a smooth subsolution for $\mathcal{E}^{\epsilon ,s}$ on $G_{t_L} \setminus (E_0 \cup \mathrm{Cut}_0)$ for sufficiently small $\epsilon $.

	It can be shown that $v_1$ is a viscosity subsolution of $\mathcal{E}^{\epsilon ,s}$ on all of $G_{t_L}\setminus \bar{E}_0$. Since $u \geq v_1$ on the boundary, it follows by the maximum principle for viscosity solution that 
	\begin{align}\label{v1-lower}
		u \geq v_1 \geq -\epsilon \text{ in } \bar{\Omega }_{L},\ \frac{\partial u}{\partial \nu } \geq - \epsilon \text{ on }\partial E_0.
	\end{align} 

2. Next, consider the function
\begin{align*}
	v_2:= \frac{L-1}{L} v + s(L-2) - (L-1).
\end{align*} 
Then $\mathcal{E}^{0,s}(v_2) = \mathcal{E}^{0,s}(v) + \frac{1}{L}|\nabla v| >0$ on $\bar{F}_{L} \setminus F_0$, where we used the assumption that $\mathcal{E}^{0,s}(v) \geq \mathcal{E}^{0,1}(v) \geq 0$ outside $F_0$. Since the domain is compact, for all sufficiently small $\epsilon $ we obtain $\mathcal{E}^{\epsilon ,s}(v_2) >0$. Note that
\begin{align*}
	u \geq - \epsilon \geq v_2 \text{ on } \partial F_0,\quad u= s(L-2) = v_2 \text{ on } \partial F_L.
\end{align*} 
By the maximum principle, 
\begin{align}\label{v2-lower}
	u \geq v_2 \geq v+ s(L-2) - L \text{ in } \bar{F}_{L}\setminus F_0,\quad \frac{\partial u}{\partial \nu } \geq - C(L) \text{ on } \partial F_L.
\end{align}

Since $\mathcal{E}^{\epsilon ,s}(s(L-2) ) = - \epsilon \leq 0$, we obtain
\begin{align}
	u \leq s(L-2) \text{ in } \bar{\Omega }_{L},\quad \frac{\partial u}{\partial \nu } \leq 0 \text{ on } \partial F_{L}.
\end{align}

3. Next we construct a supersolution along $\partial E_0$. Choose a smooth function $v_3$ which vanishes on $\partial E_0$ such that 
\begin{align*}
	H_{+} < \frac{\partial v_3}{\partial \nu } \leq H_+  + \epsilon \quad \text{ along } \partial E_0.
\end{align*}
Note that along $\partial E_0$, 
\begin{align*}
	\mathcal{E}^{0, s}(v_3) = H_{\partial E_0} - |\nabla v_3| - s \cdot \sigma ^{\epsilon } (\nu ,\nu ) \leq H_{+}   - |\nabla v_3| <0, 
\end{align*}
This implies that for sufficiently small $\delta >0$, $|\nabla v_3|>0$ and $\mathcal{E}^{0,s}(v_3)<0$ in the neighborhood $U:= \{ 0 \leq v_3 \leq \delta \} $. Now define the sped-up function
 \begin{align*}
	v_4 := \frac{v_3}{1- v_3 / \delta },\quad x \in U.
\end{align*} 
Then $v_4 \to \infty$ on $\partial U \setminus \partial E_0$, and $\mathcal{E}^{0,s}(v_4) = \mathcal{E}^{0,s}(v_3) + |\nabla v_3|\left( 1- (1- v_3 / \delta )^{-2} \right) <0 $. For sufficiently small $\epsilon $, depending on $L$, $\mathcal{E}^{\epsilon ,s}(v_4) <0 $ on the set $V:= \{0 \leq v_4 \leq L\} $. Since $u \leq L-2$ in $\bar{\Omega }_{L}$, we have $u \leq v_4$ on $\partial V$. By the maximum principle, we obtain $u \leq v_4$ on $V$, and therefore
\begin{align}\label{v4-upper}
	\frac{\partial u}{\partial \nu } \leq \frac{\partial v_4}{\partial \nu }= \frac{\partial v_3}{\partial \nu } \leq H_+ + \epsilon \quad \text{ on } \partial E_0.
\end{align} 

4. Let $\tilde{N}_t ^{\epsilon ,s}$ denote the level-set $\{U = t\} \subset \Omega _{L} \times \mathbb{R} $ of the function $U(x,z)=U_{\epsilon ,s}(x, z) := u^{\epsilon ,s}(x) - \epsilon z, -\infty< t < \infty$, where $\Omega _{L} \times \mathbb{R}$ is equipped with the product metric $\tilde{g}= g + d z^2$. Equation \ref{epsilon-s-eq} is equivalent to
\begin{align*}
	\mathrm{div}_{\tilde{g}}\left( \frac{\tilde{\nabla } U}{|\tilde{\nabla } U|} \right)  = |\tilde{\nabla }U| + s\cdot  \frac{\sigma _{ij} ^{\epsilon }\tilde{\nabla}^i U \tilde{\nabla} ^j U}{|\tilde{\nabla} U|^2 },
\end{align*}
where we abused the notation $\sigma ^{\epsilon } (x,z) = \sigma ^{\epsilon } (x)$ to denote the constant extension of $\sigma ^{\epsilon } $ along $z$-direction. By the level-set description, $\tilde{N}_t ^{\epsilon ,s}$ is a smooth solution of the $(s \sigma^{\epsilon }) $-IMCF:
\begin{align*}
	\frac{\partial \tilde{F}}{\partial t} = \frac{\tilde{\nu }}{\tilde{H}- |\tilde{\nu }|^2_{s \sigma ^{\epsilon } }},
\end{align*} 
with $\tilde{\nu } = \frac{\tilde{\nabla }U}{|\tilde{\nabla }U|}$ and $\tilde{H} = \mathrm{div}_{\tilde{g}}\left( \frac{\tilde{\nabla } U}{|\tilde{\nabla } U|} \right) $.  
Applying Lemma \ref{appen-local-estimate-mean-curv} to $\tilde{N}^{\epsilon ,s}_t$, we obtain the estimate
\begin{align*}
	 \tilde{H} - |\tilde{\nu }|^2_{s\sigma ^{\epsilon } } \leq \max \left( (\tilde{H}-|\tilde{\nu }|^2_{s\sigma ^{\epsilon } })_{\tilde{P}_r}, \frac{C(n, \|\sigma ^{\epsilon } \|_{C^1})}{r} \right) .
\end{align*}
Notice that $\tilde{H} - |\tilde{\nu }|^2_{s\sigma ^{\epsilon } } = |\tilde{\nabla }U| = \sqrt{|\nabla u |^2+ \epsilon ^2}  $ is independent of $z$. We obtain
\begin{align*}
	\sqrt{|\nabla u |^2+ \epsilon ^2}(x) &\leq \sup_{t} \max_{\partial \tilde{N}^{\epsilon ,s}_{t} \cap \tilde{B}_r(x,z)} \sqrt{|\nabla u |^2+ \epsilon ^2} + \frac{C(n, \|\sigma ^{\epsilon } \|_{C^1})}{r} \\
	&\leq \max_{\partial \Omega _{L} \cap B_r(x)} |\nabla u| + \epsilon + \frac{C}{r},
\end{align*} 
for $x \in M \setminus E_0$ and any $0<r \leq \eta _M(x)$. 

Thus we have the Lipschitz estimate
\begin{align*}
	|u|_{C^{0,1}(\bar{\Omega }_{L})} \leq C(L, n, \|\sigma ^{\epsilon } \|_{C^1}).
\end{align*} 
Following the Nash-Moser-De Giorgi estimates, for some $\alpha =\alpha (\Omega _{L})$, we obtain
\begin{align*}
	\|u\|_{C^{1, \alpha }(\bar{\Omega }_{L})} \leq C(\epsilon ,L,n, \|\sigma ^{\epsilon } \|_{C^1}).
\end{align*} 
The Schauder estimates complete the proof.

\end{proof}

Now we can prove
\begin{lem}[Lemma \ref{smooth-sol-*-epsilon}]\label{appen-smooth-sol-*-epsilon}
	A smooth solution of \ref{epsilon-eq} exists.
\end{lem}
\begin{proof}
	1. We use the method of continuity applied to \ref{epsilon-s-eq}, $0\leq s \leq 1$. First let us prove that there is a solution for $s=0$ and small $\epsilon $. Set $w = u_{\epsilon } / \epsilon $ and rewrite \hyperref[epsilon-s-eq]{$(*)_{\epsilon ,0}$ } as
	\begin{align*}
		\mathcal{F}^{\epsilon }(w) = \mathrm{div}\left( \frac{\nabla w}{\sqrt{|\nabla w|^2 + 1} } \right) - \epsilon \sqrt{|\nabla w|^2 + 1} =0,
	\end{align*} 
	with $w = 0$ on $\partial \Omega _{L}$. The map 
	\begin{align*}
		\mathcal{F}: C^{2, \alpha }_{0}(\bar{\Omega }_{L}) \times \mathbb{R} \to C^{\alpha }(\bar{\Omega }_{L}),
	\end{align*} 
	defined by $\mathcal{F}(w, \epsilon ) := \mathcal{F}^{\epsilon }(w)$ is $C^1$, and possesses the solution $\mathcal{F} ^{0}(0) = 0$. The linearization of $\mathcal{F} ^{0}$ at $w=0$ is given by the ordinary Laplace operator
	\begin{align*}
		D \mathcal{F}^{0}|_{0} = \Delta : C^{2, \alpha }_{0}(\bar{\Omega }_{L}) \to C^{\alpha }(\bar{\Omega }_{L}),
	\end{align*} 
	which is an isomorphism. Then by the Implicit Function Theorem there is a solution of $\mathcal{F}^{\epsilon }(w) =0$ for sufficiently small $\epsilon $, and hence of \hyperref[epsilon-s-eq]{$(*)_{\epsilon ,0}$ }.

	2. Next we fix $\epsilon>0 $ and vary $s$. Let $I$ be the set of $s$ such that \ref{epsilon-s-eq} possesses a solution. Now $0 \in I$ by Step 1, and taking $\epsilon < \epsilon (L)$ if necessary, $I \cap [0, 1]$ is closed by above a priori estimates. Let us prove that $I$ is open. Let $\pi (u) = u|_{\partial \Omega_{L} }$ be the boundary values map, and define 
	\begin{align*}
		\mathcal{G}: C^{2,\alpha }(\bar{\Omega }_{L}) \times \mathbb{R}\to C^{\alpha }(\bar{\Omega }_{L}) \times C^{2, \alpha }(\partial \Omega _{L})
	\end{align*} 
	by
	$\mathcal{G}(u, s) = \mathcal{G}^{s}(u) := (\mathcal{E}^{\epsilon ,s}(u), \pi (u) - s(L-2) \chi _{\partial F_L}),$
	so that \ref{epsilon-s-eq} is equivalent to $\mathcal{G}^{s}(u) = (0, 0)$. Clearly $\mathcal{G}$ is $C^1$. The linearization of $\mathcal{G}^{s}$ at a solution $u$ is given by
	\begin{align*}
		D \mathcal{G}^{s}|_u = \begin{pmatrix} D \mathcal{E}^{\epsilon ,s}|_u \\ \pi  \end{pmatrix} : C^{2,\alpha }(\bar{\Omega }_{L}) \to C^{\alpha }(\bar{\Omega }_{L}) \times C^{2,\alpha }(\partial \Omega _{L}),
	\end{align*} 
	where 
	$
		D \mathcal{E}^{\epsilon ,s}|_u(v) = \mathrm{div}\left( \frac{\nabla v}{\sqrt{|\nabla u|^2+ \epsilon ^2} } - \frac{\langle \nabla u, \nabla v \rangle}{(|\nabla u|^2+ \epsilon ^2) ^{\frac{3}{2}}} \nabla u \right) - \frac{\langle \nabla u, \nabla v \rangle}{\sqrt{|\nabla u|^2+ \epsilon ^2} } - 2 s \frac{\sigma ^{\epsilon } (\nabla u, \nabla v)}{|\nabla u|^2+ \epsilon ^2} + 2s \frac{\sigma ^{\epsilon } (\nabla u, \nabla u) }{(|\nabla u|^2+ \epsilon ^2)^2} \langle \nabla u, \nabla v \rangle = \frac{1}{\sqrt{|\nabla u|^2+ \epsilon ^2} }\left( \Delta v - \frac{\nabla ^2v(\nabla u, \nabla u)}{|\nabla u|^2 + \epsilon ^2} \right)+ B^i \nabla _{i} v = A^{ij} \nabla _{i } \nabla _j v + B^i \nabla_i v,
		$ and $A^{ij}= \frac{1}{\sqrt{|\nabla u |^2+ \epsilon ^2} }\left( g ^{ij} - \frac{\nabla ^{i} u \nabla ^{j} u}{|\nabla u|^2+ \epsilon ^2} \right) $ is uniformly elliptic and H\"older continuous. So we can apply Schauder theory to deduce that $D \mathcal{G}^{s}|_u$ is an isomorphism. Then by the Implicit Function Theorem, $I$ is open. Therefore $1 \in I$, which proves existence of $u^{\epsilon} \in C^{2,\alpha }(\bar{\Omega }_{L})$ solving \ref{epsilon-eq}. Smoothness follows by Schauder estimates. 
\end{proof}

\subsection{Limit of regularisation solutions}\label{appen-limit-regularized}
In this section, we prove those technical lemmas regarding the regularity of the limit of $\epsilon $-translating graphs. Consider the setting in Section \ref{section-limit-translating}: $u_i(x)$ is a sequence of smooth solutions of \hyperref[epsilon-eq]{$(*)_{\epsilon_i}$} with locally uniformly Lipschitz bounds on $\Omega _{L_i}$ for $\epsilon _i\to 0$ and $L_i\to \infty$, and $u_i \to u$ locally uniformly on $M \setminus E_0$ for a locally Lipschitz function $u$ satisfying $u \geq 0$ and $u(x) \to \infty$ as $d(x, E_0) \to \infty$;   $U_i(x,z) := u_i(x) - \epsilon _i z$ is a sequence of smooth functions so that  $\tilde{N}^{i}_t:= \{U_i < t\} $ are smooth solutions of the $\sigma ^{i} $-IMCF in $\Omega _{L_i} \times \mathbb{R}$, and $U_i \to U$ locally uniformly on $(M\setminus E_0) \times \mathbb{R}$ for $U(x,z) = u(x)$, $\sigma ^{i}$ has uniform $C^{1}$-bound and $\sigma ^{i} \to \sigma $ uniformly for a Lipschitz nonnegative symmetric $2$-tensor $\sigma $.

We first recall the following $C^{1,\alpha }$-regularity lemma:
\begin{lem}[Lemma \ref{lem-local-C1alpha}]
For each $t$, the level sets $\tilde{N}^{i}_{t}$ are locally uniformly bounded in $C^{1,\alpha }$, independent of $i$ and $t$.	
\end{lem}

Now we prove
\begin{lem}[Lemma \ref{lem-limit-surface-C1alpha}]
Fix $t_0\geq 0$. 
	\begin{itemize}
		\item[(1)] If $t_0$ is not a jump time, then $\tilde{N}_{t_0} = \{U =t_0\} $ is a complete hypersurface that is locally uniformly bounded in $C^{1,\alpha }$. 
		\item[(2)] If $t_0$ is a jump time, then $\tilde{N}_{t_0}, \tilde{N}^{+}_{t_0}$ are complete hypersurfaces that are locally uniformly bounded in $C^{1,\alpha }$. 
		\item[(3)] If $t_0$ is a jump time, let $\tilde{\mathcal{K}}_{t_0}$ denote the interior of $\{U= t_0\} $, then each point $\tilde{x}_0 \in \tilde{\mathcal{K}}_{t_0}$ lies in a complete hypersurface $\tilde{N}_{\tilde{x}_0} \subset \{ U=t_0\} $ that is locally uniformly $C^{1,\alpha }$-bounded.
	\end{itemize}

\end{lem}
\begin{proof}
	The proof is the same as in \cite[Lemma 7.4]{HuiskenWolff22} and \cite[Proposition 8]{Moore12}, and is based on a local graph representation of $\tilde{N}^{i}_t$.

	1. Assume $\tilde{N}_{t_0} = \tilde{N}^{+}_{t_0}$. Fix a point $\tilde{x}_0:= (x_0, z_0) \in \tilde{N}_{t_0}$. Since $(\tilde{N}^{i}_t)_{-\infty<t<\infty}$ foliate $\Omega _{L_i} \times \mathbb{R}$, for all large enough $i$,  there exists a unique $t_i $ so that $\tilde{x}_0 \in \tilde{N}^{i}_{t_i}$. Notice that $t_i \to t_0$, since $U_i \to U$ locally uniformly. Let $\mathrm{inj} (\tilde{x}_0)$ be the injectivity radius of $\tilde{x}_0$ in $(M\setminus E_0) \times \mathbb{R}$ and $d=\min(\mathrm{inj} (\tilde{x}_0), \eta (\tilde{x}_0) )$ with $\eta $ the function given by Lemma \ref{local-estimate-mean-curv}. By above regularity Lemma, the hypersurface pieces $\tilde{N}^{i}_{t_i} \cap \tilde{B}_d(\tilde{x}_0)$ are $C^{1,\alpha }$-bounded uniformly in $t_i$ and $i$. Now consider the exponential map
	\begin{align*}
		\exp_{\tilde{x}_0}= (\exp_{x_0}, \mathrm{id}_{\mathbb{R}}): T_{\tilde{x}_0}(M \times \mathbb{R}) \cap B^{n+1}_{d}(0, z_0) \to \tilde{B}_{d}(\tilde{x}_0),
	\end{align*} 
	and set 
	\begin{align*}
		\hat{N}^{i}_{t_i} := \exp^{-1}_{\tilde{x}_0} \left( \tilde{N}^{i}_{t_i} \cap \tilde{B}_d(\tilde{x}_0) \right) \subset T_{\tilde{x}_0}(M \times \mathbb{R}) \simeq \mathbb{R}^{n+1}.
	\end{align*} 
	Then the hypersurfaces $\hat{N}^{i}_{t_i}$ are $C^{1,\alpha }$-bounded uniformly in $t_i$ and $i$. In particular, we have uniform $C^{\alpha }$-bounds on the unit normal $\hat{\nu }_i$ to  $\hat{N}^{i}_{t_i}$. Set $\hat{x}_0 = (0, z_0) = \exp ^{-1}_{\tilde{x}_0}(\tilde{x}_0)$. Up to a subsequence, $\hat{\nu }_i(\hat{x}_0) \to \hat{\nu }(\hat{x}_0)$ uniformly for some vector $\hat{\nu }(\hat{x}_0)$, and there exists a uniform $r>0$ so that we can write $\hat{N}^{i}_{t_i} $ locally as graphs of $C^{1,\alpha }$-functions $\hat{w}_i$ over $\hat{T} \cap B^{n+1}_r(\hat{x}_0)$, i.e. $\hat{N}^{i}_{t_i} \cap B^{n+1}_r(\hat{x}_0) = \mathrm{graph}(\hat{w}_i)$, where $\hat{T}$ is the hyperplane normal to $\hat{\nu }(\hat{x}_0)$ and contains $\hat{x}_0$. Up to a further subsequence and for a smaller $\alpha >0$, there exists a $C^{1,\alpha }$-function $\hat{w}$ so that
	\begin{align*}
		\hat{w}_i \to \hat{w} \text{ in } C^{1,\alpha }(\hat{T}\cap B^{n+1}_r(\hat{x}_0) ).
	\end{align*} 
	Define $\hat{N}_{\hat{x}_0}:= \mathrm{graph}(\hat{w})$ in $B^{n+1}_r(\hat{x}_0)$. Then $\hat{T}= T_{\hat{x}_0} \hat{N}_{\hat{x}_0}$, the hypersurface $\tilde{N}_{\tilde{x}_0}:= \exp_{\tilde{x}_0}(\hat{N}_{\hat{x}_0})\subset M \times \mathbb{R}$ is uniformly $C^{1,\alpha }$-bounded, and $\tilde{N}^{i}_{t_i} \to \tilde{N}_{\tilde{x}_0} $ in a small neighborhood of $\tilde{x}_0$. Since $t_i \to t_0$, $\tilde{N}_{\tilde{x}_0} \subset \{ U = t_0\} = \tilde{N}_{t_0} $. Since $t_0$ is not a jump time, $ \tilde{N}_{\tilde{x}_0} = \tilde{N}_{t_0}$ in a small neighborhood of $\tilde{x}_0$. This proves that $\tilde{N}_{t_0}$ is locally uniformly bounded in $C^{1,\alpha }$. Moreover, the convergence is independent of the choice of subsequence. In particular, the unit normal vector field to $\tilde{N}_{t_0}$ is well defined and is locally uniformly bounded in $C^{\alpha }$.

	2. Assume $\tilde{N}_{t_0} \neq \tilde{N}^{+}_{t_0}$. Notice that there are only countable many such $t_0$. Fix a $\tilde{x}_0^+ \in \tilde{N}_{t_0}^{+}$. We can choose a sequence of $t_j > t_0$ and points $\tilde{x}_j \in \tilde{N}_{t_j}$ so that $t_j \to t_0$, $\tilde{x}_j \to \tilde{x}_0^+$ and $\tilde{N}_{t_j}= \tilde{N}^{+}_{t_j}$. By previous discussion,  $\tilde{N}_{t_j} \cap B^{n+1}_r(\tilde{x}_j) $ are locally uniformly $C^{1,\alpha }$-bounded, and the unit normal vectors $\nu _j$ to $\tilde{N}_{t_j}$ at $\tilde{x}_j$ are uniformly $C^{\alpha }$-bounded. Assume that up to a subsequence, $\nu _j(\tilde{x}_j) \to \nu \in T_{\tilde{x}_0^+} (M\times \mathbb{R})$ uniformly. Via the exponential map at $\tilde{x}_0^+$, let $\hat{T}$ be the hyperplane orthogonal to $\hat{\nu }$. Similarly, we use $\hat{\cdot }$ to denote the preimage under the exponential map. By taking smaller $r>0$, we can write $\hat{N}_{t_j} \cap B^{n+1}_r(\hat{x}_0^+)$ as the graph of a $C^{1,\alpha }$-function $\hat{w}_j$ over $\hat{T}\cap B^{n+1}_r(\hat{x}_0^+)$. Up to a subsequence, $\hat{w}_j \to \hat{w}$ in $C^{1,\alpha }(\hat{T}\cap B^{n+1}_r(\hat{x}_0^+) )$. Define $\hat{N}_{\hat{x}_0^+}:= \mathrm{graph}(\hat{w})$ in $B^{n+1}_r(\hat{x}_0^+)$. Then $\hat{T}= T_{\hat{x}_0^+} \hat{N}_{\hat{x}_0^+}$, $\tilde{N}_{\tilde{x}_0^+}:= \exp_{\tilde{x}_0^+}(\hat{N}_{\hat{x}_0^+})$ is uniformly $C^{1,\alpha }$-bounded and $\tilde{N}_{t_j} \to \tilde{N}_{\tilde{x}_0^+}$ in a small neighborhood of $\tilde{x}_0^+$. Since $t_j \to t_0$, we know $\tilde{N}_{\tilde{x}_0^+} \subset \{U= t_0\} $. It remains to show that $\tilde{N}_{\tilde{x}_0^+} \subset \tilde{N}^{+}_{t_0}= \partial \{U> t_0\} = \partial \{U \leq t_0\}  $. Otherwise, if $y \in \tilde{N}_{\tilde{x}_0^+} \cap \mathrm{int}\{U \leq t_0\} $, then there exist $y_j \in \tilde{N}_{t_j}= \{U= t_j\} $ so that $y_j \to y$, which shows that $y_j \in \{U>t_0\} \cap \mathrm{int} \{U \leq t_0\} $, a contradiction. In the case that the target point $\tilde{x}_0 \in \tilde{N}_{t_0}$, we can make an analogous argument by choosing $t_j<t_0$ to get the conclusion when $t_0>0$.

	3. Let $\tilde{\mathcal{K}}_{t_0}$ be the interior of a jump region $\{U= t_0\} $. For each $\tilde{x}_0 \in \tilde{\mathcal{K}}_{t_0}$, we can apply a similar argument to construct a sequence of $\tilde{N}^{i}_{t_i} $ so that $\tilde{x}_0 \in \tilde{N}^{i}_{t_i}$ and $\tilde{N}^{i}_{t_i}$ converges to a locally uniformly $C^{1,\alpha }$-bounded hypersurface $\tilde{N}_{\tilde{x}_0}$ in a small neighborhood of $\tilde{x}_0$. By successively taking subsequences, we can construct a complete locally uniformly $C^{1,\alpha }$-bounded hypersurface that we will henceforth denote by $\tilde{N}_{\tilde{x}_0}$, and $\tilde{N}_{\tilde{x}_0} \subset \{U = t_0\} $.
\end{proof}

\begin{lem}[Lemma \ref{limit-nu-construction}]
	In Case (1) or (2) of Lemma \ref{lem-limit-surface-C1alpha}, let $\tilde{\nu }$ be the unit normal $C^{\alpha }$-vector field to $\tilde{N}_{t_0}$ or $\tilde{N}_{t_0}^{+}$, then $\tilde{\nu }_i \to \tilde{\nu }$ locally uniformly for the whole sequence. In Case (3), there exists a H\"older continuous unit vector field $\tilde{\nu }$ on $\tilde{\mathcal{K}}_{t_0}$ such that for a subsequence $i_j$, $\tilde{\nu }_{i_j} \to \tilde{\nu }$ locally uniformly, and $\tilde{\nu }$ is translation invariant and normal to $\tilde{N}_{\tilde{x}_0}$ for any $\tilde{x}_0 \in \tilde{\mathcal{K}}_{t_0}$.
\end{lem}
\begin{proof}
The convergence of $\tilde{\nu }_i \to \tilde{\nu }$ in Case (1) or (2) is clear from the proof of above Lemma. We only consider the existence of $\tilde{\nu }$ and convergence $\tilde{\nu }_i \to \tilde{\nu }$ up to a subsequence in Case (3). For any $\tilde{x}_0= (x_0, z_0) \in \tilde{\mathcal{K}}_{t_0}$ and let $\tilde{N}_{\tilde{x}_0}$ be the hypersurface through $\tilde{x}_0$ which is a limit of a subsequence of $\tilde{N}^{i}_{t_i}=\{U_i = t_i\} $. Note that $\tilde{N}_{\tilde{x}_0}$ is not a-priori unique. However, for any vertical translate $\tilde{x}_{\beta }= (x_0, z_0+ \beta )$, we have $\tilde{x}_{\beta } \in \tilde{N}^{i}_{t_i - \beta \epsilon _i}$, and $ \tilde{N}^{i}_{t_i - \beta \epsilon _i} \to \tilde{N}_{\tilde{x}_0} + \beta e_{z} $, i.e. we have convergence for all vertical translates w.r.t. the same subsequence. Thus, it suffices to construct a unit vector field $\tilde{\nu } \in T \tilde{\mathcal{K}}_{t_0}$ along $\mathcal{K}_{t_0}:=\mathrm{int} \{u=t_0\} = \tilde{\mathcal{K}}_{t_0} \cap (M \times \{0\} )$ that is normal to $\tilde{N}_{x_0}$ for $x_0 \in \mathcal{K}_{t_0}$, which is then trivially extended in the $z$-direction.

Let $\{x_k\} $ be a dense countable subset of $\mathcal{K}_{t_0}$. By a diagonal sequence argument, we can choose the hypersurfaces $\tilde{N}_{x_k}$ and a subsequence $i_j$ so that $\tilde{N}^{i_j}_{t ^{k}_{i_j}} \to \tilde{N}_{x_k}$ locally uniformly in $C^{1,\alpha }$ for all $x_k$. In particular, any two hypersurfaces $\tilde{N}_{x_k}$ can only touch each other tangentially, which yields a well-defined unit normal vector field on $\mathcal{K}_{t_0}$. In view of the uniform $C^{1,\alpha }$-estimates, there exists a uniform $r_0>0$ so that for any $x_k$, all $\tilde{N}^{i_j}_{t ^{m}_{i_j}} \cap B_{2 r_0}(x_k)$ can be written as normal graph over $T_{x_k} \tilde{N}_{x_k}$ with uniformly $C^{1,\alpha }$-bound, and $|\tilde{\nu }_{i_j}(x_k) - \tilde{\nu }_{i_l}(x_m)| \leq C\cdot |x_k-x_m|^{\alpha }$ uniformly for all large $j, l$ and any $x_k, x_m$. Thus, $\tilde{\nu }_{i_j}(x_k) \to \tilde{\nu }(x_k)$ locally uniformly and $\tilde{\nu }$ extends to a H\"older continuous unit vector field on $\mathcal{K}_{t_0}$. For any $x \in \mathcal{K}_{t_0}\setminus \{x_k\} $, one can take a further subsequece of $i_j$ to construct a hypersurface $\tilde{N}_{x}$ so that $\tilde{\nu }(x) $ is normal to $\tilde{N}_{x}$.
\end{proof}

In the following, for simplicity, we still denote the subsequence $i_j$ by $i$. By the standard compactness theorem of BV functions, we have $\tilde{\nabla } U_i \to \tilde{\nabla }U$ weakly as measures. As a result of above lemmas, by taking $\tilde{\nu }$ as test functions on $\tilde{\mathcal{K}}_{t_0}$, we have
\begin{lem}\label{appen-lem-L1-jump}
	\begin{align}
	|\tilde{\nabla }U_i| \to 0 \text{ in } L^{1}_{loc}(\tilde{\mathcal{K}}_{t_0}).
\end{align}
\end{lem}
\begin{proof}
	For any geodesic ball $B_{r}(\tilde{x}) \subset \subset  \tilde{\mathcal{K}}_{t_0}$, let $\varphi \in C^{\infty}_{0}(B_r(\tilde{x}) )$ be a nonnegative cutoff function. Since $\tilde{\nu }$ is H\"older continuous, we have $\int \varphi \langle \tilde{\nabla }U_i, \tilde{\nu } \rangle \to \int \varphi \langle \tilde{\nabla }U, \tilde{\nu } \rangle = 0$. Since $|\tilde{\nabla }U_i|$ are uniformly bounded and $\tilde{\nu }_i \to \tilde{\nu }$ uniformly, we have
	\begin{align*}
		\int \varphi |\tilde{\nabla }U_i| = \int \varphi \langle \tilde{\nabla }U_i, \tilde{\nu }_i \rangle = \int \varphi \langle \tilde{\nabla }U_i, \tilde{\nu } \rangle + \int \varphi \langle \tilde{\nabla }U_i, \tilde{\nu }_i - \tilde{\nu } \rangle\to 0.
	\end{align*} 
\end{proof} 

We recall the following compactness result of \cite[Theorem 5.3]{HuiskenWolff22}.
\begin{lem}\label{lem-compactness}
	Let $\tilde{\Omega } \subset M \times \mathbb{R}$, and let $\tilde{E}_i \subset \tilde{\Omega }$ be a sequence of sets with $C^{1,\alpha }_{loc}$ boundary such that $\partial \tilde{E}_i \to \partial \tilde{E}$ locally in $C^{1,\alpha }$, with outward unit normal $\tilde{\nu} _i \in C^{\alpha }_{loc}(T \tilde{\Omega })$ to $\partial \tilde{E}_i$ satisfying $\tilde{\nu }_i \to \tilde{\nu }$ locally uniformly. Let $U_i \in C^{0,1}_{loc}(\tilde{\Omega })$ satisfy $U_i \to U$ locally uniformly,  $\sup_{K} |\tilde{\nabla} U_i| \leq C(K)$ for each $K \subset \subset \tilde{\Omega} $ and for all large $i$, and $|\tilde{\nabla }U_i| \to |\tilde{\nabla }U|$ in $L^{1}_{loc}(\tilde{\Omega })$. Let $\sigma ^{i} \to \sigma $ uniformly for a Lipschitz symmetric $2$-tensor $\sigma $ and $\|\sigma ^{i}\|_{C^1} \leq C$.
	Then, if $\tilde{E}_i$ minimizes $J_{U_i, \tilde{\nu }_i}^{\sigma ^{i}}$ in $\tilde{\Omega }$, $\tilde{E}$ minimizes $J_{U, \tilde{\nu }} ^{\sigma }$ in $\tilde{\Omega }$.
\end{lem}

\begin{lem}[Lemma \ref{lem-foliation-jump}]
	The interior $\tilde{\mathcal{K}}_{t_0}$ of the jump region is foliated by $C^{2,\alpha }$-hypersurfaces, where each such hypersurface is either a vertical cylinder or a graph over an open subset of $\{u=t_0\} $. Furthermore, each hypersurface bounds a Caccioppoli set that minimizes $J_{U, \tilde{\nu }}$ in $\tilde{\mathcal{K}}_{t_0}$, where $\tilde{\nu }$ denotes the $C^{1,\alpha }$-unit normal vector field to the hypersurface foliation. 
\end{lem}
\begin{proof}
	The proof is similar to \cite[Proposition 8]{Moore12} and \cite[Theorem 5.9]{HuiskenWolff22}. 
	For any $\tilde{x}_0 \in \tilde{\mathcal{K}}_{t_0}$, by the Compactness Lemma \ref{lem-compactness}, the construction of $\tilde{N}_{\tilde{x}_0}$, and the fact that $|\tilde{\nabla }U_i| \to 0$ in $L^{1}_{loc}(\tilde{\mathcal{K}}_{t_0})$, we know that $\tilde{N}_{\tilde{x}_0}$ bounds a Caccioppoli set $\tilde{E}_{\tilde{x}_0}$ such that $\partial \tilde{E}_{\tilde{x}_0}= \tilde{N}_{\tilde{x}_0}$, $\tilde{\nu }$ is normal to $\partial \tilde{E}_{\tilde{x}_0}$, and $\tilde{E}_{\tilde{x}_0}$ minimizes $J_{U, \tilde{\nu }} ^{\sigma }$ in $\tilde{\mathcal{K}}_{t_0}$. 

	From the construction of $\tilde{N}_{\tilde{x}_0}$, there exists a small $r>0$ so that $\tilde{N}_{\tilde{x}_0} \cap B_r(\tilde{x}_0) = \mathrm{graph}(w)$ for a function $w \in C^{1,\alpha }( T_{\tilde{x}_0}(\tilde{N}_{\tilde{x}_0}) \cap B_r(\tilde{x}_0) ) $. We denote the ball $T_{\tilde{x}_0}(\tilde{N}_{\tilde{x}_0}) \cap B_r(\tilde{x}_0) $ by $B^{n}_r(0)$. Then $\tilde{E}_{\tilde{x}_0} \cap B_r(\tilde{x}_0) $ is the subgraph $W= \{(x, t) \in B^{n}_r(0)  \times \mathbb{R}: t< w(x)\} $. Since $W$ minimizes $J_{U, \tilde{\nu }} ^{\sigma }$ in $\tilde{\mathcal{K}}_{t_0}$, $w$ minimizes the functional
\begin{align*}
	\int_{B_r^{n}(0)} \sqrt{1+|\tilde{\nabla }w|^2} d x - \int_{B_r^n(0)} \int_{0}^{w(x)} \sigma _{ij}(x, s) \tilde{\nu }^{i} \tilde{\nu }^{j} d s d x,
\end{align*} 
whose Euler-Lagrange equation is 
\begin{align}\label{div-w-eq}
	\mathrm{div}_{\tilde{g}}\left( \frac{\tilde{\nabla }w}{\sqrt{1+|\tilde{\nabla }w|^2} } \right) - \sigma _{ij} \tilde{\nu }^{i} \tilde{\nu }^{j} =0,
\end{align}
with $\tilde{\nu } = \frac{(\tilde{\nabla }w, -1)}{\sqrt{1+|\tilde{\nabla }w|^2} }$. Since $w \in C^{1,\alpha }$, this is in fact a uniformly elliptic equation with $C^{\alpha }$-coefficients. Note that $\sigma _{ij} \tilde{\nu }^{i} \tilde{\nu }^{j}$ has a uniform $C^{\alpha }$-bound depending on $\|\sigma \|_{C^{\alpha }}$, which is controlled by $\|\sigma \|_{C^{0,1}}$. Then Schauder estimate implies that $w \in C^{2,\alpha }$, and  $\tilde{\nu }$  has locally uniformly $C^{1, \alpha }$-bound.

Finally, applying the strong maximum principle to (\ref{div-w-eq}), we can rule out the possibility that any two such hypersurfaces touching each other. 

Since each $\tilde{N}_{\tilde{x}_0}$ is the limit of the graphs $\tilde{N}^{i}_{t_i}= \mathrm{graph}\left( \frac{u^{i} - t_i}{\epsilon _i} \right) $ over $\Omega _i$, it's clear that $\tilde{N}_{\tilde{x}_0}$ is either a vertical cylinder or a graph over an open subset of $\tilde{\mathcal{K}}_{t_0} \cap M$.
\end{proof}

\subsection{Outward optimizing property}\label{appen-optimizing}
In this section, we provide details of proof of several outward optimizing properties. 

\begin{lem}[Lemma \ref{lem-optimizing-property}]
	Suppose that $u$ is a weak solution of the $(|h|g)$-IMCF with initial condition $E_0$. Then
	\begin{itemize}
		\item [(i)] $E_t$ is outward optimizing in $M$ for $t>0$;
		\item [(ii)] $E_t ^{+}$ is outward optimizing in $M$ for $t \geq 0$;
		\item [(iii)] $|\partial E_t ^{+}| = |\partial E_t| + \int_{E_t ^{+}\setminus E_t}|h|$, for all $t>0$. This extends to $t=0$ precisely if $E_0$ is outward optimizing.
	\end{itemize}
\end{lem}
\begin{proof}
	The proof is a slight modification of \cite[Theorem 1.4]{HI01}. (i) By \ref{dag}, for each $t>0$, for any $F$ such that $E_t \subset F$ and $F \setminus E_t \subset K \subset \subset M \setminus E_t$, it holds that
	\begin{align*}
		|\partial ^* E_t \cap K| \leq |\partial ^* F \cap K| - \int_{F \setminus E_t} \left( |\nabla u| + |h| \right) \leq |\partial ^* F \cap K| - \int_{F \setminus E_t} |h|,
	\end{align*} 
	which proves that $E_t$ is an outward optimizing hull for $t>0$.

	(ii) By (\ref{leq-t-weak}), for each $t \geq 0$, for any $F$ such that  $F \Delta  E_t^+ \subset K \subset \subset M\setminus E_t$, it holds that
\begin{align*}
		|\partial ^* E_t^+ \cap K| \leq |\partial ^* F \cap K| - \int_{F \setminus E_t^+} \left( |\nabla u| + |h| \right) \leq |\partial ^* F \cap K| - \int_{F \setminus E_t^+} |h|,
	\end{align*} 
	which proves that $E_t^+$ is an outward optimizing hull for $t \geq 0$. 

	(iii) For $t>0$, we use $E_t^+$ as a competitor of $E_t$ to obtain
	\begin{align*}
		|\partial E_t \cap K| \leq |\partial E_t^+ \cap K| - \int_{E_t^+ \setminus E_t} |h|,
	\end{align*} 
	which also holds for $t=0$ if $E_0$ is outward optimizing. Using
	\begin{align*}
		|\partial E_t^+ \cap K| - \int_{E_t^+ \setminus E_t} \left( |\nabla u| + |h|\right)  \leq |\partial E_t \cap K|,
	\end{align*} 
	and the fact that $|\nabla u| = 0$ a.e. on $E_t^+ \setminus E_t$, we have the equality.
\end{proof}

\begin{lem}[Lemma \ref{E_t+-vs-E_t'}]
	Suppose that $u$ is a weak solution of the $(|h|g)$-IMCF with initial condition $E_0$ and $M$ has no compact components. Let $\Omega \subset M \setminus \bar{E}_0$ be a domain and $t \geq 0$ such that $E_t ^{+} \subset \Omega $, and assume $E_t$ admits a precompact outward optimizing hull $(E_{t})_{\Omega }'$ in $\Omega $. Then up to a measure zero set,
 $E_t ^+ \subset (E_t)_{\Omega }'$.
\end{lem}
\begin{proof}
	The proof is a slight modification of \cite[Proposition 8.3]{HuiskenWolff22}. For any locally integrable function $f$ and precompact Caccioppoli sets $E_1, E_2$, we have the general inequality
	\begin{align*}
		|\partial ^*(E_1\cup E_2)| + \int_{E_1 \cup E_2} f + |\partial ^* (E_1\cap E_2)| & + \int_{E_1\cap E_2} f \\
		&\leq |\partial ^* E_1| + \int_{E_1} f + |\partial ^* E_2| + \int_{E_2} f.
	\end{align*}
	By (\ref{leq-t-weak}), for each $t \geq 0$, $E_t ^+$ minimizes $J^{|h|}_{u}$, so 
	\begin{align*}
		|\partial ^* E_t ^+| - \int_{E_t^+} \left( |\nabla u| + |h| \right) \leq |\partial ^*(E_t^+ \cap (E_t)_{\Omega }')| - \int_{E_t^+ \cap (E_t)_{\Omega }'} \left( |\nabla u| + |h| \right) ,
	\end{align*} 
	where we used the fact that $E_t^+ \setminus (E_t)_{\Omega }' \subset E_t^+$ is precompact. Since $|\nabla u| = 0$ a.e. on $E_t^+ \setminus E_t$, we have
	\begin{align}\label{appen-E_t+-minimizing}
		|\partial ^* E_t^+| - \int_{E_t^+}|h| \leq |\partial ^*(E_t^+ \cap (E_t)_{\Omega }')| - \int_{E_t^+ \cap (E_t)_{\Omega }'} |h|.
	\end{align} 
	Choosing $f= - |h|, E_1 = E_t^+, E_2= (E_t)_{\Omega }'$ in above general inequality and using (\ref{appen-E_t+-minimizing}) and the strictly outward optimizing property of $(E_t)_{\Omega }'$, we obtain
	\begin{align*}
		|\partial ^* ( E_t^+ \cup  (E_t)_{\Omega }')| - \int_{ E_t^+ \cup (E_t)_{\Omega }'}|h| &\leq |\partial ^* (E_t)_{\Omega }'| - \int_{(E_t)_{\Omega }'}|h|\\
		&\leq |\partial ^* ( E_t^+ \cup  (E_t)_{\Omega }')| - \int_{ E_t^+ \cup (E_t)_{\Omega }'}|h| .
	\end{align*} 
	Hence, the equality holds, which implies that $(E_t)_{\Omega }' = E_t^+ \cup  (E_t)_{\Omega }'$ up to a set of measure zero. Therefore, $E_t^+ \subset (E_t)_{\Omega }'$ up to a set of measure zero. 
	
\end{proof}

\subsection{Upper bound by smooth solution}
In this section, we prove
\begin{lem}\label{appen-upper-smooth}
	Let $E_0$ be a precompact open set in $M$ such that $\partial E_0$ is $C^2$ with $H- |h|>0$ and $E_0 = E_0'$. Then any weak solution $(E_t)_{0<t<\infty}$ of the $(|h|g)$-IMCF with initial condition $E_0$ is bounded from above by the $C^{2,\alpha }$-classical solution for a short time, provided that $E_t$ remains precompact for a short time.
\end{lem}

\begin{proof}
Let $\bar{u}$ be the $C^{2,\alpha }$-classical solution of the equation 
	\begin{align*}
		\mathrm{div}\left( \frac{\nabla \bar{u}}{|\nabla \bar{u}|} \right) = |\nabla \bar{u}| + |h| \text{ on }  \{0 \leq \bar{u}< t_0\},
	\end{align*} 
	and let $u$ be a weak solution in the sense that $u$ minimizes $J^{|h|}_{u}$ among competitors on $M \setminus E_0$. Assume that $|\nabla \bar{u}| \geq c_0>0$. 

Set $W = \{0 \leq \bar{u}< t_0\} $. Since $E_0 = E_0'$, by Lemma \ref{E_t+-vs-E_t'}, $E_0= E_0 ^{+}$. Since $E_t$ is precompact, $E_t$ converges to $E_0^+$ in Hausdorff distance as $t \searrow 0$, which shows that $E_t \subset \subset W$ for all small $0\leq t \leq t_1$.

	We modify the proof of \cite[Theorem 2.2]{HI01} to show that $u \leq \bar{u}$ in a small neighborhood of $E_0$. Observe that for $u ^{t} = \min(u, t)$, we have $\{u ^{t} < s\} $ minimizes $J^{|h|}_{u ^{t}}$ for all $s$, thus $u ^{t}$ minimizes $J^{|h|}_{u ^{t}}$. For a fixed small $t_1< t_0$, let $\Omega := \{ \bar{u}< t_1\}  \setminus \bar{E}_0$. For each $\delta >0$, we have $\{u ^{t_1} > \bar{u} + \delta \} \subset \subset \Omega $. It's enough to show that $u ^{t_1} \leq \bar{u}+ \delta $ on $\Omega $. 

	Set $u_2 = u ^{t_1}$ and $u_1 = \bar{u}+ \delta $. 
	For any $\epsilon >0$, consider $u_1 ^{\epsilon } := (1-\epsilon )^{-1}u_1$. Then $\{u_2 > u_1^{\epsilon }\} \subset \subset \Omega $. For any Lipschitz function $w \geq u_1 ^{\epsilon }$ with $\{w \neq u_1 ^{\epsilon }\} \subset \subset \Omega $, we have $J^{|h|}_{u_1}(u_1) \leq J^{|h|}_{u_1}( (1-\epsilon ) w)$, i.e.
	\begin{align*}
	\int |\nabla u_1| + u_1\left( |\nabla u_1| +|h| \right) \leq (1-\epsilon )\int |\nabla w| + w \left( |\nabla u_1| + |h| \right) ,
\end{align*} 
which is equivalent to
\begin{align}\label{u1eps-w}
		\int |\nabla u_1 ^{\epsilon }| + u_1 ^{\epsilon }\left( |\nabla u_1 ^{\epsilon }| + |h| \right) + \epsilon \int \left( w- u_1 ^{\epsilon } \right) |\nabla u_1 ^{\epsilon }| \leq \int |\nabla w| + w\left( |\nabla u_1 ^{\epsilon }| + |h| \right) .
	\end{align} 
	Replace $w$ by $u_1 ^{\epsilon } + (u_2- u_1 ^{\epsilon })_{+}$ to obtain
	\begin{align*}
\int_{u_2 > u_1 ^{\epsilon }} |\nabla u_1 ^{\epsilon }| + u_1 ^{\epsilon }\left( |\nabla u_1 ^{\epsilon }| + |h| \right) + \epsilon \int_{u_2> u_1 ^{\epsilon }} \left( u_2- u_1 ^{\epsilon } \right) |\nabla u_1 ^{\epsilon }| \\
\leq \int_{u_2 > u_1 ^{\epsilon }} |\nabla u_2| + u_2\left( |\nabla u_1 ^{\epsilon }| + |h| \right) .
	\end{align*} 
Using $u_2- (u_2- u_1 ^{\epsilon })_{+}$ as a competitor of $u_2$, we obtain
\begin{align*}
	\int_{u_2> u_1 ^{\epsilon }} |\nabla u_2| + u_2\left( |\nabla u_2| + |h| \right) \leq \int_{u_2> u_1 ^{\epsilon }} |\nabla u_1 ^{\epsilon } | + u_1 ^{\epsilon }\left( |\nabla u_2| + |h| \right) .
\end{align*} 
Adding these, we get
\begin{align}\label{u2-u1-ineq1}
	\int_{u_2> u_1 ^{\epsilon }} (u_2 - u_1 ^{\epsilon }) \left( |\nabla u_2| - |\nabla u_1 ^{\epsilon }| \right) + \epsilon \int_{u_2> u_1 ^{\epsilon }} \left( u_2- u_1 ^{\epsilon } \right) |\nabla u_1 ^{\epsilon }| \leq 0.
\end{align} 

Using $u_1 ^{\epsilon } + (u_2 - s - u_1 ^{\epsilon })_{+}, s \geq 0$ as a competitor in (\ref{u1eps-w}), we have
\begin{align*}
	\int_{u_2-s> u_1 ^{\epsilon }} |\nabla u_1 ^{\epsilon }| + u_1 ^{\epsilon }\left( |\nabla u_1 ^{\epsilon }| + |h| \right) \leq \int_{u_2-s> u_1 ^{\epsilon }} |\nabla u_2| + (u_2-s) \left( |\nabla u_1 ^{\epsilon }| + |h| \right), 
\end{align*} 
which implies that
\begin{align*}
	\int_{u_2-s> u_1 ^{\epsilon }} \left( 1+ u_1 ^{\epsilon } - u_2 + s \right) |\nabla u_1 ^{\epsilon }| \leq  \int_{u_2-s> u_1 ^{\epsilon }} |\nabla u_2| + (u_2 - s - u_1 ^{\epsilon } ) |h|.
\end{align*} 
Integrating over $s$ and switching the order of integration, we have
\begin{align*}
	\int_{M} & |\nabla u_1 ^{\epsilon }| \int_{0}^{\infty} \chi _{\{ u_2- u_1 ^{\epsilon }> s\} }  (1+ u_1 ^{\epsilon } - u_2 + s) ds \\
	&\leq \int_{M} |\nabla u_2| \int_{0}^{\infty} \chi _{\{ u_2- u_1 ^{\epsilon }> s\} } + \int_{M} |h| \int_{0}^{\infty} \chi _{\{u_2- u_1 ^{\epsilon }> s\} }(u_2-u_1^{\epsilon }-s) ds,
\end{align*} 
which yields
\begin{align*}
	\int_{u_2> u_1 ^{\epsilon }} \left( (u_2- u_1 ^{\epsilon }) - \frac{1}{2} (u_2-u_1 ^{\epsilon })^2 \right)  |\nabla u_1 ^{\epsilon }| \leq 
	\int_{u_2> u_1 ^{\epsilon }} (u_2- u_1 ^{\epsilon }) |\nabla u_2|+  \frac{1}{2}|h| (u_2- u_1 ^{\epsilon })^2,
\end{align*} 
i.e.
\begin{align*}
	-\frac{1}{2} \int_{u_2> u_1 ^{\epsilon }} (u_2- u_1 ^{\epsilon })^2 \left( |\nabla u_1 ^{\epsilon }| + |h| \right)  \leq \int_{u_2 > u_1 ^{\epsilon }} (u_2 - u_1 ^{\epsilon }) \left( |\nabla u_2| - |\nabla u_1 ^{\epsilon }| \right) .
\end{align*} 
Inserting this into (\ref{u2-u1-ineq1}) gives
\begin{align}\label{u2-u1-final}
	-\frac{1}{2} \int_{u_2> u_1 ^{\epsilon }} (u_2- u_1 ^{\epsilon })^2 \left( |\nabla u_1 ^{\epsilon }| + |h| \right) + \epsilon \int_{u_2> u_1 ^{\epsilon }} \left( u_2- u_1 ^{\epsilon } \right) |\nabla u_1 ^{\epsilon }| \leq 0.
\end{align} 
Since $|\nabla u_1 ^{\epsilon }| \geq c_0>0$, if $u_2< u_1 ^{\epsilon } + \epsilon ^2$ for a small enough $\epsilon $, we have $\frac{1}{2} (u_2- u_1 ^{\epsilon }) |h| \leq C \epsilon ^2 \leq \frac{1}{4} c_0 \epsilon \leq \frac{1}{4} \epsilon |\nabla u_1 ^{\epsilon }| $, which yields a contradiction. We conclude that $u_2 \leq u_1 + \epsilon ^2$ implies $u_2 \leq u_1$. For general $u_2$, by subtracting a constant we can arrange that $0< \sup (u_2- u_1) \leq \epsilon ^2$, contradicting what has just been proved. This proves that $u_2 \leq u_1$ on $\Omega $ for any $\delta >0$. Taking $\delta \to 0$, we have $u \leq \bar{u}$ on $\Omega $.

\end{proof}

\subsection{Asymptotic behavior of weak solutions}\label{appen-asymptotic}
In this subsection, we provide details of the proof of several lemmas concerning the asymptotic behavior of weak solutions.

\begin{lem}[Lemma \ref{lem-blowdown}]
	Suppose that on $\Omega $, $(M, g, h)$ satisfies the asymptotically flat condition
	\begin{align*}
		|g - \delta | = o(1),\ |\bar{\nabla }g| = o (|x| ^{-1}), \ |h| = o(|x| ^{-1}),\ |\bar{\nabla }h| = o(|x| ^{-2}),
	\end{align*} 
	as $|x| \to \infty$. Let $N_t = \partial \{ u< t\} $ be a weak solution of the $(|h| g)$-IMCF for all sufficiently large $t$ so that $\{u=t\} $ is compact. Then for some constants $c_{\lambda } \to \infty$,
	\begin{align*}
		u ^{\lambda } - c_{\lambda } \to 2 \log |y|,
	\end{align*} 
	locally uniformly in $\mathbb{R}^{3} \setminus \{ 0\} $ as $\lambda \to 0$, the standard expanding sphere solution of the standard IMCF on $\mathbb{R}^{3} \setminus \{0\} $.
\end{lem}

\begin{proof}
	1.  Fix a large $t_0$ so that $\{ u = t\}  \subset \Omega $ is compact for all $t \geq t_0$. By the asymptotically flat condition and Remark \ref{rmk-choice-eta}, there is $R_0>0$ such that 
	\begin{align*}
		\eta (x) \geq c|x|,\ \ d(x, \partial E_{t_0}) \geq c |x|,
	\end{align*} 
	for $|x| \geq R_0$. By Lemma \ref{smooth-sol-*-epsilon} with $r = \min(\eta (x), d(x, \partial E_{t_0}) )$, we obtain
	\begin{align}\label{gradient-decay}
		|\nabla u| (x) \leq C |x| ^{-1}, \ \ \forall |x| \geq R_0.
	\end{align}

	For any hypersurface $N$, define the eccentricity $\theta (N) := R(N) / r(N)$, where $[r(N), R(N)]$ is the smallest interval such that $N$ is contained in the annulus $\bar{D}_{R} \setminus D_r$. Under the asymptotically flat condition, there exists $A>0$ and $t_1$ so that $(D_{e ^{A t}})_{t_1 \leq t< \infty}$ provides a subsolution of \ref{epsilon-eq}, see also the proof of Lemma \ref{smooth-sol-*-epsilon}. For any $t \geq t_1$, by constructing such comparison starting from $D_{R(N_t)}$, i.e. $e ^{A t} = R(N_t)$, we obtain $u(x) \geq A ^{-1} \log |x|$ on $\Omega \setminus D_{R(N_t)}$. Thus for any $\tau \geq 0$,
	\begin{align}\label{RNt-growth}
		R(N_{t+ \tau }) \leq e ^{A(t+ \tau )} = e ^{A \tau } R(N_t).
	\end{align}
	Suppose that $r = r(N_t) \geq R_0$. By (\ref{gradient-decay}), $u > t - C_2$ everywhere on $\partial D_r$. Therefore, $N_{t- C_2}$ does not meet $\partial D_r$. By (i) of Lemma \ref{lem-connected}, $N_{t- C_2}$ can not have any components outside of $D_r$, so $R(N_{t- C_2}) \leq r$. Combining with (\ref{RNt-growth}), 
	\begin{align}\label{rNt-growth}
		R(N_t) \leq e ^{A C_2} R(N_{t- C_2}) \leq e ^{A C_2} r(N_t),\quad t \geq t_2,
	\end{align}
	for some $t_2$.

	2. Now consider any sequence $\lambda _i \to 0$. The estimates (\ref{gradient-decay}), (\ref{RNt-growth}) and (\ref{rNt-growth}) are scale-invariant, so they are valid for $N^{\lambda _i}_{t} = \partial \{ u ^{\lambda _i} < t\} $, on the complement of a subset shrinking to $\{0\} $ as $i \to \infty$. Moreover, on the complement of a subset shrinking to $\{0\} $ as $i \to \infty$, we have 
	\begin{align*}
		|h ^{\lambda _i}(y)| \leq o(1) \cdot  |y| ^{-1},
	\end{align*} 
	which is locally uniformly bounded and converges to $0$ locally uniformly.
	By Arzel\`a-Ascoli, there exists a subsequence, which is denoted by $\lambda _j$, numbers $c_j = \max_{\mathbb{S}_1(0)}|u ^{\lambda _j}| \to \infty$, and a locally Lipschitz function $v \in C ^{0,1}_{loc}(\mathbb{R}^3 \setminus \{ 0\} )$, such that
	\begin{align*}
		u ^{\lambda _j} - c_j \to v \quad \text{ locally uniformly in } \mathbb{R}^{3} \setminus \{0\},
	\end{align*}
	with local $C^1$-convergence of the level set, except at most countable jump times. Notice that $|h ^{\lambda _j}| \to 0$ locally uniformly. In view of the Compactness Lemma \ref{lemma-compactness}, with focus only on the $J^{|h ^{\lambda }_i|}_{u_i}$-minimizing property, we know that $v$ is a weak $J^{0}$-solution in the sense of (\ref{weak-J}) in $(\mathbb{R}^{3}\setminus \{0\} , \delta )$, i.e. $v$ is a weak solution of the standard IMCF as studied in \cite{HI01}.

	Define $P_t := \partial \{ v< t\} $ for $-\infty< t< \infty$. By (\ref{rNt-growth}) and approximating $N_{t_0}$ by $t_i \nearrow t_0$ for jump times $t_0$, each nonempty $P_t$ is a compact subset of $\mathbb{R}^{3}\setminus \{0\} $ with
	\begin{align*}
		\theta (P_t) \leq e ^{A C_2}.
	\end{align*} 
	Now (\ref{RNt-growth}) implies that $v$ is not a constant, so some level set, say $P_{t_3}$, is nonempty. Then (\ref{RNt-growth}) and (\ref{rNt-growth}) imply that $P_t$ is nonempty and compact in $\mathbb{R}^{n}\setminus \{ 0\} $ for all $-\infty< t< \infty$. By \cite[Proposition 7.2]{HI01}, we know $v$ is the unique expanding sphere solution of the standard IMCF. 

\end{proof}

\end{appendix}

\bibliographystyle{alpha}
\bibliography{./math}

\end{document}